\documentclass[12pt]{article}

\usepackage{amsmath}
\usepackage{amsthm}

\usepackage{amssymb}

\usepackage{mathptmx}

\usepackage{caption, subcaption}%

\usepackage{xparse} 
\usepackage{stmaryrd} 

\usepackage{mathrsfs}

\usepackage{multirow}

\usepackage{geometry}
\usepackage{bm}

\usepackage{enumitem}

\usepackage{color}
\usepackage{graphicx}
\usepackage{epstopdf}

\makeatletter
\def\widebreve{\mathpalette\wide@breve}
\def\wide@breve#1#2{\sbox\z@{$#1#2$}%
	\mathop{\vbox{\m@th\ialign{##\crcr
				\kern0.08em\brevefill#1{0.8\wd\z@}\crcr\noalign{\nointerlineskip}%
				$\hss#1#2\hss$\crcr}}}\limits}
\def\brevefill#1#2{$\m@th\sbox\tw@{$#1($}%
	\hss\resizebox{#2}{\wd\tw@}{\rotatebox[origin=c]{90}{\upshape(}}\hss$}
\makeatletter

\usepackage[colorlinks=true]{hyperref}

\allowdisplaybreaks[3]
\def\vec#1{{\bf #1}}

\DeclareMathOperator*{\essinf}{ess\,inf}

\renewcommand{\qed}{\hfill $\blacksquare$} 

\newtheorem{theorem}{Theorem}[section]
\newtheorem{proposition}{Proposition}[section]
\newtheorem{lemma}{Lemma}[section]

\newtheorem{remark}{Remark}[section]
\newtheorem{corollary}{Corollary}[section]

\providecommand{\keywords}[1]
{
	\small	
	\textbf{\textit{Keywords:}} #1
}

\title{{\large \bf Minimum Principle on Specific Entropy 
		and   
		 High-Order Accurate 
		 Invariant Region Preserving
		Numerical  
		Methods  
		for  
		Relativistic Hydrodynamics}}

\author{
	Kailiang Wu\thanks{Department of Mathematics, Southern University of Science and Technology, Shenzhen 518055, P.R.~China.  ({\tt wukl@sustech.edu.cn}). 
}   
}

\date{\today}

\begin{document}
	
	\maketitle

	\vspace{-7mm}

	\begin{abstract}
	This paper first explores Tadmor's minimum entropy principle  
for the special relativistic hydrodynamics (RHD) equations and   
incorporates this 
principle into
the design of 
robust 
high-order discontinuous Galerkin (DG) 
and finite volume schemes  
for RHD on general meshes. 
The proposed schemes are rigorously proven to preserve 
numerical solutions in a global {\em invariant region}  
constituted by all the known intrinsic constraints: 
minimum entropy principle,   
the subluminal constraint on fluid velocity,  
the positivity of pressure, and the positivity of rest-mass density.  
Relativistic effects lead to some essential difficulties in the present study, 
which are not encountered in 
the non-relativistic case. 
Most notably, in the RHD case the specific entropy is a highly nonlinear {\em implicit} function of the conservative variables, 
and, moreover, there is also no explicit formula of the flux in terms of the conservative variables. 
In order to overcome the resulting challenges, we first propose a novel {\bf equivalent} form of the invariant region, by skillfully 
introducing two auxiliary variables. 
As a notable feature, 
all the constraints in the novel form are {\bf explicit} and {\bf linear} with respect to the conservative variables. 
This provides a highly effective approach to  
theoretically analyze the invariant-region-preserving (IRP) property of numerical schemes 
for RHD, {\em without any assumption on the IRP property of the exact Riemann solver}. 
Based on this, we prove the convexity of the invariant region and establish the 
generalized Lax--Friedrichs splitting properties via technical estimates, lying the foundation for our rigorous IRP analyses. 
It is rigorously shown that the first-order Lax--Friedrichs type scheme for the RHD equations satisfies a local minimum entropy principle and is IRP under a CFL condition. 
Provably IRP high-order accurate DG and finite volume methods are then developed for the RHD with the help of 
a simple scaling limiter, which is designed by following the bound-preserving type limiters in the literature. 
Several numerical examples demonstrate the effectiveness and robustness 
of the proposed schemes.
		
		\vspace{2mm}
		\noindent	
		%
		

	\end{abstract}

	\newpage

	\section{Introduction} 
	\label{sec:intro}
	
In the study of the fluid dynamics, when the fluid
flow moves close to the speed of light or when sufficiently strong gravitational fields are involved, 
the special or general relativistic effect has to be taken into account accordingly. 
Relativistic hydrodynamics (RHD) plays an important role in a wide range of applications, 
such as astrophysics and high energy physics, and has been applied to investigate   
astrophysical scenarios from stellar to galactic scales.

In the framework of special relativity, 
the motion of ideal 
relativistic fluid is governed by the conservation of mass density $D$, momentum ${\bf m}$, and energy $E$. 
In the laboratory frame, the $d$--dimensional special RHD equations can be written into the following nonlinear hyperbolic system 
\begin{equation}\label{eq:RHD}
	{\frac{\partial {\bf U}}{\partial t}} + \sum_{i=1}^d \frac{ {\bf F}_i ({\bf U}) }{ \partial x_i } = {\bf 0}.
\end{equation}
with the conservative vector $\vec U$ and the flux ${\bf F}_i$ defined by 
\begin{align}\label{eq:dDU}
	& {\bf U} = (D,~{\bm m}^\top,~E)^\top = ( \rho W,~\rho h W^2 {\bm v}^\top,~\rho h W^2 - p )^\top,
	\\ \label{eq:dDF}
	& {\bf F}_i = ( D v_i,~v_i {\bm m}^\top  + p {\bf e}_i^\top,~m_i )^\top = ( \rho W v_i,~\rho h W^2 v_i  {\bm v}^\top + p {\bf e}_i^\top,~ \rho h W^2v_i )^\top,
\end{align}
where and hereafter the geometrized unit system is employed so that the speed of light $c=1$ in vacuum. 
In \eqref{eq:dDU}--\eqref{eq:dDF}, $\rho$ denotes the rest-mass density, 
$p$ is the thermal pressure, the column vector 
${\bm v}=(v_1,\dots, v_d)^\top$ represents the velocity field of 
the fluid, $W=1/\sqrt{1- \|{\bm v}\|^2}$ denotes the Lorentz factor with $\|\cdot\|$ denoting the vector 2-norm, 
$h = 1 + e + \frac{p}{\rho}$ stands for the specific enthalpy with $e$ being the specific internal energy, and 
the row vector ${\bf e}_i$ denotes the $i$-th column of the identity matrix of size $d$. 
To close the system \eqref{eq:RHD}, an equation
of state (EOS) is needed. 
We focus on the ideal EOS:
\begin{equation}\label{eq:iEOS}
	p = (\Gamma - 1) \rho e. 
\end{equation}
Here the constant $\Gamma \in (1,2]$ denotes the adiabatic index, for which the restriction $\Gamma \le 2$ is required by the compressibility assumptions and the relativistic causality (see, e.g., \cite{WuTang2015}). 

As we can see from \eqref{eq:dDU}--\eqref{eq:dDF}, the conservative vector $\vec U$ and 
the flux ${\bf F}_i$ are explicitly expressed in terms of  
the primitive quantities ${\bf V} := ( \rho, {\bm v}^\top, p )^\top$. 
However, unlike the non-relativistic case, 
for RHD there are no explicit formulas for either
the flux ${\bf F}_i$ or the primitive vector ${\bf V}$ in terms of ${\bf U}$. 
Therefore, in order to update the flux ${\bf F}_i({\bf U})$ in the computation, one has to first 
perform the inverse 
transformation of \eqref{eq:dDU} from the conservative vector ${\bf U}$ to the primitive vector ${\bf V}$. 
Given a conservative vector ${\bf U}=(D, {\bm m}^\top, E)^\top$, we can compute the values of the corresponding primitive quantities $\{p({\bf U}), {\bm v} ({\bf U}), \rho ({\bf U}) \}$ as follows: first solve a nonlinear algebraic equation 
\begin{equation}\label{eq:Eqforp}
	\frac{ \| {\bm m} \|^2 }{E+p} 
	+ D \sqrt{  1 - \frac{ \| {\bm m} \|^2 }{ (E+p)^2 }  } + \frac{p}{\Gamma - 1} - E=0, \qquad p \in [0, + \infty),
\end{equation}
by utilizing certain root-finding algorithm to get the pressure $p({\bf U})$; then calculate the velocity and rest-mass density by 
\begin{equation}\label{eq:vUdU}
	{\bm v} ({\bf U}) = 
	\frac{ \bm m }{E + p ({\bf U})}, \qquad \rho( {\bf U} ) = D \sqrt{ 1 - \left \| {\bm v} ({\bf U}) \right \|^2 }.
\end{equation}
Both the physical significance and the
hyperbolicity of \eqref{eq:RHD} require that the following constraints 
\begin{equation}\label{eq:PHYSconstraints}
	\rho > 0,~~ p > 0,~~ \|{\bm v}\| < c=1, 
\end{equation}
always hold. 
In other words, the conservative vector ${\bf U}$ must stay in the admissible state set 
\begin{equation}\label{eq:RHD:definitionG}
	{\mathcal G} := \left\{ \vec U=(D,{\bm m}^\top,E)^{\top} \in \mathbb R^{d+2}:~\rho(\vec U)>0,~p(\vec U)>0,~\|{\bm v}(\vec U)\| <1  \right\},
\end{equation}
where the functions $\rho(\vec U)$, $p(\vec U)$, and ${\bm v}(\vec U)$ are highly nonlinear and have no explicit formulas, as defined above. It was observed in \cite{mignone2005hllc} and rigorously proven in \cite[Lemma 2.1]{WuTang2015} that 
the set ${\mathcal G}$ is convex and is exactly {\em equivalent} to the following set 
\begin{equation}\label{eq:RHD:definitionG2}
	{\mathcal G}_1 := \left\{ \vec U=(D,{\bm m}^\top,E)^{\top}\in \mathbb R^{d+2}:~D>0,~E>\sqrt{D^2 + \|{\bm m}\|^2}  \right\}.
\end{equation}
Moreover, if ${\bf U} \in {\mathcal G}_1$, then the nonlinear equation \eqref{eq:Eqforp} has a unique positive solution \cite{WuTang2015}.

Due to the nonlinear hyperbolic nature of the equations \eqref{eq:RHD}, 
discontinuous solutions can develop from even smooth initial data, 
and weak solutions must therefore be considered. 
As well-known, weak solutions are not uniquely defined in general; 
the following inequality, known as the entropy condition, is usually imposed as an admissibility criterion to select the ``physically relevant'' solution among all weak solutions: 
\begin{equation}\label{EntropyIEQ}
	\frac{\partial  {\mathscr E} }{\partial t} + \sum_{i=1}^d \frac{\partial {\mathscr F}_i} {\partial x_i} \le 0,   
\end{equation}
which is interpreted in the sense of distribution.
Here ${\mathscr E}({\bf U})$ is a strictly convex function of ${\bf U}$ and called an entropy function, 
and ${\mathscr F}_i({\bf U})$ is the associated entropy fluxes such that the relation 
$
\frac{\partial {\mathscr E}}{\partial {\bf U}} \frac{ \partial {\bf F}_i }{\partial {\bf U}} 
= \frac{ \partial {\mathscr F}_i }{ \partial {\bf U} } 
$  holds. 
{\em Entropy solutions} are defined to be 
weak solutions which in addition satisfy \eqref{EntropyIEQ} for {\em all entropy pairs}  
$({\mathscr E},{\mathscr F}_i)$. 
For the (non-relativistic) gas dynamics equations, 
Tadmor \cite{TADMOR1986211} proved that entropy solutions satisfy a local minimum principle on the specific entropy 
$S({\bm x},t) = \log \big( {p}{\rho^{-\Gamma}} \big)$:  
\begin{equation}\label{eq:minEntropy}
	S({\bm x},t+\tau) \ge \min \left\{ S({\bm y},t):~ \|{\bm y} - {\bm x} \| \le \tau  v_{\rm max}  \right\},
\end{equation}
where $\tau > 0$ and $v_{\rm max}$ denotes the maximal wave speed. 
This implies the spatial minimum of the specific entropy, $\min_{ \bm x} S({\bm x},t)$, 
is a {\em nondecreasing} function of time $t$, and $S({\bm x},t) \ge \min_{ \bm x} S({\bm x},0)$. 
(Entropy principles were also shown by Guermond and Popov \cite{Guermond2014}    
with viscous regularization of the non-relativistic Euler equations.) 
In this paper, we will explore   
such a minimum entropy principle in the RHD case and prove that it 
also holds for 
entropy solutions of \eqref{eq:RHD}. 
The entropy principle, along with the intrinsic physical constraints in \eqref{eq:RHD:definitionG}, 
imply a global invariant region for the solution of the RHD equations \eqref{eq:RHD} with initial data ${\bf U}_0({\bm x})$, that is, 
\begin{equation}\label{eq:IR}
	\Omega_{S_0} := \left\{ \vec U=(D,{\bm m}^\top,E)^{\top}\in \mathbb R^{d+2}:~\rho(\vec U)>0,~p(\vec U)>0,~\|{\bm v}(\vec U)\| <1,~
	S({\bf U}) \ge S_0  \right\},
\end{equation}
where $S_0 := \essinf_{ {\bm x}} S ( {\bf U}_0 ({\bm x}) ) $.

It is natural and interesting to explore robust 
numerical RHD schemes, whose solutions always stay in the invariant region $\Omega_{S_0}$, 
i.e., satisfy the minimum entropy principle at the discrete level and also  
preserve the intrinsic physical constraints \eqref{eq:PHYSconstraints}. 
Note that, to obtain a well-defined specific entropy for the numerical solution, it is necessary to first guarantee 
the positivity of pressure and rest-mass density. 
The subluminal constraint on the fluid velocity is also crucial 
for the relativistic causality, because its violation would yield imaginary Lorentz factor. 
In fact, violating any of the constraints \eqref{eq:PHYSconstraints} 
will cause numerical instability and 
the break down of the computation.
Therefore, the preservation of the minimum entropy principle should be considered together with the constraints \eqref{eq:PHYSconstraints}. 
Recent years have witnessed some advances in developing 
high-order numerical schemes, which provably preserve the constraints \eqref{eq:PHYSconstraints}, for the special RHD 
\cite{WuTang2015,QinShu2016,WuTang2017ApJS}  
and the general RHD \cite{Wu2017}. 
Those works were motivated by \cite{zhang2010,zhang2010b,Xu2014,Hu2013} 
on designing bound-preserving  
high-order schemes for scalar conservation laws and the non-relativistic Euler equations.  
More recently, bound-preserving numerical schemes were also developed  
for the special relativistic magnetohydrodynamics (RMHD) in one and multiple dimensions \cite{WuTangM3AS,WuShu2020NumMath}, 
as extension of  
the 
positivity-preserving MHD schemes \cite{Wu2017a,WuShu2018,WuShu2019}. 
In addition, a flux-limiting approach which preserves the positivity of the rest-mass density was designed in 
\cite{Radice2014}. A reconstruction technique was proposed in \cite{BalsaraKim2016} 
to enforce the subluminal constraint on the fluid velocity. 
A flux-limiting entropy-viscosity approach was developed in \cite{guercilena2017entropy}
for RHD, based on a measure of the entropy generated by the solution. 
Systematic review of numerical RHD schemes is beyond the scope of the present paper; we 
refer interested readers to the review articles 
\cite{marti2003numerical,font2008numerical,marti2015grid} and a limited list of some recent works  \cite{he2012adaptive,zhao2013runge,wu2016direct,duan2019high,bhoriya2020entropy,WuShu2019SISC,mewes2020numerical} as well as references therein. 
%
Yet, up to now, there is still no work that studied   
the minimum entropy principle for the RHD equations \eqref{eq:RHD} at either the PDE level or the numerical level, 
and high-order schemes which provably preserve the invariant region \eqref{eq:IR} have {\it not} yet been developed for RHD. 

For the non-relativistic counterparts such as the compressible Euler system, 
the minimum entropy principle and  
invariant-region-preserving (IRP) numerical schemes have been well studied in the literature. 
Tadmor \cite{TADMOR1986211} discovered \eqref{eq:minEntropy} and proved, for the
compressible Euler equations of the gas dynamics, that first-order approximations such as the Godunov and Lax-Friedrichs schemes satisfy a minimum entropy principle. 
Using a slope reconstruction with limiter, Khobalatte and Perthame \cite{khobalatte1994maximum} 
developed second-order kinetic schemes that preserve a discrete minimum principle for the specific entropy. 
It was also observed in \cite{khobalatte1994maximum} that enforcing the discrete minimum entropy principle 
could help to damp numerical oscillations near the discontinuities. 
Zhang and Shu \cite{zhang2012minimum} proposed a framework of enforcing the minimum entropy principle
for high-order accurate finite volume and discontinuous Galerkin (DG) schemes, by extending their positivity-preserving high-order schemes \cite{zhang2010b,zhang2012maximum}, for the (non-relativistic) gas dynamics equations. 
The resulting high-order schemes in \cite{zhang2012minimum} were proven to 
preserve a discrete minimum entropy principle 
and the positivity of density and pressure, under a condition accessible by a simple bound-preserving limiter without destroying the high-order accuracy.  
Lv and Ihme \cite{LV2015715} proposed an entropy-bounded DG scheme
for the Euler equations on arbitrary meshes. 
Guermond, Popov, and their collaborators (cf.~\cite{guermond2016invariant,guermond2017invariant,guermond2018second,guermond2017invariant2,guermond2019invariant}) developed the IRP approximations 
in the context of continuous finite elements with convex limiting for solving  
general hyperbolic systems including the compressible Euler equations. 
Jiang and Liu proposed new IRP limiters for the DG schemes to the isentropic Euler equations 
\cite{jiang2019invariant}, the compressible Euler equations \cite{jiang2016invariant}, and general multi-dimensional hyperbolic conservation laws \cite{jiang2018invariant}. 
Gouasmi {\it et al.} \cite{gouasmi2020minimum}
proved a minimum entropy principle on entropy solutions to the compressible multicomponent Euler equations at the smooth and discrete levels. 

The aim of this paper is twofold. The first is to show that the minimum entropy principle \eqref{eq:minEntropy}, which was 
originally demonstrated by Tadmor \cite{TADMOR1986211} for the (non-relativistic) gas dynamics, 
is also valid for the RHD equations \eqref{eq:RHD} with the ideal EOS \eqref{eq:iEOS}. 
A key point in the present study is to prove a 
condition on smooth function ${\mathcal H}(S)$ such that the entropy function
${\mathscr E} ({\bf U}) = -D{\mathcal H}(S)$ is strictly convex. 
The second goal is to develop high-order accurate IRP DG and finite volume methods which 
provably preserve the numerical solutions in the invariant region $\Omega_{S_0}$, 
i.e., preserve a discrete minimum entropy principle and the intrinsic physical constraints \eqref{eq:PHYSconstraints}. 
In fact, achieving these two goals is nontrivial. 
Due to the nonlinearity and the implicit form of the function $S({\bf U})$, 
it is not easy to study the convexity of the entropy function
${\mathscr E} ({\bf U})$ in the RHD case; see Proposition \ref{prop:H}. 
Also, analytically judging whether an arbitrarily given state ${\bf U}$ belongs to $\Omega_{S_0}$ is already a difficult task; it is more challenging to design and analyze numerical schemes that provably preserve the solutions in  $\Omega_{S_0}$. 
We will address the difficulties via 
a novel {\em equivalent} form of the invariant region; see Theorem \ref{thm:EqOmega}. 
As a notable feature, 
all the constraints in the novel form are {\em explicit} and {\em linear} with respect to the conservative variables. 
This provides a highly effective approach to  
theoretically analyze the IRP property of RHD schemes. 
Based on this, we will prove the convexity of the invariant region (Section \ref{sec:convex}) and establish the 
generalized Lax--Friedrichs splitting properties via highly technical estimates (Section \ref{sec:gLF}), which lie the foundation for analyzing our IRP schemes in Sections \ref{sec:1Dscheme}--\ref{sec:2Dschemes}. 
The high-order accurate IRP schemes are constructed  
with the aid of 
a simple scaling limiter, which is designed by following the bound-preserving type limiters and frameworks in the literature  \cite{zhang2010b,zhang2012minimum,QinShu2016,jiang2018invariant,WuTang2017ApJS}. 

It is worth noting that the proposed IRP analysis approach has some essential and significant differences from those in the literature (cf.~\cite{zhang2012minimum,guermond2016invariant,guermond2018second,JIANG2018}). 
For example, some standard analyses were often based on the IRP property of the exact Riemann solver for the studied equations, 
while our IRP analyses do {\em not} require any assumption on 
the IRP property of the exact (or any approximate) Riemann solver.\footnote{It is certainly reasonable to assume that 
	the exact Riemann solver preserves the invariant domain.   
	In fact, this is provenly true for a number of other systems, 
	but has not yet been  
	proven for the RHD equations \eqref{eq:RHD}. 
	Rigorous analysis on the IRP property of the exact Riemann solver, including the preservation of the constraints \eqref{eq:PHYSconstraints},    
	is highly nontrivial.} 
This makes our analysis approach potentially extensible to some other complicated 
physical systems for which the exact Riemann solver is not easily available.

This paper is organized as follows. We study the 
minimum entropy principle at the PDE level in Section 
\ref{sec:PDE}. After establishing some auxiliary theories for IRP analysis in Section \ref{sec:theory}, 
we present the IRP schemes in Sections \ref{sec:1Dscheme}--\ref{sec:2Dschemes} for one- and multi-dimensional RHD equations, respectively. 
Numerical examples are provided in Section \ref{sec:examples} 
 and will confirm that 
incorporating the minimum entropy principle into a scheme  
could be helpful for damping some undesirable numerical oscillations, as 
observed in e.g.~\cite{khobalatte1994maximum,zhang2012minimum,JIANG2018} for 
some other systems. 
Section \ref{sec:con} concludes the paper.

\section{Minimum entropy principle at the PDE level}\label{sec:PDE}

\subsection{Convex entropy functions}\label{sec:entropyfun} 
We first explore the space of admissible entropy functions. 
Let ${\mathcal H}(S)$ be a function of the specific entropy $S$. As shown 
in 
\cite{bhoriya2020entropy}
for $d=1$, for any smooth function ${\mathcal H}(S)$, the smooth solutions of 
the RHD equations \eqref{eq:RHD} satisfy 
\begin{equation}\label{HScon}
\frac{\partial }{ \partial t } \Big( -\rho W {\mathcal H}(S)  \Big) + \sum_{i=1}^d \frac{\partial }{ \partial {x_i} } \Big( -\rho W v_i {\mathcal H}(S)  \Big) = 0. 
\end{equation}
This implies that $({\mathscr E},{\mathscr F}_i) = \big( -D{\mathcal H}(S),-D v_i {\mathcal H}(S) \big)$ 
is an entropy--entropy flux pair, if ${\mathscr E} ({\bf U}) = -D {\mathcal H}(S)$ is a strictly convex function of the conservative variables ${\bf U} \in {\mathcal G}$. 

It is well-known that, for a special choice ${\mathcal H}(S) =  S $ or ${\mathcal H}(S) =  \frac{ S}{\Gamma -1} $,  the corresponding ${\mathscr E} ({\bf U})$ is a valid entropy function for the RHD; see, for example, 
\cite{guercilena2017entropy,bhoriya2020entropy,duan2019high,WuShu2019SISC}. 
However, it is unclear, for a general ${\mathcal H}(S)$,  
what is the condition on ${\mathcal H}(S)$ such that the corresponding ${\mathscr E} ({\bf U})$ is strictly convex.  
This has not been addressed for the RHD case in the literature and is now explored in the following proposition.

\begin{proposition}\label{prop:H}
For a smooth function ${\mathcal H}(S)$, the corresponding ${\mathscr E} ({\bf U}) = -D{\mathcal H}(S)$ is a strictly convex entropy function if and only if 
\begin{equation}\label{EntropyFunHcond}
{\mathcal H}'(S) > 0, \qquad  {\mathcal H}'(S) - \Gamma {\mathcal H}''(S) > 0.
\end{equation}
\end{proposition}

\begin{proof}
We study the convexity of ${\mathscr E} ({\bf U})$ by investigating the positive definiteness of the associated Hessian matrix 
$${\mathscr E}_{{\bf u}{\bf u}} := \left( \frac{ \partial^2 {\mathscr E}} { \partial u_i \partial u_j }  \right)_{ 1\le i,j \le d+2 },$$
where $u_i$ denotes the $i$th component of ${\bf U}$. A straightforward computation gives 
$$
\frac{\partial^2 {\mathscr E}} { \partial u_i \partial u_j } = - {\mathcal H}'(S) 
\left( 
\frac{ \partial D}{ \partial u_i} \frac{\partial S }{\partial u_j} 
+ \frac{ \partial D}{ \partial u_j} \frac{\partial S }{\partial u_i} 
\right) 
- D {\mathcal H}''(S) \frac{\partial S} {\partial u_i } \frac{\partial S} {\partial u_j }
- D {\mathcal H}'(S) \frac{\partial^2 S} { \partial u_i \partial u_j } ,
$$
which implies that 
	\begin{align} 
	\nonumber
	{\mathscr E}_{{\bf u}{\bf u}} &= - {\mathcal H}'(S) \left( {\bf e}_1 S_{\bf u}^\top 
	+ S_{\bf u} {\bf e}_1^\top + D  S_{\bf uu} \right) - D {\mathcal H}''(S)  S_{\bf u} S_{\bf u}^\top
	\\ \label{proof223}
	&
	= - {\mathcal H}'(S) {\bf A}_1 +  \frac{D}{\Gamma} \left(  {\mathcal H}'(S) - \Gamma {\mathcal H}''(S) \right)  S_{\bf u} S_{\bf u}^\top,
\end{align}
where  
${\bf e}_1=(1,{\bf 0}_{d+1}^\top)^\top$, ${\bf 0}_{d+1}$ denotes the zero vector of length $d+1$, 
and 
$$
{\bf A}_1:= {\bf e}_1 S_{\bf u}^\top 
+ S_{\bf u} {\bf e}_1^\top + D  S_{\bf uu} 
+ \frac{1}{\Gamma} D S_{\bf u} S_{\bf u}^\top.$$  
Since $S$ cannot be explicitly formulated in terms of ${\bf U}$, direct derivation of $S_{\bf u}$ and $S_{\bf uu}$ is difficult. Let us consider the primitive variables ${\bf V} = ( \rho, {\bm v}^\top, p)^\top$. Note that both $S$ and $\bf U$ can be explicitly formulated in terms of $\bf V$, then it is easy to derive 
$$
\frac{\partial S}{\partial {\bf V}} = \left( - \Gamma / \rho, {\bf 0}_d^\top, 1/p \right), 
\qquad 
\frac{\partial {\bf U}}{\partial {\bf V}} = \begin{pmatrix}
W & \rho W^3 {\bm v}^\top  & 0\\
W^2 {\bm v} & \rho h W^2 {\bf I}_d + 2 \rho h W^4 {\bm v} {\bm v}^\top & \frac{\Gamma W^2}{\Gamma -1} {\bm v}
\\
W^2 & 2\rho h W^4 {\bm v}^\top & \frac{\Gamma W^2}{\Gamma -1} - 1
\end{pmatrix},
$$
where ${\bf I}_d$ denotes the identity matrix of size $d$. The inverse of the matrix  $\frac{\partial {\bf U}}{\partial {\bf V}}$ gives 
$$
\frac{\partial {\bf V}}{\partial {\bf U}} = 
\frac1{\rho h ( 1 - c_s^2 \| {\bm v} \|^2) }
\begin{pmatrix}
\rho h ( 1 -(\Gamma-1)\| {\bm v} \|^2 ) W^{-1} & - \rho ( 1 + (\Gamma -1) \| {\bm v} \|^2  ) {\bm v}^\top & \rho \Gamma \| {\bm v} \|^2
\\
(\Gamma -1) W^{-3} {\bm v} & 
{\bf A}_2
& \Gamma \left(  \| {\bm v} \|^2 - 1 \right) {\bm v}
\\
- ( \Gamma p + (\Gamma -1 )\rho ) W^{-1}
	& - ( 2\Gamma p + (\Gamma -1 )\rho ) {\bm v}^\top & \Gamma p \left( 1 + \| {\bm v} \|^2 \right) + (\Gamma -1 )\rho
\end{pmatrix}
$$
with $c_s=\sqrt{ \frac{\Gamma p}{\rho h}} $ denoting the acoustic wave speed in the RHD case (note that $0<c_s<1$), and 
$${\bf A}_2:= \left( 1 - \| {\bm v} \|^2 \right) \left[ \left( 1-  c_s^2 \| {\bm v} \|^2    \right)  {\bf I}_d + 
\left( \Gamma -1 + c_s^2 \right) {\bm v} {\bm v}^\top \right].$$ 
It follows that 
\begin{align*}
& S_{\bf u}^\top = \frac{\partial S}{\partial {\bf U}} = \frac{\partial S}{\partial {\bf V}}  \frac{\partial {\bf V}}{\partial {\bf U}} 
= \frac{\Gamma-1}{p}   \left( -h W^{-1},~ - {\bm v}^\top,~ 1  \right).
\end{align*}
The derivative of $S_{\bf u}^\top$ with respect to ${\bf V}$ gives
$$
S_{\bf u v} = \begin{pmatrix}
\frac{ \Gamma }{\rho^2} \sqrt{ 1 - \| {\bm v} \|^2 } & \frac{\Gamma-1}{p} h W {\bm v}^\top & \frac{ \Gamma - 1 }{p^2} \sqrt{ 1 - \| {\bm v} \|^2 }
\\
{\bf 0}_d & -\frac{\Gamma - 1}{p} {\bf I}_d & \frac{\Gamma-1}{p^2} {\bm v}
\\
0 & {\bf 0}_d^\top & -\frac{\Gamma-1}{p^2}
\end{pmatrix}.
$$
Then we obtain 
\begin{align*}
{\bf A}_1 &= {\bf e}_1 S_{\bf u}^\top 
+ S_{\bf u} {\bf e}_1^\top + D  S_{\bf u v}   \frac{\partial {\bf V}}{\partial {\bf U}} 
+ \frac{1}{\Gamma} D S_{\bf u} S_{\bf u}^\top
\\
 & =  \frac{ 1-\Gamma  }{p h ( h-1 ) ( 1- c_s^2 \| {\bm v} \|^2  )}
\begin{pmatrix}
a_1
& a_2 {\bm v}^\top  
& a_3
\\
a_2 {\bm v}
& 
{\bf A}_3
& a_4 {\bm v}
\\
a_3
& 
a_4 {\bm v}^\top
& 
a_5
\end{pmatrix}
\end{align*}
with 
\begin{align*}
&a_1 := {  h(\Gamma-1)W^{-1} } > 0,
\qquad a_2 := {  (2h-1)(\Gamma -1) }, 
\qquad a_3 := - {  (\Gamma-1) ( h + (h-1) \| {\bm v} \|^2 ) },
\\
& {\bf A}_3 := \frac{ { (h-1) ( 1- c_s^2 \| {\bm v} \|^2  )  } }{W} {\bf I}_d + 
W
\left( 
{ (h-1) ( 1- c_s^2 \| {\bm v} \|^2  )  } +
{    \frac{1}{h} (\Gamma -1)(2h-1)^2  }
\right)  {\bm v}{\bm v}^\top,
\\
& a_4 := W 
\left(    h ( 1- c_s^2 \| {\bm v} \|^2 ) - { 
	 \Gamma (2h-1) }
\right),
\qquad a_5 := {  W \left( (h-1)(2\Gamma -1) \| {\bm v} \|^2 + ( \Gamma -1 )h \right) }.
\end{align*}
Let us define the invertible matrix
$$
{\bf P}_1 = 
\begin{pmatrix}
1 & {\bf 0}_d^\top & 0 
\\
-\frac{a_2}{a_1} {\bm v} & {\bf I}_d & {\bf 0}_d
\\
-\frac{a_3}{a_1} & {\bf 0}_d^\top & 1
\end{pmatrix}.
$$
A straightforward computation gives 
\begin{equation}\label{proof224}
{\bf P}_1 {\bf A}_1 {\bf P}_1^\top 
=  \frac{  1-\Gamma  }{ph ( h-1 ) ( 1- c_s^2 \| {\bm v} \|^2  )}
\begin{pmatrix}
a_1
& {\bf 0}_{d+1}^\top
\\
{\bf 0}_{d+1}
& a_6 {\bf A}_4
\end{pmatrix}
\end{equation}
with $a_6:= (h-1) { W \left(1 - c_s^2 \| {\bm v} \|^2 \right) } > 0$, and 
\begin{equation}\label{eq:A4}
{\bf A}_4 := 
\begin{pmatrix}
(1-\|{\bm v}\|^2) {\bf I}_d + {\bm v} {\bm v}^\top~ & ~-{\bm v}\\
-{\bm v}^\top~ & ~\|{\bm v}\|^2
\end{pmatrix}.
\end{equation}
Note that 
\begin{equation}\label{proof225}
{\bf P}_1 S_{\bf u} = \frac{\Gamma-1}{p} \Big( -hW^{-1}, 2(h-1) {\bm v}^\top, (1-h)(1+\|{\bm v}\|^2)  \Big)^\top 
=: \frac{\Gamma-1}{p} {\bm b}_1.
\end{equation}
Combining equations \eqref{proof223}, \eqref{proof224} and \eqref{proof225} gives 
\begin{equation}\label{proof123e}
{\bf P}_1 {\mathscr E}_{{\bf u}{\bf u}} {\bf P}_1^\top 
=  a_7 {\mathcal H}'(S) {\bf A}_5 +  a_8 \big(  {\mathcal H}'(S) - \Gamma {\mathcal H}''(S) \big) {\bm b}_1 {\bm b}_1^\top 
\end{equation}
with $a_7:=\frac{\Gamma -1}{ph( h-1 ) ( 1- c_s^2 \| {\bm v} \|^2  )}>0$, $a_8:=\frac{D(\Gamma -1)^2}{p^2\Gamma }>0$, and 
$$
{\bf A}_5 := 
\begin{pmatrix}
a_1
& {\bf 0}_{d+1}^\top 
\\
{\bf 0}_{d+1}
& a_6 {\bf A}_4
\end{pmatrix}.
$$
Let us study the property of ${\bf A}_4$ defined in \eqref{eq:A4}. 
The matrix $(1-\|{\bm v}\|^2) {\bf I}_d + {\bm v}{\bm v}^\top$ is symmetric and 
its eigenvalues consist of $1$ and $1-\|{\bm v}\|^2$, which are all positive, implying that 
$(1-\|{\bm v}\|^2) {\bf I}_d + {\bm v} {\bm v}^\top$ is positive definite. 
Furthermore, a straightforward calculation shows $\det ( {\bf A}_4 )=0$. Therefore, ${\bf A}_4$ 
is positive semi-definite, and ${\rm rank}({\bf A}_4) = d$. 
Since $a_1>0$ and $a_6>0$, it follows that 
${\bf A}_5$ is positive semi-definite, and ${\rm rank}({\bf A}_5) = d+1$. 
Hence, there exists a rank-$(d+1)$ matrix ${\bf A}_6 \in \mathbb R^{(d+1)\times (d+2)}$ such that 
\begin{equation}\label{proofwdfd}
{\bf A}_6^\top {\bf A}_6={\bf A}_5.
\end{equation}

Because ${\mathscr E}_{{\bf u}{\bf u}}$ and  ${\bf P}_1 {\mathscr E}_{{\bf u}{\bf u}} {\bf P}_1^\top $ are congruent, it suffices to prove that the matrix 
${\bf P}_1 {\mathscr E}_{{\bf u}{\bf u}} {\bf P}_1^\top $ is positive definite if and only if ${\mathcal H}(S)$ satisfies the condition \eqref{EntropyFunHcond}. 

{\tt (\romannumeral1)}. First prove the condition \eqref{EntropyFunHcond} 
	is sufficient for the positive definiteness of ${\bf P}_1 {\mathscr E}_{{\bf u}{\bf u}} {\bf P}_1^\top $.   
Because ${\bf A}_5$ and ${\bm b}_1 {\bm b}_1^\top$ are both positive semi-definite, 
by \eqref{proof123e} we know that ${\bf P}_1 {\mathscr E}_{{\bf u}{\bf u}} {\bf P}_1^\top$ 
is positive semi-definite under the condition \eqref{EntropyFunHcond}. It means 
\begin{equation}\label{eq:proof21}
{\bm z}^\top {\bf P}_1 {\mathscr E}_{{\bf u}{\bf u}} {\bf P}_1^\top {\bm z} \ge 0, \qquad \forall {\bm z} 
\in \mathbb R^{d+2}. 
\end{equation}
Hence, it suffices to show ${\bm z}={\bf 0}$ when ${\bm z}^\top {\bf P}_1 {\mathscr E}_{{\bf u}{\bf u}} {\bf P}_1^\top {\bm z} = 0$. Using \eqref{proof123e} and \eqref{proofwdfd}, we have 
$$
{\bm z}^\top {\bf P}_1 {\mathscr E}_{{\bf u}{\bf u}} {\bf P}_1^\top {\bm z} = 
a_7 {\mathcal H}'(S) \| {\bf A}_6 {\bm z} \|^2 + a_8 \big(  {\mathcal H}'(S) - \Gamma {\mathcal H}''(S) \big) 
\left| {\bm b}_1^\top {\bm z} \right|^2 = 0,
$$
which implies ${\bf A}_6 {\bm z}={\bf 0}_{d+1}$ and ${\bm b}_1^\top {\bm z} =0$. Then ${\bf A}_5 {\bm z}={\bf A}_6^\top {\bf A}_6 {\bm z}={\bf 0}$. Let ${\bm z}=:(z^{(1)},{\bm z}^{(2)},z^{(3)})^\top$ with 
${\bm z}^{(2)}\in \mathbb R^d$. From ${\bf A}_5 {\bm z}={\bf 0}$ we can deduce that 
$a_1 z^{(1)} = 0$ and $a_6{\bf A}_4(  {\bm z}^{(2)},z^{(3)})^\top = {\bf 0}$.  
It further yields $z^{(1)} = 0$ and 
\begin{align}\label{proof1112}
& (1-\|{\bm v}\|^2) {\bm z}^{(2)}  +  {\bm v}{\bm v}^\top {\bm z}^{(2)} -  z^{(3)} {\bm v}  = {\bf 0}_d
\\ \label{proof11142}
& -{\bm v}^\top {\bm z}^{(2)} + \|{\bm v} \|^2 z^{(3)} = 0.
\end{align}
Combining $z^{(1)} =0$ and ${\bm b}_1^\top {\bm z} =0$ gives 
$$
2(h-1) {\bm v}^\top {\bm z}^{(2)}  + (1-h) ( 1 + \|{\bm v} \|^2 ) z^{(3)} = 0,
$$
which, together with \eqref{proof11142}, imply 
${\bm v}^\top {\bm z}^{(2)} = z^{(3)}=0$. Substituting it into \eqref{proof1112} gives 
${\bm z}^{(2)}={\bf 0}_d$. Therefore, we have ${\bm z}={\bf 0}$ when ${\bm z}^\top {\bf P}_1 {\mathscr E}_{{\bf u}{\bf u}} {\bf P}_1^\top {\bm z} = 0$. 
This along with \eqref{eq:proof21} yield that ${\bf P}_1 {\mathscr E}_{{\bf u}{\bf u}} {\bf P}_1^\top$ 
is positive definite under the condition \eqref{EntropyFunHcond}. 
This completes the proof of sufficiency.

{\tt (\romannumeral2)}. Then prove the condition \eqref{EntropyFunHcond} 
is necessary for the positive definiteness of ${\bf P}_1 {\mathscr E}_{{\bf u}{\bf u}} {\bf P}_1^\top $.  
Assume that ${\bf P}_1 {\mathscr E}_{{\bf u}{\bf u}} {\bf P}_1^\top $ is positive definite, then 
\begin{equation}\label{eq:proof2}
{\bm z}^\top {\bf P}_1 {\mathscr E}_{{\bf u}{\bf u}} {\bf P}_1^\top {\bm z} 
= a_7 {\mathcal H}'(S) {\bm z}^\top {\bf A}_5 {\bm z} + a_8 \big(  {\mathcal H}'(S) - \Gamma {\mathcal H}''(S) \big) 
\left| {\bm b}_1^\top {\bm z} \right|^2
>0, ~~ \forall {\bm z} 
\in \mathbb R^{d+2} \setminus \{{\bf 0}\}. 
\end{equation}
Note that the matrix ${\bf A}_5$ does not have full rank. There exist two vectors ${\bm z}_1, {\bm z}_2 \in \mathbb R^{d+2}\setminus \{{\bf 0}\}$ such that ${\bf A}_5 {\bm z}_1 = {\bf 0}$ and ${\bm b}_1^\top {\bm z}_2 = 0$, respectively. 
It follows from \eqref{eq:proof2} that 
\begin{align*}
& 0< {\bm z}_2^\top {\bf P}_1 {\mathscr E}_{{\bf u}{\bf u}} {\bf P}_1^\top {\bm z}_2 = 
a_7 {\mathcal H}'(S) {\bm z}_2^\top {\bf A}_5 {\bm z}_2= a_7 {\mathcal H}'(S) \| {\bf A}_6 {\bm z}_2 \|^2 ,
\\
& 0< {\bm z}_1^\top {\bf P}_1 {\mathscr E}_{{\bf u}{\bf u}} {\bf P}_1^\top {\bm z}_1=
a_8 \big(  {\mathcal H}'(S) - \Gamma {\mathcal H}''(S) \big) 
\left| {\bm b}_1^\top {\bm z}_1 \right|^2,
\end{align*}
which implies ${\mathcal H}'(S)>0$ and ${\mathcal H}'(S) - \Gamma {\mathcal H}''(S)>0$, respectively. 
This completes the proof of necessity. 
\end{proof}

\begin{remark}
	Proposition \ref{prop:H} and equation \eqref{HScon} imply that there exists a family of (generalized) entropy pairs $({\mathscr E},{\mathscr F}_i)$ associated with the $d$-dimensional ($1\le d \le 3$) RHD equations \eqref{eq:RHD}, 
	\begin{equation}\label{eq:entropys}
		{\mathscr E}({\bf U}) = - D {\mathcal H} (S), \qquad {\mathscr F}_i ( {\bf U} ) = - D v_i {\mathcal H}(S), \quad i=1,\dots,d,
	\end{equation}
	generated by the smooth functions ${\mathcal H}(S)$ satisfying \eqref{EntropyFunHcond}. 
	Our found condition \eqref{EntropyFunHcond} is consistent with the one derived by Harten 
	\cite[Section 2]{HARTEN1983151} for the 2D non-relativistic Euler equations.  However, the analysis in the RHD case does not directly follow from \cite{HARTEN1983151} and is more difficult. 
	Due to the complicated structures of the matrix ${\mathscr E}_{{\bf u}{\bf u}}$, 
	some standard approaches for investigating its positive definiteness, e.g., checking the  
	positivity of its leading principal minors, can be intractable in our RHD case. 
\end{remark}

\subsection{Minimum principle of the specific entropy}\label{sec:ME}
We are now in a position to 
verify that 
Tadmor's minimum entropy principle \eqref{eq:minEntropy} also hold for   
the RHD system \eqref{eq:RHD}. 
We consider the convex entropy ${\mathscr E}({\bf U})=-D {\mathcal H}(S)$ established in Section \ref{sec:entropyfun}   
for all smooth functions ${\mathcal H}(S)$ satisfying \eqref{EntropyFunHcond}.

Assume that ${\bf U}({\bm x},t)$ is an entropy solution of the RHD equations \eqref{eq:RHD}. 
According to  
\cite[Theorem 4.1]{tadmor1984skew} and following \cite[Lemma 3.1]{TADMOR1986211}, we have, 
for all smooth functions ${\mathcal H}(S)$ satisfying \eqref{EntropyFunHcond}, 
\begin{equation}\label{eq:enEQ}
	\int_{ \|  {\bm x} -{\bm x}_0 \| \le R } D({\bm x},t+\tau) {\mathcal H} ( S ( {\bm x}, t+\tau ) ) 
	{\rm d} {\bm x} \ge \int_{ \|  {\bm x} -{\bm x}_0 \| \le R + \tau v_{\rm max} } 
	D({\bm x},t) {\mathcal H} ( S ( {\bm x}, t  ) )  {\rm d} {\bm x}, \quad \forall R >0,
\end{equation}
where $v_{\max}$ denotes the maximal wave speed in the domain; we can take $v_{\max} $ as  the speed of light $c = 1$, a simple upper bound of all  wave speeds in the RHD case. 
Note that the density involved in \eqref{eq:enEQ} is $D=\rho W$, instead of the rest-mass density $\rho$. 

Consider a special function ${\mathcal H}_0(S)$ \cite{TADMOR1986211} defined by 
$$
{\mathcal H}_0 ( S ) := \min \{ S-S_0, 0 \}, \qquad S_0 =  \min \left\{ S({\bm y},t):~ \|{\bm y} - {\bm x}_0 \| \le R+ \tau  v_{\rm max}  \right\}.
$$
As observed in \cite{TADMOR1986211}, the function ${\mathcal H}_0(S)$, although not smooth, can be written 
as the limit of a sequence of smooth functions satisfying \eqref{EntropyFunHcond}; see also \cite[Section 3.1]{gouasmi2020minimum} for a detailed review. 
Therefore, by passing to the limit, the inequality \eqref{eq:enEQ} holds for ${\mathcal H} = {\mathcal H}_0 $, which gives 
\begin{align*}
	\int_{ \|  {\bm x} -{\bm x}_0 \| \le R } D({\bm x},t+\tau) \min\{  S ( {\bm x}, t+\tau ) - S_0, 0 \}
	{\rm d} {\bm x} 
	\ge \int_{ \|  {\bm x} -{\bm x}_0 \| \le R + \tau v_{\rm max} } 
	D({\bm x},t) {\mathcal H}_0 ( S ( {\bm x}, t ) )   
	{\rm d} {\bm x} = 0. 
\end{align*}
Because $D({\bm x},t+\tau) >0$, we obtain $\min\{  S ( {\bm x}, t+\tau ) - S_0, 0 \}=0$ for 
$\| {\bm x} - {\bm x}_0 \|\le R$. It leads to 
\begin{equation}\label{key342}
	S ( {\bm x}, t+\tau ) \ge  S_0 = \min \left\{ S({\bm y},t):~ \|{\bm y} - {\bm x}_0 \| \le R+ \tau  v_{\rm max}  \right\}, \qquad \forall \| {\bm x} -{\bm x}_0 \| \le R, 
\end{equation}
which yields the local minimum entropy principle \eqref{eq:minEntropy}. 
In particular, it implies that the spatial minimum of the specific entropy, $\min_{ \bm x} S({\bm x},t)$, 
is a nondecreasing function of time $t$, yielding 
\begin{equation}\label{key3421}
	S ( {\bm x}, t) \ge \min_{ \bm x} S( {\bm x}, 0 ), \qquad \forall t\ge 0. 
\end{equation}

In the above derivation it is implicitly assumed that ${\bf U}({\bm x},t)$ always satisfies the physical constraints \eqref{eq:PHYSconstraints}. 
The entropy principle \eqref{key3421} and the constraints  \eqref{eq:PHYSconstraints}  
constitute the global invariant region 
$\Omega_{S_0}$, defined in \eqref{eq:IR}, 
for entropy solutions of the RHD equations \eqref{eq:RHD}.

\section{Auxiliary theories for numerical analysis}\label{sec:theory}

In order to analyze the local minimum entropy principle of 
numerical schemes, we introduce a (more general) ``local'' invariant region for an arbitrarily  given $\sigma$:  
\begin{equation}\label{eq:IRgen}
	\Omega_{\sigma} := \left\{ \vec U=(D,{\bm m}^\top,E)^{\top}\in \mathbb R^{d+2}:~\rho(\vec U)>0,~p(\vec U)>0,~\|{\bm v}(\vec U)\| <1,~
	S({\bf U}) \ge \sigma  \right\}.
\end{equation}
The special choice $\sigma = S_0 = \essinf_{ {\bm x}} S ( {\bm x}, 0 ) $ corresponds to the global invariant region 
$\Omega_{S_0}$ in \eqref{eq:IR}. 
It is evident that the following ``monotonicity'' holds for $\Omega_{\sigma}$.
\begin{lemma}[{\bf Monotonic decreasing}]\label{thm:Mono}
	If $\sigma_1 \ge \sigma_2$, then $\Omega_{\sigma_1} \subseteq  \Omega_{\sigma_2}$.
\end{lemma}
Thanks to ${\mathcal G}={\mathcal G}_1$ proved in \cite[Lemma 2.1]{WuTang2015}, we immediately have 
\begin{lemma}[{\bf First equivalent form}]\label{thm:EqOmega1}
	The invariant region set $\Omega_{\sigma}$ 
	is equivalent to 
	\begin{equation}\label{eq:IRgen2}
		\Omega_{\sigma}^{(1)} = \left\{ \vec U=(D,{\bm m}^\top,E)^{\top}\in \mathbb R^{d+2}:~D>0,~E>\sqrt{D^2 + \|{\bm m}\|^2},~
		S({\bf U}) \ge \sigma  \right\}.
	\end{equation}
\end{lemma}

Note that the specific entropy $S=\log(p \rho^{-\Gamma})$ is a nonlinear function of $(\rho,p)$, 
and, as mentioned in Section \ref{sec:intro}, the functions $\rho({\bf U})$ and $p({\bf U})$ are  already highly nonlinear and without explicit formulas. 
The combination of these nonlinear functions leads to 
$S({\bf U})$, which is certainly a 
highly nonlinear function and also cannot be explicitly formulated in terms of ${\bf U}$. 
This 
causes it difficult to study the minimum entropy principle at the numerical level and explore IRP schemes for RHD. 
In order to overcome the challenges, several important properties of the invariant region $\Omega_{\sigma}$ 
will be derived in this section.

\subsection{An explicit and linear equivalent form of invariant region}

To address the difficulties arising from the nonlinearity of  $S({\bf U})$, we discover the following novel equivalent form of $\Omega_{\sigma}$. 

\begin{theorem}[{\bf Second equivalent form}]\label{thm:EqOmega}
	The invariant region set $\Omega_{\sigma}$  
	is equivalent to 
	\begin{equation}\label{eq:EqOmega_sigma}
		\Omega_{\sigma}^{(2)} = \left\{ \vec U=(D,{\bm m}^\top,E)^{\top}\in \mathbb R^{d+2}:~D>0,~~\varphi_\sigma ( {\bf U}; {\bm v}_*, \rho_* )\ge 0,~\forall  {\bm v}_* \in \mathbb B_1({\bf 0}),~\forall \rho_* \in \mathbb R^+ \right\},
	\end{equation}
	where $\mathbb B_1({\bf 0}):= \{ {\bm x} \in \mathbb R^d: |{\bm x}|<1  \}$ denotes the open unit ball centered at ${\bf 0}$ in $\mathbb R^d$, and 
	\begin{equation}\label{varphi_def}
		\varphi_\sigma ( {\bf U}; {\bm v}_*, \rho_* ) : = E-{\bm m}\cdot {\bm v}_*-D \sqrt{1-\|{\bm v}_* \|^2} + {\mathrm e}^{\sigma} \left(  \rho_*^\Gamma - \frac{\Gamma}{\Gamma -1}   
		D \rho_*^{\Gamma -1}  \sqrt{1-\|{\bm v}_* \|^2} \right). 
	\end{equation}
\end{theorem}

Before the proof of Theorem \ref{thm:EqOmega}, we mention a crucial feature of the above equivalent form $\Omega_{\sigma}^{(2)}$. 
Note that all the ``nonlinear'' constraints in $\Omega_{\sigma}^{(1)}$ or $\Omega_{\sigma}$ 
are {\em equivalently} transformed into a {\bf \em linear}  constraint 
$\varphi_\sigma ( {\bf U}; {\bm v}_*, \rho_* )\ge 0$ in \eqref{eq:EqOmega_sigma}. 
As a result, all the constraints in $\Omega_{\sigma}^{(2)}$  are {\bf \em explicit} 
and {\bf \em linear} with respect to 
$\bf U$, although two (additional) auxiliary variables ${\bm v}_*$ and $\rho_*$ are introduced here. 
Such linearity makes $\Omega_{\sigma}^{(2)}$ very useful for 
analytically
verifying the IRP property of RHD schemes. 
This becomes a key to our IRP analysis, which is significantly different from the standard bound-preserving and IRP analysis techniques in e.g.~\cite{zhang2010b,zhang2012minimum,jiang2018invariant}.

We first give two lemmas, which will be used in the proof of Theorem \ref{thm:EqOmega}.

\begin{lemma}\label{app:lem0}
	For any $\eta >-\frac12$,  it holds 
	$$
	(\eta \Gamma +1)^{\frac1{\Gamma}} \ge (2 \eta  +1)^{\frac1{2}},
	$$
	where the (constant) adiabatic index $\Gamma \in (1,2]$.
\end{lemma}

\begin{proof}
	Consider the function 
	$
	f(x) = ( \eta x + 1 )^{\frac 1 x}
	$ with $x \in (1,2]$. 
	Note that $\eta x+1>0$, $f(x)>0$ and the derivative with respect to $x$ satisfies 
	\begin{equation*}
	\frac{ x^2 }{f(x)} f'(x) = -\log( \eta x + 1 ) + \frac{\eta x}{\eta x + 1}  
	= \log \left( 1 + \frac{-\eta x}{\eta x+1} \right) + \frac{\eta x}{\eta x + 1}  \le 0,
	\end{equation*}
	where we have used the elementary inequality $\log(1+z)\le z$ for $z > -1$. It follows that 
	$f'(x)\le 0$ for all $x \in (1,2]$. 
	This implies  
$
	f(\Gamma) = (\eta \Gamma +1)^{\frac1{\Gamma}} = f(2) - \int_{\Gamma}^2 f'(x) {\rm d} x \ge f(2) = (2 \eta +1)^{\frac1{2}}. 
$
\end{proof}

\begin{lemma}\label{app:lem1}
	For any ${\bm v},{\bm v}_* \in \mathbb B_1({\bf 0})$, it holds 
	\begin{equation}\label{keyIEQ31}
	\frac{ \Gamma ( 1 - {\bm v} \cdot {\bm v}_* ) }{  1 - \| {\bm v} \|^2   }  - \Gamma + 1 \ge \left(  \frac{  \sqrt{ 1- \| {\bm v} \|^2 } }{ 
		\sqrt{ 1- \| {\bm v}_* \|^2 } }  \right)^{-\Gamma}.
	\end{equation}
\end{lemma}

\begin{proof}
	Note that  
	$$
	\eta := \frac{  1 - {\bm v} \cdot {\bm v}_*  }{  1 - \| {\bm v} \|^2   } -1 \ge \frac{ 1 - \|{\bm v}\| }{ 
		1 - \| {\bm v } \|^2} - 1 = \frac{1}{ 1 + \| {\bm v} \| } - 1 > - \frac12.	
	$$
	With the aid of Lemma \ref{app:lem0}, we get 
	$
	(\eta \Gamma +1)^{\frac1{\Gamma}} \ge \sqrt{2 \eta + 1}, 
	$
	which yields 
	\begin{align*}
	\left( \frac{ \Gamma ( 1 - {\bm v} \cdot {\bm v}_* ) }{  1 - \| {\bm v} \|^2   }  - \Gamma + 1 
	\right)^{\frac1\Gamma}  \ge \sqrt{  2 \left( \frac{  1 - {\bm v} \cdot {\bm v}_*  }{  1 - \| {\bm v} \|^2   } -1  \right) + 1 }
	\ge \frac{ \sqrt{ 1 - \| {\bm v}_* \|^2 } } {\sqrt{1-\|{\bm v}\|^2}}.
	\end{align*}
	Then, by raising both sides to the power of $\Gamma$, we obtain \eqref{keyIEQ31}. 
\end{proof}

We are now ready to prove Theorem \ref{thm:EqOmega}.

\noindent 
{\it Proof of Theorem \ref{thm:EqOmega}}. 
The proof of $\Omega_{\sigma}=\Omega_{\sigma}^{(2)}$ is split into two parts --- showing that $\Omega_{\sigma}^{(2)} \subseteq \Omega_{\sigma}$ 
and that $\Omega_{\sigma} \subseteq \Omega_{\sigma}^{(2)}$.

{\tt (\romannumeral1). Prove that $ {\bf U} \in \Omega_{\sigma}^{(2)}  \Rightarrow {\bf U} \in \Omega_{\sigma}$}. 
When ${\bf U}=(D,{\bm m}^\top,E)^{\top} \in \Omega_{\sigma}^{(2)}$, by definition we have $D>0$ and 
\begin{equation}\label{eq:proofdsa}
\varphi_\sigma ( {\bf U}; {\bm v}_*, \rho_* )\ge 0,\qquad~\forall  {\bm v}_* \in \mathbb B_1({\bf 0}),\quad \forall \rho_* \in \mathbb R^+.
\end{equation}
If we take a special $({\bm v}_*, \rho_*) = \Big( \frac{ {\bm m} }{\sqrt{D^2 + \| {\bm m} \|^2}}, 
\frac{ D^2 }{\sqrt{D^2 + \| {\bm m} \|^2}}
 \Big)$, which satisfy  ${\bm v}_* \in \mathbb B_1({\bf 0})$ and $\rho_* > 0$, and then we obtain 
\begin{align*}
0 & \le \varphi_\sigma \left( {\bf U};  \frac{ {\bm m} }{\sqrt{D^2 + \| {\bm m} \|^2}}, 
\frac{ D^2 }{\sqrt{D^2 + \| {\bm m} \|^2}} \right)
\\
& = E - \sqrt{D^2 + \|{\bm m}\|^2} - \frac{ {\rm e}^\sigma }{ \Gamma -1 } \left( 
\frac{D^2}{\sqrt{D^2 + \| {\bm m} \|^2}}
\right)^\Gamma
\\
& < E - \sqrt{D^2 + \|{\bm m}\|^2}, 
\end{align*} 
which implies the second constraint in ${\mathcal G}_1$. This, along with $D>0$, yield ${\bf U} \in {\mathcal G}_1={\mathcal G}$. Therefore, the corresponding primitive quantities of ${\bf U}$  satisfy  
\begin{equation}\label{key12344}
\rho ({\bf U})>0, \quad p({\bf U})>0, \quad \| {\bm v} ({\bf U}) \|<1. 
\end{equation}
Taking another special $({\bm v}_*, \rho_*) = ({\bm v}({\bf U}), \rho({\bf U}) )$,  
in \eqref{eq:proofdsa} gives  
\begin{align*}
0 & \le \varphi_\sigma \left( {\bf U}; {\bm v} ({\bf U}), \rho ({\bf U})  \right)
\\
& = E-{\bm m}\cdot {\bm v}-D \sqrt{1-\|{\bm v} \|^2} + {\mathrm e}^{\sigma} \left(  \rho^\Gamma - \frac{\Gamma}{\Gamma -1}   
D \rho^{\Gamma -1}  \sqrt{1-\|{\bm v} \|^2} \right)
\\
& = \frac{1}{\Gamma -1} ( p - {\rm e}^\sigma \rho^\Gamma   ),
\end{align*} 
which, together with $\Gamma >1$, imply $p \ge {\rm e}^\sigma \rho^\Gamma $. It follows that 
$S({\bf U})=\log(p \rho^{-\Gamma}) \ge \sigma$. Combining it with \eqref{key12344}, we obtain ${\bf U} \in \Omega_{\sigma}$.

{\tt (\romannumeral2). Prove that $ {\bf U} \in \Omega_{\sigma}  \Rightarrow {\bf U} \in \Omega_{\sigma}^{(2)}$}. 
When ${\bf U}=(D,{\bm m}^\top,E)^{\top} \in \Omega_{\sigma}$, the corresponding primitive quantities  satisfy  
\begin{equation}\label{keydff}
\rho>0, \quad \| {\bm v} \|<1, \quad p \ge {\rm e}^\sigma \rho^\Gamma. 
\end{equation}
This immediately gives 
$
D = \rho W = \rho \left( 1 - \| {\bm v} \|^2 \right)^{-\frac12} > 0. 
$  
It remains to prove $\varphi_\sigma ( {\bf U}; {\bm v}_*, \rho_* )\ge 0$ for any $  {\bm v}_* \in \mathbb B_1({\bf 0})$ and any $\rho_* >0$. Let us rewrite $\varphi_\sigma ( {\bf U}; {\bm v}_*, \rho_* )$ as  
\begin{align*}
& \varphi_\sigma ( {\bf U}; {\bm v}_*, \rho_* )
=  {\Pi}_1 p + {\Pi}_2, 
\end{align*} 
with  
\begin{align*}
{\Pi}_1 & := \frac{ \Gamma }{\Gamma -1}  \left( \frac{ 1 - {\bm v} \cdot {\bm v}_* }{ 1 - \| {\bm v} \|^2 }   \right) - 1  
\ge 2 \left( \frac{ 1 - {\bm v} \cdot {\bm v}_* }{ 1 - \| {\bm v} \|^2 }   \right) - 1   
= \frac{ (1 - \| {\bm v} \| )^2 + 2 \left(\| {\bm v} \| - {\bm v} \cdot {\bm v } _* \right)  }{ 1 - \| {\bm v} \|^2 } > 0,
\\
{\Pi}_2 & := \rho \left(  \frac{ 1 - {\bm v} \cdot {\bm v}_* }{ 1 - \| {\bm v} \|^2 } 
-  \frac{ \sqrt{ 1 - \|{\bm v}_*\|^2 } } { \sqrt{ 1- \|{\bm v} \|^2 } }  \right) 
+ {\mathrm e}^{\sigma} \left(  \rho_*^\Gamma - \frac{\Gamma}{\Gamma -1}   
\rho \rho_*^{\Gamma -1}  \frac{ \sqrt{ 1 - \|{\bm v}_*\|^2 } } { \sqrt{ 1- \|{\bm v} \|^2 } }  \right),
\end{align*}
where $\Gamma \in (1,2]$ and $\| {\bm v} \|<1$ are used in showing $\Pi_1 > 0$. 
Then, using $p\ge {\rm e}^\sigma \rho^\Gamma $ gives 
\begin{align} \nonumber
\varphi_\sigma ( {\bf U}; {\bm v}_*, \rho_* ) & \ge {\Pi}_1 {\rm e}^\sigma \rho^\Gamma + {\Pi}_2
\\ \nonumber
& = \frac{ \rho  }{ 1 - \| {\bm v} \|^2 } \left( 1 - {\bm v} \cdot {\bm v}_* - \sqrt{ 1 - \| {\bm v} \|^2 } 
\sqrt{ 1-\| {\bm v}_* \|^2 } \right)  + {\rm e}^\sigma \rho^\Gamma \Pi_3  
\\ \nonumber
&\ge  \frac{ \rho  }{ 1 - \| {\bm v} \|^2 } \left( 1 -
\sqrt{ \| {\bm v} \|^2 + ( 1 - \| {\bm v} \|^2 )  }  
 \sqrt{ \| {\bm v}_* \|^2 + (  1 - \| {\bm v}_* \|^2 ) }  
 \right)  +  {\rm e}^\sigma \rho^\Gamma \Pi_3
 \\ \label{eq:dffdff}
 & = {\rm e}^\sigma \rho^\Gamma \Pi_3, 
\end{align}
where the Cauchy–Schwarz inequality has been used, and $\Pi_3 := \phi_\sigma \big( \frac{\rho}{\rho_*};  {\bm v}, {\bm v}_* \big)$ with 
$$
\phi_\sigma ( x ;  {\bm v}, {\bm v}_* ) :=  x^{-\Gamma} - \frac{ \Gamma \sqrt{ 1- \| {\bm v}_* \|^2 } }{ 
(\Gamma -1) \sqrt{ 1- \| {\bm v} \|^2 }  
} x^{-\Gamma +1} +  \frac{ \Gamma ( 1 - {\bm v} \cdot {\bm v}_* ) }{ (\Gamma -1 ) \left( 1 - \| {\bm v} \|^2  \right) }  - 1, \quad x > 0.  
$$
It is easy to verify that the function $\phi_\sigma ( x ;  {\bm v}, {\bm v}_* )$ is strictly decreasing 
on the interval $\left(0,  \frac{  \sqrt{ 1- \| {\bm v} \|^2 } }{ 
	 \sqrt{ 1- \| {\bm v}_* \|^2 }  
} \right]$ and  
strictly increasing on $\left[\frac{  \sqrt{ 1- \| {\bm v} \|^2 } }{ 
	\sqrt{ 1- \| {\bm v}_* \|^2 }  
}, +\infty \right)$ with respect to $x$.  Thefore, we have 
\begin{align}\nonumber
\Pi_3 
& \ge \min_{x \in \mathbb R^+}  \phi_\sigma ( x ;  {\bm v}, {\bm v}_* ) = 
\phi_\sigma \left(  \frac{  \sqrt{ 1- \| {\bm v} \|^2 } }{ 
	\sqrt{ 1- \| {\bm v}_* \|^2 } } ;  {\bm v}, {\bm v}_* \right) 
\\ \label{proofdf3}
&= \frac{1}{\Gamma - 1} \left[ 
-  \left(  \frac{  \sqrt{ 1- \| {\bm v} \|^2 } }{ 
	\sqrt{ 1- \| {\bm v}_* \|^2 } }  \right)^{-\Gamma}  +  \frac{ \Gamma ( 1 - {\bm v} \cdot {\bm v}_* ) }{  1 - \| {\bm v} \|^2   }  - \Gamma + 1 \right] \ge 0,
\end{align}
where we have used $\Gamma \in (1,2]$ and the inequality \eqref{keyIEQ31} derived in Lemma \ref{app:lem1}.  
Then, combining \eqref{proofdf3} with \eqref{eq:dffdff}, we conclude $\varphi_\sigma ( {\bf U}; {\bm v}_*, \rho_* )\ge {\rm e}^\sigma \rho^\Gamma \Pi_3 \ge 0$. This along with $D > 0$ imply ${\bf U} \in \Omega_{\sigma}^{(2)}$. The proof is completed. \qed 

Considering $\sigma \to -\infty$ in $\Omega_\sigma$ and using Theorem \ref{thm:EqOmega}, one can also obtain:

\begin{corollary}
	The admissible state set ${\mathcal G}$ is equivalent to 
	\begin{equation*}
		{\mathcal G}_2 := \left\{ \vec U=(D,{\bm m}^\top,E)^{\top}\in \mathbb R^{d+2}:~D>0,~~E>{\bm m}\cdot {\bm v}_*+D \sqrt{1-\|{\bm v}_* \|^2},~\forall  {\bm v}_* \in \mathbb B_1({\bf 0}) \right\}.
	\end{equation*}
\end{corollary}

\subsection{Convexity of invariant region}\label{sec:convex}

The convexity of an invariant region is a highly desirable property, as it can be used to simplify 
the IRP analysis of those numerical schemes that can be reformulated into some suitable convex combinations; 
see, for example, \cite{zhang2010b,zhang2012minimum,Wu2017a,jiang2018invariant}. 
In the RHD case, the convexity of the invariant region $\Omega_\sigma$ is discussed below. 

\begin{lemma}\label{lem:convexity}
	For any fixed $\sigma \in \mathbb R$, the invariant region $\Omega_\sigma$ is a convex set.
\end{lemma}

\begin{proof}
	Since $\Omega_\sigma = \Omega_\sigma^{(2)}$, one only needs to show the convexity of $\Omega_\sigma^{(2)}$.  	
	For any ${\bf U}_1=( D_1, {\bm m}_1^\top, E_1 )^\top$, ${\bf U}_2 = ( D_2, {\bm m}_2^\top, E_2 )^\top \in \Omega_\sigma^{(2)}$ and any $\lambda \in [0,1]$, we have $\lambda D_1  + (1-\lambda) D_2 >0$ 
	and for any ${\bm v}_* \in \mathbb B_1({\bf 0})$, $\rho_* \in \mathbb R^+$, it holds 
	\begin{equation*}
		\varphi_\sigma \big(  \lambda {\bf U}_1+ (1-\lambda) {\bf U}_2; {\bm v}_*, \rho_* \big) 
		=  \lambda \varphi_\sigma (  {\bf U}_1; {\bm v}_*, \rho_* ) + (1-\lambda)  \varphi_\sigma ( {\bf U}_2; {\bm v}_*, \rho_* ) \ge 0, 
	\end{equation*}
	where the linearity of $\varphi_\sigma (  {\bf U}; {\bm v}_*, \rho_* )$ with respect to ${\bf U}$ has been used. Hence, we obtain $ \lambda {\bf U}_1 + (1-\lambda) {\bf U}_2 \in \Omega_\sigma^{(2)}$, 
	and by definition, $\Omega_\sigma^{(2)}$ is a convex set. 
\end{proof}

The following more general conclusion can be shown by using Lemma \ref{lem:convexity} and Lemma \ref{thm:Mono}. 

\begin{lemma}\label{cor:conv}
	Let $\sigma_1$ and $\sigma_2$ be two real numbers.  For any ${\bf U}_1 \in 	\Omega_{\sigma_1} $ and any ${\bf U}_2 \in 	\Omega_{\sigma_2} $, 
	$$\lambda  {\bf U}_1 + (1-\lambda) {\bf U}_2 \in \Omega_{\min\{\sigma_1,\sigma_2\}}, \qquad \forall \lambda \in [0,1].$$	
\end{lemma}

\begin{proof}
	With the help of Lemma \ref{thm:Mono}, we have ${\bf U}_i \in \Omega_{\sigma_i} \subseteq \Omega_{\min\{\sigma_1,\sigma_2\}}$, $i=1,2$.  
	The proof is then completed by using the convexity of $\Omega_{\min\{\sigma_1,\sigma_2\}}$. 
\end{proof}

\begin{remark}
	An alternative proof of Lemma \ref{cor:conv} or Lemma \ref{lem:convexity} is based on the first equivalent set  $\Omega_\sigma^{(1)}$. 
	Note that ${\mathcal G}_1$ is a convex set \cite{WuTang2015}. For any 
	${\bf U}_1 \in 	\Omega_{\sigma_1}^{(1)} \subset {\mathcal G}_1 $ and any ${\bf U}_2 \in 	\Omega_{\sigma_2} ^{(1)} \subset {\mathcal G}_1 $, one has 
	\begin{equation}\label{eq:proofalt}
		{\bf U}_\lambda : = \lambda  {\bf U}_1 + (1-\lambda) {\bf U}_2 \in {\mathcal G}_1, \qquad \forall \lambda \in [0,1].
	\end{equation}	
	Let ${\mathscr E}({\bf U}) = - D S({\bf U})$. According to Proposition \ref{prop:H}, ${\mathscr E}({\bf U})$ is a convex function of ${\bf U}$ on ${\mathcal G}_1$. Using Jensen's inequality gives 
	$
	{\mathscr E}({\bf U}_\lambda) \le \lambda  {\mathscr E}( {\bf U}_1 ) + (1-\lambda) {\mathscr E} ({\bf U}_2) 
	$ for all $\lambda \in [0,1]$. Let $D_\lambda>0$ denote the first component of ${\bf U}_\lambda$. It follows that 
	$$
	-D_\lambda S( {\bf U}_\lambda ) \le - \lambda D_1 S( {\bf U}_1 ) - (1-\lambda) D_2 S ({\bf U}_2) 
	\le -\lambda D_1 \sigma_1 - (1-\lambda) D_2 \sigma_2 \le -D_\lambda \min\{ \sigma_1, \sigma_2 \}.
	$$
	This implies $S( {\bf U}_\lambda ) \ge \min\{ \sigma_1, \sigma_2 \}$, which along with \eqref{eq:proofalt}  
	yield ${\bf U}_\lambda \in \Omega_{\min\{\sigma_1,\sigma_2\}},\forall\lambda \in [0,1]$.
\end{remark}

\subsection{Generalized Lax-Friedrichs splitting properties}\label{sec:gLF}

In the bound-preserving analysis of numerical schemes with the Lax-Friedrichs (LF) flux, the 
following property \eqref{eq:LFproperty} is usually expected: 
\begin{equation}\label{eq:LFproperty}
	{\bf U} \pm \frac{ {\bf F}_i ({\bf U}) }{\alpha} \in \Omega_\sigma, \quad \forall{\bf U} \in \Omega_\sigma,~\forall \alpha \ge \alpha_i,
\end{equation}
where $\alpha_i$ denotes a suitable upper bound of the wave speeds in the $x_i$-direction, and, in the RHD case, it can be taken as the speed of light $c=1$ for simplicity. 
We refer to \eqref{eq:LFproperty} as the {\em LF splitting property}. 
This property is valid for the admissible state set $\mathcal G$ or ${\mathcal G}_1$ and played an important role in constructing bound-preserving schemes for the RHD; see \cite{WuTang2015,QinShu2016,WuTang2017ApJS}. However, unfortunately, the property \eqref{eq:LFproperty} does {\em not} hold in general for the invariant region $\Omega_\sigma$ as the entropy principle $S({\bf U})\ge \sigma$ is included.

Since \eqref{eq:LFproperty} does not hold, we would like to look for some alternative properties that are valid but weaker than \eqref{eq:LFproperty}. 
By considering the convex combination of some LF splitting terms, 
we achieve the {\em generalized LF (gLF) splitting properties}, whose derivations are highly nontrivial. 
Built on a technical inequality \eqref{eq:keyINEQ} constructed in Section \ref{sec:ineq}, 
the gLF splitting properties are presented in Section \ref{sec:gLxF}.

\subsubsection{A constructive inequality}\label{sec:ineq}

We first construct an important inequality \eqref{eq:keyINEQ}, which will be the key to 
establishing the gLF splitting properties.

\begin{theorem}\label{lem:LLFsplitIEQ}
	If ${\bf U} \in \Omega_\sigma$, then for any ${\bm v}_*=( v_{1,*},\dots,v_{d,*} )^\top \in \mathbb B_1({\bf 0})$, any $\rho_* \in \mathbb R^+$, and any $\theta \in [-1,1]$, it holds
	\begin{equation}\label{eq:keyINEQ}
		\varphi_\sigma\Big( {\bf U} + \theta  {\bf F}_i ({\bf U}); {\bm v}_*, \rho_*   \Big) + \theta {\rm e}^\sigma v_{i,*} \rho_*^\Gamma  \ge 0,
	\end{equation}
	where $i \in \{1,\dots,d\}$, and the function $\varphi_\sigma$ is defined in \eqref{varphi_def}.
\end{theorem}

\begin{proof}
	Due to the relativistic effects, the flux ${\bf F}_i ({\bf U})$ also cannot be explicitly formulated in terms of ${\bf U}$. Therefore, we have to work on the corresponding primitive quantities $\{ \rho, {\bm v}, p\}$  of ${\bf U}$, which satisfy $\rho>0$, $\|{\bm v} \|<1$ and $p \ge {\rm e}^\sigma \rho^\Gamma$ because ${\bf U} \in \Omega_{\sigma}$. 
	We observe that 
	\begin{equation*}
		\varphi_\sigma\Big( {\bf U} + \theta  {\bf F}_i ({\bf U}); {\bm v}_*, \rho_*   \Big) + \theta {\rm e}^\sigma v_{i,*} \rho_*^\Gamma = \widehat \Pi_1 + \widehat \Pi_2 p + {\rm e}^\sigma \widehat \Pi_3, 
	\end{equation*}
	with 
	\begin{align*}
		\widehat \Pi_1 & := \rho W^2 ( 1 + \theta v_{i} ) \left( 1 - {\bm v} \cdot {\bm v}_* - \sqrt{ 1 - \| {\bm v} \|^2 } 
		\sqrt{1 - \| {\bm v}_* \|^2 } \right), 
		\\
		\widehat \Pi_2 & :=    \frac{\Gamma}{\Gamma -1} (1+\theta v_i) \left( \frac{ 1 - {\bm v} \cdot {\bm v}_* }{ 1-\|{\bm v}\|^2 } \right)  - (1 + \theta v_{i,*}),
		\\
		\widehat \Pi_3 & := \rho_*^\Gamma \left(1+\theta v_{i,*} \right)- \frac{\Gamma }{\Gamma -1} (1+\theta v_i )  
		\rho W \rho_*^{\Gamma -1} \sqrt{ 1 - \|{\bm v}_*\|^2 }. 
	\end{align*}
	Using the Cauchy--Schwarz inequality gives
	$$ {\bm v} \cdot {\bm v}_* + \sqrt{ 1 - \| {\bm v} \|^2 } 
	\sqrt{1 - \| {\bm v}_* \|^2 } \le \sqrt{ \| {\bm v} \|^2 + ( 1 - \| {\bm v} \|^2 )  }  
	\sqrt{ \| {\bm v}_* \|^2 + (  1 - \| {\bm v}_* \|^2 ) } =1,
	$$
	which implies $\widehat \Pi_1 \ge 0$. It follows that 
	\begin{equation}\label{eq:proof34}
		\varphi_\sigma\Big( {\bf U} + \theta  {\bf F}_i ({\bf U}); {\bm v}_*, \rho_*   \Big) + \theta {\rm e}^\sigma v_{i,*} \rho_*^\Gamma \ge \widehat \Pi_2 p + {\rm e}^\sigma \widehat \Pi_3.
	\end{equation}
	Recalling that $\Gamma \in (1,2] $, $\|{\bm v}\|<1$ and $\|{\bm v}_*\|<1$, we obtain 
	\begin{equation}\label{def_tPi2}
		\widehat \Pi_2  \ge 2 \left( \frac{ 1+\theta v_i}{ 1-\|{\bm v}\|^2 } \right) (1- {\bm v} \cdot {\bm v}_* ) -   (1+\theta v_{i,*})=: \widetilde \Pi_2 > 0,
	\end{equation}
	where the positivity of $\widetilde \Pi_2$ is deduced by using the Cauchy--Schwarz inequality as follows: 
	\begin{align}\nonumber
		\widetilde \Pi_2 & = 2 \left( \frac{ 1+\theta v_i}{ 1-\|{\bm v}\|^2 } \right) - 1 - 2 \left( \frac{ 1+\theta v_i}{ 1-\|{\bm v}\|^2 } \right) \left[ 
		\left( v_i + \frac{\theta ( 1 - \| {\bm v} \|^2 ) }{ 2  (1+\theta v_i)  } \right) v_{i,*} + \sum_{j \neq i}  v_j v_{j,*}
		\right] 
		\\ \nonumber
		& \ge 2 \left( \frac{ 1+\theta v_i}{ 1-\|{\bm v}\|^2 } \right)- 1 - 2 \left( \frac{ 1+\theta v_i}{ 1-\|{\bm v}\|^2 } \right) 
		\left[ { \left( v_i + \frac{\theta ( 1 - \| {\bm v} \|^2 )  }{ 2  (1+\theta v_i) } \right)^2 +  \sum_{j\neq i} v_j^2 } \right]^{\frac12} \| {\bm v}_* \| 
		\\ \nonumber
		& = 2 \left( \frac{ 1+\theta v_i}{ 1-\|{\bm v}\|^2 } \right) - 1 - 2 \left( \frac{ 1+\theta v_i}{ 1-\|{\bm v}\|^2 } \right) \| {\bm v}_* \| 
		\left[ { 1 - \frac{  1 - \| {\bm v} \|^2   }{1+\theta v_i} + \frac{\theta^2 ( 1 - \| {\bm v} \|^2 )^2  }{ 4 (1+\theta v_i)^2 } } \right]^{\frac12}
		\\ \nonumber
		& \ge 2 \left( \frac{ 1+\theta v_i}{ 1-\|{\bm v}\|^2 } \right) - 1 - 2  \left( \frac{ 1+\theta v_i}{ 1-\|{\bm v}\|^2 } \right) \| {\bm v}_* \| 
		\left[ \left( 1- \frac{ 1 - \| {\bm v} \|^2 }{2(1+\theta v_i)} \right)^2  \right]^{\frac12}
		\\ \label{def_tPi2:posi}
		& = \left[ 2 \left( \frac{ 1+\theta v_i}{ 1-\|{\bm v}\|^2 } \right) - 1 \right] ( 1 - \| {\bm v}_*\| ) 
		\ge \left[ 2 \left( \frac{ 1 - \| {\bm v} \|}{ 1-\|{\bm v}\|^2 } \right) - 1 \right] ( 1 - \| {\bm v}_*\| ) 
		> 0.
	\end{align}
	Combining $\widehat \Pi_2 >0 $ and $p\ge {\rm e}^\sigma \rho^\Gamma$, we then derive from \eqref{eq:proof34} that 
	\begin{equation}\label{eq:proof345}
		\varphi_\sigma\Big( {\bf U} + \theta  {\bf F}_i ({\bf U}); {\bm v}_*, \rho_*   \Big) + \theta {\rm e}^\sigma v_{i,*} \rho_*^\Gamma \ge \widehat \Pi_2 {\rm e}^\sigma \rho^\Gamma + {\rm e}^\sigma  \widehat \Pi_3 = {\rm e}^\sigma \rho^\Gamma \widehat \Pi_4,
	\end{equation}
	with $\widehat \Pi_4 := \widehat \Pi_2  + \rho^{-\Gamma} \widehat \Pi_3 =\hat \phi_\sigma 
	\big(  \frac{\rho}{\rho_*} ; {\bm v}, {\bm v}_* \big)$, and 
	$$
	\hat \phi_\sigma  (  x ; {\bm v}, {\bm v}_* ) := \widehat \Pi_2 + 
	x^{-\Gamma} \left(1+\theta v_{i,*} \right)-  x^{-\Gamma +1}  \left( \frac{\Gamma }{\Gamma -1}  (1+\theta v_i )  
	\frac{ \sqrt{ 1 - \|{\bm v}_*\|^2 } }{ \sqrt{ 1 - \| {\bm v} \|^2 } } \right). 
	$$
	The subsequent task is to show that $\widehat \Pi_4$ is always nonnegative. Define
	$$
	\eta_s  := 
	\frac{ 1+\theta v_i}{ 1 + \theta v_{i,*} } \left( \frac{ 1 - {\bm v} \cdot {\bm v}_* }{ 1-\|{\bm v}\|^2 } \right) - 1, \qquad 
	x_s := \frac{ \left(1+\theta v_{i,*} \right) \sqrt{ 1- \| {\bm v} \|^2 } }{ (1+\theta v_i ) 
		\sqrt{ 1- \| {\bm v}_* \|^2 } } .$$ 
	By studying the derivative of $\hat \phi_\sigma$ with respect to $x$, we observe that the function $\hat \phi_\sigma $ is strictly decreasing 
	on the interval $(0,  x_s ]$ and  
	strictly increasing on $[ x_s  , +\infty )$. We therefore have 
	\begin{equation}\label{keyd123}
		\begin{aligned}
			\widehat \Pi_4 &\ge \min_{x \in \mathbb R^+} \hat \phi_\sigma  (  x ; {\bm v}, {\bm v}_* ) = 
			\hat \phi_\sigma  (  x_s ; {\bm v}, {\bm v}_* ) 
			\\
			&= \widehat \Pi_2 - \frac{1}{\Gamma -1} x_s^{-\Gamma} \left(1+\theta v_{i,*} \right)
			= \frac{ 1+\theta v_{i,*} }{\Gamma -1} \Big(  \eta_s \Gamma + 1  -  x_s^{-\Gamma} \Big).
		\end{aligned}
	\end{equation}
	Using the formulation of $\widetilde \Pi_2$ defined in \eqref{def_tPi2}, we  reformulate $\eta_s$ and observe that   
	$$
	\eta_s = \frac12 \left( \frac{ \widetilde \Pi_2 }{ 1+ \theta v_{i,*}   } + 1 \right) -1 
	=  \frac{ \widetilde \Pi_2} { 2 (1+ \theta v_{i,*}) } - \frac12 > -\frac12,
	$$
	where the last step follows from the positivity of $\widetilde \Pi_2$, which has been proven in \eqref{def_tPi2:posi}.  Thanks to Lemma \ref{app:lem0}, we obtain $ (\eta_s \Gamma + 1 )^{1/\Gamma} \ge (2 \eta_s + 1 )^{1/2}$, or equivalently, 
	$$
	\eta_s \Gamma + 1 \ge (2 \eta_s + 1 )^{ \frac{\Gamma}2 },  
	$$ 
	which, along with \eqref{keyd123}, imply 
	\begin{equation}\label{estPi4}
		\widehat \Pi_4 \ge \frac{ 1+\theta v_{i,*} }{\Gamma -1} \left[ 
		\big(2 \eta_s + 1 \big)^{ \frac{\Gamma}2 }
		-  \left( x_s^{-2} \right)^{ \frac{\Gamma}2 }
		\right].
	\end{equation}
	Next, we would like to show that $2 \eta_s + 1 \ge x_s^{-2} $. Define 
	$a_1:=\theta v_i$, $a_2:=\sqrt{ \|{\bm v} \|^2 - \theta^2 v_i^2 }$, $b_1:=\theta v_{i,*}$, and $b_2:= \sqrt{ \|{\bm v}_* \|^2 - \theta^2 v_{i,*}^2 }$. 
	By an elementary inequality\footnote{Subtracting the right-hand side terms from the left-hand side terms leads to  $-(a_2-b_2-a_1 b_2+a_2 b_1)^2\le 0$.}
	$$
	( 1+a_1 )^2 (1-b_1^2-b_2^2) + ( 1+b_1 )^2( 1-a_1^2-a_2^2 ) \le 2 (1-a_1 b_1-a_2 b_2 )(1+a_1)(1+b_1),
	$$
	we get 
	\begin{equation}\label{key566}
		x_s^{-2} = \frac{ ( 1+a_1 )^2 (1-b_1^2-b_2^2) }{ ( 1+b_1 )^2( 1-a_1^2-a_2^2 ) }
		\le 2 \frac{  (1-a_1 b_1-a_2 b_2)(1+a_1) }{ ( 1+b_1 )( 1-a_1^2-a_2^2 ) } - 1 = 2 \widetilde \eta_s +1
	\end{equation}
	with $\widetilde \eta_s:= \frac{ 1+\theta v_i}{ 1 + \theta v_{i,*} } \left( \frac{ 1 - a_1 b_1 - a_2 b_2 }{ 1-\|{\bm v}\|^2 } \right) - 1$. With the aid of the Cauchy--Schwarz inequality, we obtain  
	\begin{align*}
		{\bm v} \cdot {\bm v}_* & =  \theta^2 v_i  v_{i,*} + \left( v_i \sqrt{1-\theta^2} \right) \left( v_{i,*} \sqrt{1-\theta^2} \right)  + \sum_{j\neq i} v_j v_{j,*} 
		\\
		& \le \theta^2 v_i  v_{i,*} +  \left({ \left( v_i \sqrt{1-\theta^2} \right)^2 + \sum_{j\neq i} v_j^2} \right)^{\frac12}
		\left( { \left( v_{i,*} \sqrt{1-\theta^2} \right)^2 + \sum_{j\neq i} v_{j,*}^2} \right)^{\frac12}
		= a_1 b_1 + a_2 b_2,
	\end{align*}
	which implies $\eta_s \ge \widetilde \eta_s$. It then follows from \eqref{key566} that 
	$ x_s^{-2} \le 2 \eta_s + 1$, which together with \eqref{estPi4} imply $\widehat \Pi_4 \ge 0$. Then by using 
	\eqref{eq:proof345}, we conclude 
	$\varphi_\sigma\Big( {\bf U} + \theta  {\bf F}_i ({\bf U}); {\bm v}_*, \rho_*   \Big) + \theta {\rm e}^\sigma v_{i,*} \rho_*^\Gamma \ge 
	{\rm e}^\sigma \rho^\Gamma \widehat \Pi_4 \ge 0$. 
	Hence, the inequality \eqref{eq:keyINEQ} holds. 
	The proof is completed.
\end{proof}

\begin{remark}
	It is worth emphasizing the importance of the last term ($\theta {\rm e}^\sigma v_{i,*} \rho_*^\Gamma$) at the left-hand side of \eqref{eq:keyINEQ}.
	This term is very technical, necessary and crucial in deriving
	the gLF splitting properties. 
	The value of this term is not always positive or negative. However, without this term, the inequality \eqref{eq:keyINEQ} does not hold.  
	More importantly, 
	this term can be canceled out dexterously in IRP analysis; see the proofs of gLF splitting properties in 
	the theorems in Section \ref{sec:gLxF}.
\end{remark}

\subsubsection{Derivation of gLF splitting properties}\label{sec:gLxF}

We first present the one-dimensional version.

\begin{theorem}[{\bf 1D gLF splitting}]\label{theo:MHD:LLFsplit1D}
	If $ \hat{\bf U}= (\hat D, \hat{\bm m}^\top, \hat E)^{\top} \in \Omega_{\sigma} $ and $ \check{\bf U}=(\check D, \check{\bm m}^\top, \check{E})^{\top} \in \Omega_{\sigma} $, 
	then for any  $\alpha \ge c=1 $ and any given $i\in \{1,\dots,d\}$, it holds
	\begin{equation}\label{eq:MHD:LLFsplit1D}
		{\bf G}_{i,\alpha}( \hat{\bf U}, \check{\bf U} )
		:=\frac{1}{2} \bigg( \hat{\bf U} - \frac{ {\bf F}_i(\hat{\bf U})}{\alpha}
		+
		\check{\bf U} + \frac{ {\bf F}_i(\check{\bf U})}{\alpha} \bigg)
		\in \Omega_{\sigma}. 
	\end{equation}
\end{theorem}

\begin{proof}
	The first component of ${\bf G}_{i,\alpha}$  
	equals $\frac12 \big( \hat D \big(1-\frac{\hat v_i}{\alpha}\big) + \check D \big(1+\frac{\check v_i}{\alpha}\big) \big) > 0$.
	With the help of Theorem \ref{lem:LLFsplitIEQ}, we obtain, for any ${\bm v}_* \in \mathbb B_1({\bf 0})$ and any $\rho_* \in \mathbb R^+$, that 
	\begin{equation*}
		2 \varphi_\sigma\big( {\bf G}_{i,\alpha} ; {\bm v}_*, \rho_*   \big)  = 
		\varphi_\sigma\left( {\bf U} - \frac{{\bf F}_i ({\bf U})}{\alpha}  ; {\bm v}_*, \rho_*   \right) - {\rm e}^\sigma \frac{ v_{i,*} \rho_*^\Gamma}{\alpha}  
		+ \varphi_\sigma\left( {\bf U} + \frac{{\bf F}_i ({\bf U})}{\alpha}; {\bm v}_*, \rho_*   \right) + {\rm e}^\sigma \frac{ v_{i,*} \rho_*^\Gamma}{\alpha}  
		\ge 0,
	\end{equation*}
	where we have used the linearity of $\varphi_\sigma ( {\bf U}; {\bm v}_*, \rho_* )$ with respect to ${\bf U}$. 
	Then, using Theorem \ref{thm:EqOmega} concludes that ${\bf G}_{i,\alpha} \in \Omega_\sigma^{(2)} = \Omega_\sigma$.
\end{proof}

\begin{remark}
	Another approach to show Theorem \ref{theo:MHD:LLFsplit1D} is based on the 
	assumption that the exact Riemann solver preserves the invariant domain, which is 
	reasonable but not yet proven for RHD.   	
	Our present analysis approach is very direct, does not rely on any assumption, and would motivate the further study of IRP schemes for some other complicated physical systems such as the MHD and RMHD equations.

\end{remark}

Next, we 
present the multidimensional gLF splitting property on a general polygonal or polyhedron cell. 
For any vector 
${\bm \xi}=(\xi_1,\cdots,\xi_d)^\top \in {\mathbb{R}}^d$, we define ${\bm \xi} \cdot {\bf F}({\bf U}) := \sum_{k=1}^d \xi_k{\bf F}_k({\bf U})$.

\begin{theorem}\label{lem:main}
	For $ 1\le j \le N$, let $s_j>0$ and the unit vector ${\bm \xi}^{(j)}=(\xi_1^{(j)},\dots,\xi_d^{(j)} )^\top$ satisfy 
	\begin{equation}\label{eq:sumUpsilon1}
		\sum_{j=1}^N s_{j} {\bm \xi}^{(j)} ={\bf 0}.
	\end{equation}
	Given admissible states      
	${\bf U}^{(ij)} \in \Omega_\sigma$, $1\le i \le Q$, 
	$1\le j \le N$, then 
	for any $ \alpha \ge c= 1 $, it holds 
	\begin{align}\label{eq:defOverlineU}
		\overline{\bf U}: =  \frac{1}{{\sum\limits_{j = 1}^{N} { s_j     } }}\sum\limits_{j = 1}^{N} 
		\sum_{i=1}^Q { { { s_j \omega_i \bigg(  {{\bf U}^{(ij)}  - \frac{1}{\alpha}
						{\bm \xi}^{(j)} \cdot {\bf F} ({\bf U}^{(ij)} ) 
					} \bigg)} } } \in \Omega_\sigma, 
	\end{align}
	where 
	the sum of all positive numbers
	$\left\{\omega _i \right\}_{i=1}^Q$ equals one. 
\end{theorem}

Before the proof, we would like to briefly  
explain the result in Theorem \ref{lem:main}, whose meaning will become more clear in 
the IRP analysis in Section \ref{sec:2Dschemes}.  
Let us consider a cell of the computational mesh, and assume it is a non-self-intersecting $d$-polytope with $N$ edges ($d=2$) or 
faces ($d=3$). The index $j$ 
on the variables in Theorem \ref{lem:main} represents 
the $j$th edge or face of the polytope, while
$i$ stands for the $i$th (quadrature) point on each edge or face, with $\omega_i$ denoting the 
associated quadrature weight at that point. 
Besides, $s_j$ and ${\bm \xi}^{(j)}$ respectively correspond to the $(d-1)$-dimensional Hausdorff measure and the unit 
outward normal vector of the $j$th edge or face. One can verify that 
the condition \eqref{eq:sumUpsilon1} holds naturally. 
In addition, 
${\bf U}^{(ij)}$ stands for the approximate values of  
${\bf U}$ 
at the $i$th quadrature point 
on the $j$th edge or face. 

\begin{proof}
	Let ${\bf Q}_{\bm \xi}^{(j)} \in \mathbb R^{d\times d}$ be a rotational matrix associated with the unit vector ${\bm \xi}^{(j)}$ and satisfying 	
	\begin{equation}\label{eq:DefQ}
		{\bf e}_1^\top {\bf Q}_{\bm \xi}^{(j)} 
		=  ({\bm \xi}^{(j)})^\top,
	\end{equation}
	where ${\bf e}_1=(1, {\bf 0}_{d-1}^\top)^\top$, and ${\bf 0}_{d-1}$ denotes the zero vector in $\mathbb R^{d-1}$. 
	It can be verified that the system \eqref{eq:RHD} satisfies the following rotational invariance property 
	\begin{equation}\label{eq:roto}
		{\bm \xi}^{(j)} \cdot \vec F(\vec U^{(ij)})  = {\bf Q}_j^{-1} {\bf F}_1 ( {\bf Q}_j {\bf U}^{(ij)}  ),
	\end{equation}
	where $
	{\bf Q}_j:= {\rm diag} \{1, {\bf Q}_{\bm \xi}^{(j)}, 1\}  
	$. Notice that the matrix ${\bf Q}_{\bm \xi}^{(j)}$ is orthogonal. 
	For any fixed $j$ and any ${\bm v}_*  \in \mathbb B_1({\bf 0})$, 
	define 
	$\widehat {\bm v}_* := {\bf Q}_{\bm \xi}^{(j)}  {\bm v}_*  \in \mathbb B_1({\bf 0})$. Utilizing \eqref{eq:DefQ} gives 
	\begin{equation}\label{eq:hatv1s}
		\widehat v_{1,*} = {\bf e}_1^\top {\bm v}_* = {\bf e}_1^\top  {\bf Q}_{\bm \xi}^{(j)} {\bm v}_*  
		= {\bm \xi}^{(j)} \cdot {\bm v}_* .
	\end{equation}
	For $\vec U^{(ij)} \in \Omega_\sigma$, with the aid of the first equivalent form $\Omega_\sigma^{(1)}$ in \eqref{eq:IRgen2}, one can verify that  
	$
	\widehat {\bf U}^{(ij)} := {\bf Q}_j {\bf U}^{(ij)} \in \Omega_\sigma^{(1)} = \Omega_\sigma. 
	$
	For any ${\bm v}_*  \in \mathbb B_1({\bf 0})$ and any $\rho_*>0$, we have 
	\begin{align}\nonumber
		& \quad \varphi_\sigma \left(  {{\bf U}^{(ij)}  - {\alpha}^{-1}
			{\bm \xi}^{(j)} \cdot {\bf F} ({\bf U}^{(ij)} ) ; {\bm v}_*,\rho_* 
		} \right) -\alpha^{-1} {\rm e}^\sigma \left( {\bm \xi}^{(j)} \cdot {\bm v}_* \right) \rho_*^\Gamma
		\\  \nonumber
		&   =  \varphi_\sigma \left( {\bf Q}_j^{-1} \widehat{\bf U}^{(ij)}  - {\alpha}^{-1}
		{\bf Q}_j^{-1} {\bf F}_1 ( \widehat{\bf U}^{(ij)}  ) ; ({\bf Q}_{\bm \xi}^{(j)})^{-1} \widehat {\bm v}_*  ,\rho_* 
		\right) -{\alpha}^{-1} {\rm e}^\sigma \widehat v_{1,*} \rho_*^\Gamma
		\\ \label{eq:wkl3da}
		&  = \varphi_\sigma \left(  \widehat{\bf U}^{(ij)}  - {\alpha}^{-1}
		{\bf F}_1 ( \widehat{\bf U}^{(ij)}  ) ;  \widehat {\bm v}_*  ,\rho_* 
		\right) -{\alpha}^{-1} {\rm e}^\sigma \widehat v_{1,*} \rho_*^\Gamma \ge 0,
	\end{align}
	where we have used \eqref{eq:roto}--\eqref{eq:hatv1s} 
	in the first equality and the orthogonality of ${\bf Q}_{\bm \xi}^{(j)}$ in the second equality, and the inequality follows from Theorem \ref{lem:LLFsplitIEQ} for $\widehat {\bf U}^{(ij)} \in \Omega_\sigma$,
	$\widehat {\bm v}_* \in \mathbb B_1({\bf 0}) $, $\rho_*\in \mathbb R^+$ and $0<\alpha^{-1}\le 1$. It then follows that 
	\begin{align*}
		\left( \sum\limits_{j = 1}^{N} { s_j     } \right) \varphi_\sigma\big( \overline {\bf U} ; {\bm v}_*, \rho_*   \big) & 
		= \sum\limits_{j = 1}^{N} 
		\sum_{i=1}^Q s_j \omega_i \varphi_\sigma \left(   { { {    {{\bf U}^{(ij)}  - {\alpha}^{-1}
						{\bm \xi}^{(j)} \cdot {\bf F} ({\bf U}^{(ij)} ) 
		} } } }; {\bm v}_*, \rho_* \right) 
		\\
		& \ge \alpha^{-1} {\rm e}^\sigma \sum\limits_{j = 1}^{N} 
		\sum_{i=1}^Q s_j  \omega_i \left( {\bm \xi}^{(j)} \cdot {\bm v}_* \right) \rho_*^\Gamma 
		=  \alpha^{-1} {\rm e}^\sigma \rho_*^\Gamma \left( \sum\limits_{j = 1}^{N}  s_j {\bm \xi}^{(j)} \right) \cdot {\bm v}_*  = 0,
	\end{align*}
	where we have used the linearity of $\varphi_\sigma ( {\bf U}; {\bm v}_*, \rho_* )$ with respect to ${\bf U}$, 
	the inequality \eqref{eq:wkl3da}, and the condition \eqref{eq:sumUpsilon1}. 
	Therefore, $\varphi_\sigma ( \overline {\bf U} ; {\bm v}_*, \rho_*  )\ge 0$ for any ${\bm v}_*  \in \mathbb B_1({\bf 0})$ and any $\rho_*>0$. Note that the first component of $\overline {\bf U}$ equals to 
	$$
	\frac{1}{{\sum\limits_{j = 1}^{N} { s_j     } }}\sum\limits_{j = 1}^{N} 
	\sum_{i=1}^Q { { { s_j \omega_i D^{(ij)} \bigg(  1 - \frac{1}{\alpha}
				{\bm \xi}^{(j)} \cdot {\bm v}^{(ij)}  
			} \bigg)} } \ge \frac{1}{{\sum\limits_{j = 1}^{N} { s_j     } }}\sum\limits_{j = 1}^{N} 
	\sum_{i=1}^Q { { { s_j \omega_i D^{(ij)} \bigg(  1 - \frac{\| {\bm v}^{(ij)} \|}{\alpha}  
			} \bigg)} }  > 0. 
	$$
	Hence, $\overline {\bf U} \in \Omega_\sigma^{(2)}$. Thanks to Theorem \ref{thm:EqOmega}, we have $\overline {\bf U} \in\Omega_\sigma$. The proof is completed.   
\end{proof}

As two special cases of Theorem \ref{lem:main}, the following corollaries show the gLF splitting properties  
on 2D and 3D Cartesian mesh cells.

\begin{corollary}\label{theo:RHD:LLFsplit2D}
	If $\bar{\bf U}^i$, $\tilde{\bf U}^{i}$, $\hat{\bf U}^i$,
	$\check{\bf U}^i
	\in \Omega_\sigma$ for $i=1,\cdots,Q$, 
	then for any 
	$ \alpha \ge c= 1$, 
	it holds
	\begin{equation} \label{eq:RHD:LLFsplit2D}
	\begin{split}
	\overline{\bf U}:= 
	\frac{1}{ 2\left( \frac{ 1 }{\Delta x} + \frac{1}{\Delta y} \right)}
	\sum\limits_{i=1}^{Q} { \omega _i  }
	\left(  
	\frac{ {\bf G}_{1,\alpha} ( \bar{\bf U}^i, \tilde{\bf U}^{i} )   }{\Delta x} 
	+
	\frac{ {\bf G}_{2,\alpha} ( \hat{\bf U}^i, \check{\bf U}^i ) }{\Delta y} 
	\right)
	\in \Omega_\sigma, 
	\end{split}
	\end{equation}
	where $\Delta x>0,\Delta y >0$, and the sum of the positive numbers
	$\left\{\omega _i \right\}_{i=1}^{Q}$ equals one. 
\end{corollary}

\begin{corollary}\label{theo:RHD:LLFsplit3D}
	If $\bar{\bf U}^i$, $\tilde{\bf U}^{i}$, $\hat{\bf U}^i$,
	$\check{\bf U}^i$, $\acute{\bf U}^i$, $\grave{\bf U}^i
	\in \Omega_\sigma$ for $i=1,\cdots,Q$, 
	then for any 
	$ \alpha \ge 1$, 
	it holds
	\begin{equation} \label{eq:RHD:LLFsplit3D}
	\begin{split}
	\overline{\bf U}:= 
	\frac{1}{ 2\left( \frac{ 1 }{\Delta x} + \frac{1}{\Delta y} + \frac1{\Delta z} \right)}
	\sum\limits_{i=1}^{Q} { \omega _i  }
	\left(  
	\frac{ {\bf G}_{1,\alpha} ( \bar{\bf U}^i, \tilde{\bf U}^{i} )   }{\Delta x} 
	+
	\frac{ {\bf G}_{2,\alpha} ( \hat{\bf U}^i, \check{\bf U}^i ) }{\Delta y} 
	+ 
	\frac{ {\bf G}_{3,\alpha} ( \acute{\bf U}^i, \grave{\bf U}^i ) }{\Delta z}
	\right)
	\in \Omega_\sigma, 
	\end{split}
	\end{equation}
	where $\Delta x>0,\Delta y >0,\Delta z>0$, and the sum of the positive numbers
	$\left\{\omega _i \right\}_{i=1}^{Q}$ equals one. 
\end{corollary}

\section{One-dimensional invariant-region-preserving schemes}\label{sec:1Dscheme}

This section applies the above theories to 
study the IRP schemes 
for 
the RHD 
system \eqref{eq:RHD} in one spatial dimension. 
To avoid confusing subscripts, we will use the symbol $x$ to represent the variable $x_1$ in \eqref{eq:RHD}.
Let $I_j=[x_{j-\frac{1}{2}},x_{j+\frac{1}{2}}]$ and 
$\cup_j I_j$ be a partition of the spatial domain $\Sigma$. 
Denote $\Delta x_j = x_{j+\frac{1}{2}} - x_{j-\frac{1}{2}}$.
Assume that the time interval is divided into the mesh $\{t_n\}_{n=0}^{N_t}$
with $t_0=0$ and the time step-size $\Delta t$ determined by certain CFL condition. 
Let  
$\overline {\bf U}_j^n $ denote the numerical cell-averaged approximation of the exact solution ${\bf U}(x,t)$ over $I_j$ at $t=t_n$. 

Let ${\bf U}_0 (x):={\bf U}(x,0)$ be an initial solution and $S_0 := \min_{ x} S ( {\bf U}_0 (x) ) $. We are interested in numerical schemes preserving $\overline {\bf U}_j^n \in \Omega_{S_0}$ for the 
system \eqref{eq:RHD}.

Note that the initial cell average $\overline {\bf U}_j^0 \in \Omega_{S_0}$ for all $j$, according to the following lemma. 

\begin{lemma}\label{lem:initialS} 
	Assume that ${\bf U}_0(x)$ is an (admissible) initial data of the RHD system \eqref{eq:RHD} on 
	the domain $\Sigma$, i.e., it satisfies ${\bf U}_0(x) \in {\mathcal G}$ for all $x \in \Sigma$. 
	For any $I \subseteq \Sigma$, we have 
	$$
	\overline {\bf U}_I := \frac{1}{ |I| } \int_I {\bf U}_0 ( x ) {\rm d} x  
	\in \Omega_{S_0},
	$$
	where $|I|= \int_I {\rm d} x >0$, and $S_0=\min\limits_{ x} S ( {\bf U}_0 (x) )$.

\end{lemma}

\begin{proof}
	Let $\overline {\bf U}_I =( \overline D_I, \overline{\bm m}_I^\top, \overline E_I )^\top$. It is evident that $\overline D_I>0$. 
	Define 
	$
	q({\bf U}):= E-\sqrt{ D^2 + \| {\bm m} \|^2 }, 
	$
	it is easy to show that $q({\bf U})$ is a concave function on $\mathbb R^{d+2}$, and 
	the second constraint in ${\mathcal G}_1$ is equivalent $q({\bf U}) > 0$. 
	According to Jensen's inequality, 
	$q(\overline {\bf U}_I) \ge \frac{1}{ |I| } \int_I q( {\bf U}_0({\bm x}) ) {\rm d} {\bm x} >0$. 
	Therefore, $\overline {\bf U}_I \in {\mathcal G}_1 = {\mathcal G}$. 
	
	Let ${\mathscr E}({\bf U}) = - D S({\bf U})$. According to Proposition \ref{prop:H}, ${\mathscr E}({\bf U})$ is a convex function of ${\bf U}$ on ${\mathcal G}$. 
	By following \cite[Lemma 2.2]{zhang2012minimum} and using  
	Jensen's inequality 
	$$
	{\mathscr E} \left(  \frac{1}{ |I| } \int_I {\bf U}_0 ( x ) {\rm d} x  \right) \le \frac{1}{ |I| } \int_I  {\mathscr E} ( {\bf U}_0 ( x ) ) {\rm d} x,
	$$
	one obtains   
	$S \left(\overline {\bf U}_I \right) \ge S_0$, which together with $\overline {\bf U}_I\in {\mathcal G}$, imply 
	$\overline {\bf U}_I 
	\in \Omega_{S_0}$. 
\end{proof}

\subsection{First-order scheme}
Consider the first-order LF type scheme
\begin{equation}\label{eq:1DMHD:LFscheme}
\overline {\bf U}_j^{n+1} = \overline {\bf U}_j^n - \frac{\Delta t}{\Delta { x}_j} \Big( \widehat {\bf F}_1 ( \overline {\bf U}_j^n ,\overline {\bf U}_{j+1}^n) - \widehat {\bf F}_1 ( \overline {\bf U}_{j-1}^n, \overline {\bf U}_j^n )  \Big) ,
\end{equation}
where the numerical flux $\widehat {\bf F}_1 (\cdot,\cdot) $ is defined by  
\begin{equation}\label{eq:LFflux1D} 
\widehat {\bf F}_1 ( {\bf U}^- , {\bf U}^+ ) = \frac{1}{2} \Big( {\bf F}_1 (  {\bf U}^- ) + {\bf F}_1 ( {\bf U}^+ ) - 
\alpha (  {\bf U}^+ -  {\bf U}^- ) \Big),
\end{equation}
and $\alpha$ denotes the numerical viscosity parameter, which can be taken as $\alpha=c=1$, a simple upper bound
of all wave speeds in the theory of special relativity.


\begin{lemma}\label{lem:1Dlocal}
If $\overline {\bf U}_{j-1}^n$, $\overline {\bf U}_j^n$ and $\overline {\bf U}_{j+1}^n$ all belong to $\Omega_\sigma$ for certain $\sigma$, then the state $\overline {\bf U}_j^{n+1}$, computed by the scheme \eqref{eq:1DMHD:LFscheme} under the CFL condition
\begin{equation}\label{eq:CFL:LF}
  \alpha \Delta t \le \Delta x_j,
\end{equation}
belongs to $\Omega_\sigma$. 
\end{lemma}

\begin{proof}
We rewrite the scheme \eqref{eq:1DMHD:LFscheme} in the following form
\begin{equation*}
\overline {\bf U}_j^{n+1} = (1-\lambda) \overline {\bf U}_j^n  + \lambda {\bf G}_{1,\alpha} 
( \overline {\bf U}_{j+1}^n, \overline {\bf U}_{j-1}^n ),
\end{equation*}
where ${\bf G}_{1,\alpha}$ is defined in \eqref{eq:MHD:LLFsplit1D}, and $\lambda := \alpha \Delta t/\Delta x_j \in (0,1]$. 
Thanks to Theorem \ref{theo:MHD:LLFsplit1D}, we have 
$${\bf G}_{1,\alpha} 
( \overline {\bf U}_{j+1}^n, \overline {\bf U}_{j-1}^n ) \in \Omega_\sigma,$$ 
which leads to 
$\overline {\bf U}_j^{n+1} \in \Omega_\sigma$ by the convexity of $\Omega_\sigma$ (Lemma \ref{lem:convexity}).  
\end{proof}

Lemma \ref{lem:1Dlocal} implies a discrete (local) minimum entropy principle of the scheme  \eqref{eq:1DMHD:LFscheme}.

\begin{theorem}\label{thm:1Dlocal}
	If $\overline {\bf U}_{j-1}^n$, $\overline {\bf U}_j^n$ and $\overline {\bf U}_{j+1}^n$ all belong to ${\mathcal G}$, then the state $\overline {\bf U}_j^{n+1}$, computed by the scheme \eqref{eq:1DMHD:LFscheme} under the CFL condition \eqref{eq:CFL:LF}, 
	belongs to $\mathcal G$ and satisfies 
	$$
	S(\overline {\bf U}_j^{n+1}) \ge \min \left\{  S (\overline {\bf U}_{j-1}^{n}), S (\overline {\bf U}_{j}^{n}), S (\overline {\bf U}_{j+1}^{n}) \right\}.
	$$
\end{theorem}

\begin{proof}
By Lemma \ref{lem:1Dlocal} with $\sigma = \min_{j-1\le k \le j+1} S (\overline {\bf U}_k^{n}) $, one directly draws the conclusion.
\end{proof}

The IRP property of the scheme  \eqref{eq:1DMHD:LFscheme} is shown in the following theorem.

\begin{theorem}
	Under the CFL condition \eqref{eq:CFL:LF}, the scheme \eqref{eq:1DMHD:LFscheme} always preserves 
	$ \overline {\bf U}_j^{n} \in \Omega_{S_0} $ for all $j$ and $n\ge 0$.  
\end{theorem}

\begin{proof}
We prove it by the mathematical induction for the time level number $n$. 
By Lemma \ref{lem:initialS}, we have $ \overline {\bf U}_j^{0} \in \Omega_{S_0} $
for all $j$, i.e., the conclusion holds for $n=0$. 
If assuming that $\overline {\bf U}_j^n\in \Omega_{S_0} $ for all $j$, then, by Lemma \ref{lem:1Dlocal} 
with $\sigma = S_0$, we have $\overline {\bf U}_j^{n+1}\in \Omega_{S_0} $ for all $j$. 
\end{proof}

\subsection{High-order schemes}

We now study the IRP high-order DG and finite volume schemes for 1D RHD equations \eqref{eq:RHD}. 
For the moment, we use the forward Euler method for time discretization, while high-order time discretization will be discussed later.
We consider the high-order finite volume schemes as well as the scheme satisfied by
the cell averages of a standard DG method, 
which have the following unified form
\begin{equation}\label{eq:1DMHD:Hcellaverage}
\overline {\bf U}_j^{n+1} = \overline {\bf U}_j^{n} - \frac{\Delta t}{\Delta x_j}
\Big(  \widehat {\bf F}_1 ( {\bf U}_{j+ \frac{1}{2}}^-, {\bf U}_{j+ \frac{1}{2}}^+ )
- \widehat {\bf F}_1 ( {\bf U}_{j- \frac{1}{2}}^-, {\bf U}_{j-\frac{1}{2}}^+)
\Big) ,
\end{equation}
where $\widehat {\bf F}_1 ( \cdot , \cdot )$ 
is taken as the LF flux in \eqref{eq:LFflux1D}, and 
\begin{equation}\label{eq:DG1Dvalues}
{\bf U}_{j + \frac{1}{2}}^\pm = \mathop {\lim }\limits_{ \epsilon \to 0^\pm }  {\bf U}_h \big( x_{j + \frac{1}{2}} + \epsilon \big).
\end{equation}
Here ${\bf U}_h(x)$ is a piecewise polynomial vector function of degree $k$, i.e., 
$$
{\bf U}_h \in {\mathbb V}_h^k := \left\{ {\bf u}=(u_1,\cdots,u_{d+2})^\top:~ u_\ell \big|_{I_j} \in {\mathbb P}^{k},~
\forall \ell,j   \right\},
$$
with ${\mathbb P}^{k}$ denoting the space of polynomials of degree up to $k$. 
Specifically, the function ${\bf U}_h(x)$ with $\int_{I_j} {\bf U}_h {\rm d} x = \overline {\bf U}_j^{n} $
is an approximation to ${\bf U}( x,t_n)$ within the cell $I_j$; it is either reconstructed in the finite volume methods from $\{\overline {\bf U}_j^n\}$ or directly evolved in the DG methods. 
The evolution equations for the high-order ``moments'' of ${\bf U}_h(x)$ in the DG methods are omitted because here we are only concerned with the IRP property of the schemes.

\subsubsection{Theoretical analysis}

If the polynomial degree $k=0$, i.e., ${\bf U}_h(x)=\overline{\bf U}_j^{n}$, $\forall x \in I_j$, then the scheme \eqref{eq:1DMHD:Hcellaverage} reduces to the 
first-order scheme \eqref{eq:1DMHD:LFscheme}, which is IRP under the CFL condition \eqref{eq:CFL:LF}. 
When the polynomial degree $k\ge 1$, 
the solution $\overline{\bf U}_j^{n+1}$ of 
the high-order scheme \eqref{eq:1DMHD:Hcellaverage} 
does not always belong to $\Omega_{S_0}$ even if $\overline {\bf U}_j^{n} \in \Omega_{S_0}$ for all $j$. 
In the following theorem, we present a satisfiable condition for achieving the provably IRP property of the scheme \eqref{eq:1DMHD:Hcellaverage} when $k\ge 1$. 

Let $\{ \widehat x_j^{(\mu)} \}_{\mu=1}^{ {L}}$ be the ${L}$-point Gauss--Lobatto quadrature nodes in the interval $I_j$,
and the associated weights are denoted by $\{\widehat \omega_\mu\}_{\mu=1} ^{L}$ with $\sum_{\mu=1}^{L} \widehat\omega_\mu = 1$.
We require $2{L}-3\ge k$ such that the algebraic precision of the quadrature is at least $k$, for example, one can particularly take ${L}=\lceil \frac{k+3}{2} \rceil$.

\begin{theorem} \label{thm:PP:1DRHD}
	If the piecewise polynomial vector function ${\bf U}_h$ satisfies 
	\begin{equation}
	\\ \label{eq:1DDG:con2}
	 {\bf U}_h ( \widehat x_j^{(\mu)} ) \in \Omega_{S_0}, \quad \forall \mu \in \{1,2,\cdots,{L}\}, ~\forall j,
	\end{equation}
	then, 	 under the CFL condition
	\begin{equation}\label{eq:CFL:1DMHD}
	 \frac{\alpha  \Delta t}{ \Delta x_j} \le \widehat \omega_1=\frac{1}{ {L} ( {L}-1 ) }, 
	\end{equation}
	the solution $\overline {\bf U}_j^{n+1}$, computed by the high-order scheme \eqref{eq:1DMHD:Hcellaverage}, 
	belongs to $\Omega_{S_0}$ for all $j$.
\end{theorem}

\begin{proof}
		The exactness of the $L$-point Gauss--Lobatto quadrature rule for the polynomials of degree $k$ implies
	$$
	\overline{\bf U}_j^n = \frac{1}{\Delta x_j} \int_{I_j} {\bf U}_h ({x}) {\rm d} x = \sum \limits_{\mu=1}^{L} \widehat \omega_\mu {\bf U}_h (\widehat { x}_j^{(\mu)} ).
	$$
	Noting $\widehat \omega_1 = \widehat \omega_{L}$ and $\widehat x_j^{1,{L}}={\tt x}_{j\mp\frac12}$, we can then rewrite the scheme \eqref{eq:1DMHD:Hcellaverage} into the convex combination form
	\begin{align} \label{eq:1DMHD:convexsplit}
	\overline{\bf U}_j^{n+1} =
	\sum \limits_{\mu=2}^{{L}-1} \widehat \omega_\mu {\bf U}_h ( \widehat {\tt x}_j^{(\mu)} )  
	+  (\widehat \omega_1-\lambda) \left( {\bf U}_{j-\frac{1}{2}}^+ +   {\bf U}_{j+\frac{1}{2}}^- \right)
	+ \lambda {\Xi}_- + \lambda  {\Xi}_+,
	\end{align}
	where $\lambda = \alpha \Delta t_n/\Delta x \in (0,\widehat \omega_1 ]$, and
	\begin{align*}
	& {\Xi}_\pm = \frac{1}{2} \left( {\bf U}_{j+\frac{1}{2}}^\pm - \frac{ {\bf F}_1 ( {\bf U}_{j+\frac{1}{2}}^\pm ) } {  \alpha }
	+ {\bf U}_{j-\frac{1}{2}}^\pm +  \frac{ {\bf F}_1 ( {\bf U}_{j-\frac{1}{2}}^\pm ) } {  \alpha }    \right).
	\end{align*}
	According to the gLF splitting property in Theorem \ref{theo:MHD:LLFsplit1D} and ${\bf U}_{j+\frac{1}{2}}^\pm \in \Omega_{S_0}$ by \eqref{eq:1DDG:con2}, 
	we obtain ${\Xi}_\pm \in \Omega_{S_0}$. 
	Using the convexity of $\Omega_{S_0}$ (Lemma \ref{lem:convexity}), 
	we conclude $\overline{\bf U}_j^{n+1} \in \Omega_{S_0}$ from \eqref{eq:1DMHD:convexsplit}.
\end{proof}

\subsubsection{Invariant-region-preserving limiter}\label{sec:limiter}

In general, the high-order scheme \eqref{eq:1DMHD:Hcellaverage} 
does not meet the condition 
\eqref{eq:1DDG:con2} automatically. 
Now, we design a simple limiter to effectively enforce the condition 
\eqref{eq:1DDG:con2}, without losing high-order accuracy and conservation. 
Our limiter is motivated by the existing bound-preserving and physical-constraint-preserving limiters  
 (cf.~\cite{zhang2010b,zhang2012minimum,WuTang2015,QinShu2016,WuTang2017ApJS,Wu2017}).

Before presenting the limiter, we define 
\begin{equation}\label{eq:Defq}
q({\bf U}) := E - \sqrt{D^2 + \|{\bm m}\|^2},
\end{equation}
then the second constraint in \eqref{eq:IRgen2} becomes $q({\bf U})>0$. 
As observed in \cite{WuTang2015}, the function $q({\bf U})$ is strictly concave. 
This function will play an important role, similar to the pressure function in the non-relativistic case, in our limiter for RHD. (Unlike the non-relativistic case, the pressure function $p({\bf U})$ is not concave in the RHD case \cite{WuTang2015}.)

Denote $\mathbb X_j:=\{ \widehat x_j^{(\mu)} \}_{\mu=1}^{L}$ and 
\begin{align*}
&\overline{\mathbb V}_h^k:= \left\{ {\bf u} \in {\mathbb V}_h^k:~~\frac{1}{\Delta x_j} \int_{I_j} {\bf u} (x) {\rm d}x \in \Omega_{S_0},~ \forall j \right\},
\\ 
& \widetilde {\mathbb V}_h^k:= \left\{ {\bf u} \in {\mathbb V}_h^k:~~{\bf u}\big|_{I_j} (x) \in \Omega_{S_0},~\forall x \in \mathbb X_j,~\forall j \right\}. 
\end{align*}
For any $ {\bf U}_h \in  \overline{\mathbb V}_h^k$ with ${\bf U}_h \big|_{I_j} =: {\bf U}_j (x) = 
\big( D_j(x), m_j(x), E_j(x) \big)^\top$, we define 
the IRP limiting operator ${\Pi}_h: \overline{\mathbb V}_h^k \to \widetilde {\mathbb V}_h^k$ by 
\begin{equation}\label{eq:limiter}
{\Pi}_h {\bf U}_h 
\big|_{I_j} =   \widetilde{ \bf  U }_j(x), \quad  \forall j,
\end{equation}
with the limited polynomial vector function $\widetilde { \bf  U }_j(x)$ constructed via the following three steps.

\begin{description}
	\item[Step (\romannumeral1):] First, modify the density to enforce its positivity via  
\begin{equation}\label{eq:IRP1}
	\widehat D_j(x) = \theta_1 \left( D_j( x) - \overline D_j^n \right) + \overline D_j^n,\qquad \theta_1 := \min\left\{1, \left| \frac{ \overline D_j^n - \varepsilon_1 }{ \overline D_j^n - \min_{x \in \mathbb X_j} D_j(x)  } \right| \right\},
\end{equation} 
	where $\varepsilon_1$ is a small positive number as the desired lower bound for density, is introduced to avoid the effect of the round-off error, and can be taken as $\varepsilon_1= \min\{10^{-13},\overline D_j^n \}$.
	\item[Step (\romannumeral2):] Then, modify $\widehat {\bf U}_j ( x ) = 
	\big( \widehat D_j(x), m_j(x), E_j(x) \big)^\top$ to enforce the positivity of $q({\bf U})$ via  
\begin{equation}\label{eq:IRP2}
	\widebreve {\bf U}_j (x) = \theta_2 \left( \widehat {\bf U}_j(x) -\overline { \bf  U }_j^n \right) + \overline { \bf  U }_j^n, \qquad \theta_2:= \min \left\{ 1, \left|  \frac{ q( \overline {\bf U}_j) - \varepsilon_2 }{ q( \overline {\bf U}_j^n) - \min_{x \in \mathbb X_j} q( \widehat {\bf U}_j (x)) } \right| \right\},
\end{equation} 
	where $\varepsilon_2$ is a small positive number as the desired lower bound for $q({\bf U})$, is introduced to avoid the effect of the round-off error, and can be taken as $\varepsilon_2= \min\{10^{-13},q(\overline {\bf U}_j^n) \}$.
	\item[Step (\romannumeral3):] Finally, modify $\widebreve {\bf U}_j (x) $ to enforce the 
	entropy principle $S({\bf U})\ge S_0$ via
\begin{equation}\label{eq:IRP3}
\widetilde {\bf U}_j (x) = \theta_3 \left( \widebreve {\bf U}_j(x) -\overline { \bf  U }_j^n \right) + \overline { \bf  U }_j^n, \qquad \theta_3 := \min_{x \in \mathbb X_j} \tilde \theta (x),
\end{equation} 
where, for $ x \in  \{ x \in \mathbb X_j: S ( \widebreve {\bf U}_j (x)) \ge S_0 \}$, $\tilde \theta (x)=1$, and, for $ x \in  \{ x \in \mathbb X_j: S ( \widebreve {\bf U}_j (x)) < S_0 \}$, $\tilde \theta (x)$ is the (unique) solution to the nonlinear equation 
$$
S \Big( (1-\tilde \theta) \overline { \bf  U }_j^n + \tilde \theta \widebreve {\bf U}_j(x) \Big) = S_0, \qquad \tilde \theta \in [0,1).
$$
\end{description}

The limiter $ {\Pi}_h$ is a combination of 
the bound-preserving limiter \eqref{eq:IRP1}--\eqref{eq:IRP2} (cf.~\cite{QinShu2016}) and the {\em entropy limiter} \eqref{eq:IRP3}. 
According to the above definition of the limiter $ {\Pi}_h$ and the Jensen's inequality for the 
concave function $q({\bf U})$, we immediately obtain the following proposition:

\begin{proposition}\label{prop:limiter1D}
	For any ${\bf U}_h \in \overline {\mathbb V}_h^{k}$, one has $ {\Pi}_h {\bf U}_h \in \widetilde {\mathbb V}_h^{k}$.
\end{proposition}
Proposition \ref{prop:limiter1D} indicates that the 
limited solution \eqref{eq:limiter} satisfies the condition \eqref{eq:1DDG:con2}. 
	Note that such type of local scaling limiters keep the conservation
	$ \int_{I_j}  {\Pi}_h  ( {\bf u} ) {\rm d} x =   \int_{I_j}  {\bf u} {\rm d} x,~\forall {\bf u} \in \overline{\mathbb V}_h^k $,  
	and do not destroy the high-order accuracy; see \cite{zhang2010,zhang2010b,ZHANG2017301} for details.

\begin{remark}
The invariant region $\Omega_{S_0}$ can also be reformulated as  
\begin{equation*}
\Omega_{S_0}^{(1)} = \left\{ \vec U=(D,{\bm m}^\top,E)^{\top}\in \mathbb R^{d+2}:~D>0,~q({\bf U})>0,~
\widetilde q ( {\bf U} ) \ge 0  \right\},
\end{equation*}
where 
$
\widetilde q ({\bf U}) := D ( S({\bf U})-S_0 ), 
$ 
and we have used $D>0$ to reformulate the third constraint $S({\bf U})\ge S_0$ in \eqref{eq:IRgen2}. 
Thanks to Proposition \ref{prop:H}, the function $\widetilde q({\bf U})$ is strictly concave for ${\bf U} \in {\mathcal G}$. Motivated by this property and  \cite{JIANG2018}, we can also use another (simpler but possibly more restrictive) approach to enforce the entropy principle $S({\bf U})\ge S_0$ by modifying  $\widebreve {\bf U}_j (x) $ to 
\begin{equation*}
\widetilde {\bf U}_j (x) = \tilde  \theta_3 \left( \widebreve {\bf U}_j(x) -\overline { \bf  U }_j^n \right) + \overline { \bf  U }_j^n, \qquad \tilde \theta_3 := \min \left\{ 1, \left|  \frac{ \widetilde q( \overline {\bf U}_j^n)  }{ \widetilde q( \overline {\bf U}_j^n) - \min_{x \in \mathbb X_j} \widetilde q( \widebreve {\bf U}_j (x)) } \right| \right\}.
\end{equation*} 
\end{remark}

\begin{remark}
	Another way to define an IRP limiting operator is 
\begin{equation*}
\widetilde {\Pi}_h {\bf U}_h 
\big|_{I_j} = \theta_j ( {\bf U}_j(x) -\overline { \bf  U }_j^n ) + \overline { \bf  U }_j^n, \quad  \forall j,
\end{equation*}
with 
$
\theta_{j} = \min\left\{   1, \frac{ \overline D_j^n - \varepsilon_1 }{ \overline D_j^n - \min_{x \in \mathbb X_j} D_j(x) }, \frac{ q( \overline {\bf U}_j^n) - \varepsilon_2 }{ q( \overline {\bf U}_j^n) - \min_{x \in \mathbb X_j} q( {\bf U}_j (x)) }, \frac{ \widetilde q( \overline {\bf U}_j^n)  }{ \widetilde q( \overline {\bf U}_j^n) - \min_{x \in \mathbb X_j} \widetilde q( {\bf U}_j (x)) } \right\}.$ 
This simpler definition is motivated by the IRP limiter in \cite{JIANG2018} for the non-relativistic Euler system. It seems, however, unpractical for numerical implementation in our RHD case, because ${\bf U}_j (x)$ does not necessarily belong to $\mathcal G$  for all $x \in \mathbb X_j$ so that $\widetilde q( {\bf U}_j (x))$ is generally not well-defined in $\theta_j$. 
\end{remark}

With the aid of the IRP limiter $\Pi_h$ defined in \eqref{eq:limiter}, we modify the high-order scheme \eqref{eq:1DMHD:Hcellaverage} into 
\begin{equation}\label{eq:1DMHD:HcellaverageIRP}
\overline {\bf U}_j^{n+1} = \overline {\bf U}_j^{n} - \frac{\Delta t}{\Delta x_j}
\Big(  \widehat {\bf F}_1 ( \widetilde {\bf U}_{j+ \frac{1}{2}}^-, \widetilde {\bf U}_{j+ \frac{1}{2}}^+ )
- \widehat {\bf F}_1 (  \widetilde {\bf U}_{j- \frac{1}{2}}^-, \widetilde {\bf U}_{j-\frac{1}{2}}^+)
\Big)=: \overline {\bf U}_j^{n} + \Delta t {\bf L}_j ( {\Pi}_h {\bf U}_h ),
\end{equation}
where 
\begin{equation}\label{eq:DG1DvaluesIRP}
\widetilde {\bf U}_{j + \frac{1}{2}}^\pm = \mathop {\lim }\limits_{ \epsilon \to 0^\pm } {\Pi}_h {\bf U}_h \big( x_{j + \frac{1}{2}} + \epsilon \big).
\end{equation}
Based on Theorem \ref{thm:PP:1DRHD} and Proposition \ref{prop:limiter1D}, we know that the resulting scheme \eqref{eq:1DMHD:HcellaverageIRP} is IRP under the CFL condition \eqref{eq:CFL:1DMHD}.

The scheme \eqref{eq:1DMHD:HcellaverageIRP} is only first-order accurate
in time. To achieve high-order accurate IRP scheme in both time and space, one can replace the forward Euler
time discretization in \eqref{eq:1DMHD:HcellaverageIRP} with any high-order accurate strong-stability-preserving (SSP) methods \cite{GottliebShuTadmor2001}. For
example, the third-order accurate SSP Runge-Kutta (SSP-RK) method:  
\begin{equation}\label{eq:RK31D}
\begin{cases}
\overline {\bf U}_j^* = \overline {\bf U}_j^{n} + \Delta t {\bf L}_j ( {\Pi}_h {\bf U}_h ),
\\
\overline {\bf U}_j^{**} = \frac{3}{4} \overline {\bf U}_j^{n}  
+ \frac{1}{4} \Big( \overline {\bf U}_j^{*} + \Delta t {\bf L}_j ( {\Pi}_h {\bf U}_h^* ) \Big)
\\ 
\overline {\bf U}_j^{n+1} = \frac{1}{3} \overline {\bf U}_j^{n}  
+ \frac{2}{3} \Big( \overline {\bf U}_j^{**} + \Delta t {\bf L}_j ( {\Pi}_h {\bf U}_h^{**} ) \Big),
\end{cases}
\end{equation}
and the third-order accurate SSP multi-step (SSP-MS) method: 
\begin{equation}\label{eq:MS31D}
\overline {\bf U}_j^{n+1}
= \frac{16}{27} ( \overline {\bf U}_j^{n} + 3 \Delta t    {\bf L}_j ( {\Pi}_h {\bf U}_h^{n} ) ) 
+ \frac{11}{27} \left( \overline {\bf U}_j^{n-3} + \frac{12}{11} \Delta t    {\bf L}_j ( {\Pi}_h {\bf U}_h^{n-3} )  \right).
\end{equation}
Since SSP methods are formally convex combinations of the forward Euler method, 
 the resulting high-order
schemes \eqref{eq:RK31D} and \eqref{eq:MS31D} are still IRP, according to the convexity of $\Omega_{S_0}$ (Lemma \ref{lem:convexity}).

\section{Multidimensional invariant-region-preserving  schemes}\label{sec:2Dschemes}

In this section, we study IRP schemes for 
the two-dimensional (2D) RHD equations, 
keeping in mind that the proposed methods and analyses are extensible to 
the 3D case.    Assume that the physical domain $\Sigma$ in the 2D space is discretized by a mesh ${\mathcal T}_h$. 
In general, the mesh may be unstructured 
and consists of 
polygonal cells. 
We also partition the time interval into
a mesh $\{t_n\}_{n=0}^{N_t}$
with $t_0=0$ and the time step-size $\Delta t$ determined by certain CFL condition. 

We will use 
the 
capital letter $K$ to denote an arbitrary cell in ${\mathcal T}_h$. 
Let $\mathscr{E}_{K}^{j}$, $j=1,\cdots,N_K$,  denote the 
edges of $K$, and  $K_j$ be the adjacent cell which shares the edge $\mathscr{E}_{K}^{j}$ with $K$.  
We denote by ${\bm \xi}_K^{(j)}=\big(\xi_{1,K}^{(j)},\xi_{2,K}^{(j)} \big)$ the unit normal vector of $\mathscr{E}_K^j$ pointing from $K$ to $K_j$. 
The notations 
$|K|$ and $|{\mathscr E}_{K}^{j}|$ are used to denote the area of $K$ and the length of ${\mathscr E}_{K}^{j}$, respectively.

Let  
$\overline {\bf U}_K^n $ denote the numerical cell-averaged approximation of the exact solution ${\bf U}({\bm x},t)$ over $K$ at $t=t_n$. 
Define ${\bf U}_0 ({\bm x}):={\bf U}({\bm x},0)$ as the initial data and $S_0 := \min_{\bm x} S ( {\bf U}_0 ({\bm x}) ) $. 
According to Lemma \ref{lem:initialS}, the initial cell average $\overline {\bf U}_K^0$ always belongs to $\Omega_{S_0}$ for all $K\in {\mathcal T}_h$. 
We would like to seek numerical schemes maintaining $\overline {\bf U}_K^n \in \Omega_{S_0}$ for all $K\in {\mathcal T}_h$ and $n\ge 1$.

\subsection{First-order scheme}

Consider the 
following first-order scheme on the mesh ${\mathcal T}_h$ for the RHD equations \eqref{eq:RHD}: 
\begin{equation}\label{eq:2DLFscheme}
\overline {\bf U}_{K}^{n+1}
= \overline{\bf U}_{K}^{n} - \frac{\Delta t}{|K|}   \sum_{j=1}^{N_K}
\left|{\mathscr E}_{K}^{j}\right| \widehat{\bf F} \big(  \overline{\bf U}_K^n , \overline{\bf U}_{K_j}^n  ; {\bm \xi}^{(j)}_K \big),
\end{equation}
with the numerical flux $\widehat {\bf F}$ taken as the LF flux
\begin{equation}\label{eq:LFflux2D}
\widehat{\bf F} \left( {\bf U}^-, {\bf U}^+; {\bm \xi} \right) 
= \frac12 
\Big( 
 {\bm \xi} \cdot {\bf F} ({\bf U}^-)  + 
 {\bm \xi} \cdot {\bf F} ({\bf U}^+) 
-   \alpha ( {\bf U}^+ - {\bf U}_h^- )  \Big),
\end{equation}
where the numerical viscosity parameter $\alpha$ is chosen as the speed of light in vacuum $c=1$, which is a simple upper bound of all wave speeds in the theory of special relativity.

\begin{lemma}\label{lem:2Dlocal}
	If $\overline {\bf U}_{K}^n \in \Omega_\sigma$, $\overline {\bf U}_{K_j}^n \in \Omega_\sigma,$ 
	$1\le j \le N_K$, for certain $\sigma$, then the state $\overline {\bf U}_K^{n+1}$, computed by the scheme \eqref{eq:2DLFscheme} under the CFL condition
	\begin{equation}\label{eq:CFL:LF2D}
	  \frac{\alpha \Delta t} { 2 |K|} \sum_{j=1}^{N_K} \left|{\mathscr E}_{K}^{j}\right|  \le 1,
	\end{equation}
	belongs to $\Omega_\sigma$. 
\end{lemma}

\begin{proof}
	Substituting the numerical flux \eqref{eq:LFflux2D} into the scheme \eqref{eq:2DLFscheme} and then using the identity 
	$$ 
	\sum_{j=1}^{N_K} \left|{\mathscr E}_{K}^{j}\right| {\bm \xi}^{(j)}_K = {\bf 0},
	$$
	we rewrite the scheme \eqref{eq:2DLFscheme} in the following form
	\begin{equation*}
	\overline {\bf U}_K^{n+1} = (1-\lambda) \overline {\bf U}_K^n  + \lambda {\Xi},
	\end{equation*}
	where $\lambda := \frac{\alpha \Delta t} { 2 |K|} \sum_{j=1}^{N_K} \left|{\mathscr E}_{K}^{j}\right| \in (0,1]$ under the condition \eqref{eq:CFL:LF2D}, and 
	$$
	{\Xi} := \frac{1}{ \sum_{j=1}^{N_K}  \left|{\mathscr E}_{K}^{j}\right| } \
	\sum_{j=1}^{N_K}  \left|{\mathscr E}_{K}^{j}\right| \left( 
	\overline{\bf U}_{K_j}^n - \frac{1}{\alpha} {\bm \xi}^{(j)}_K \cdot {\bf F} \big( \overline{\bf U}_{K_j}^n \big)
	\right).
	$$
	Thanks to the gLF splitting property in Theorem \ref{lem:main}, we have ${\Xi} \in \Omega_\sigma$, which leads to 
	$\overline {\bf U}_K^{n+1} \in \Omega_\sigma$ by the convexity of $\Omega_\sigma$ (Lemma \ref{lem:convexity}).  
\end{proof}

Lemma \ref{lem:2Dlocal} implies a discrete (local) minimum entropy principle of the scheme  \eqref{eq:2DLFscheme}.

\begin{theorem}\label{thm:2Dlocal}
	If $\overline {\bf U}_{K}^n \in {\mathcal G}$, $\overline {\bf U}_{K_j}^n \in {\mathcal G},$ 
	$1\le j \le N_K$, then the state $\overline {\bf U}_K^{n+1}$, computed by the scheme \eqref{eq:2DLFscheme} under the CFL condition \eqref{eq:CFL:LF2D}, 
	belongs to $\mathcal G$ and satisfies 
	$$
	S(\overline {\bf U}_K^{n+1}) \ge \min \left\{  S (\overline {\bf U}_{K}^{n}), 
	\min \limits_{ 1\le j \le N_K } S (\overline {\bf U}_{K_j}^{n})
	 \right\}.
	$$
\end{theorem}

\begin{proof}
	By Lemma \ref{lem:2Dlocal} with $\sigma = \min \big\{  S (\overline {\bf U}_{K}^{n}), 
	\mathop{\min}\limits_{ 1\le j \le N_K } S (\overline {\bf U}_{K_j}^{n})
	\big\}$, we directly draw the conclusion.
\end{proof}

The IRP property of the scheme  \eqref{eq:2DLFscheme} is shown in the following theorem.

\begin{theorem}
	Under the CFL condition \eqref{eq:CFL:LF2D}, the scheme \eqref{eq:2DLFscheme} always preserves 
	$ \overline {\bf U}_K^{n} \in \Omega_{S_0} $ for all $K \in {\mathcal T}_h$ and $n\ge 0$.  
\end{theorem}

\begin{proof}
	The conclusion follows from Lemma \ref{lem:2Dlocal} 
	with $\sigma = S_0$ and the principle of mathematical induction for the time level number $n$.  
\end{proof}

\subsection{High-order schemes}

This subsection discusses the provably IRP high-order finite volume or DG schemes for the 2D RHD equations \eqref{eq:RHD}. We will focus
on the first-order forward Euler method for time discretization, and our analysis also
works for high-order explicit time discretization using the SSP methods, which are formed by convex combinations of the forward Euler method \cite{GottliebShuTadmor2001}.

To achieve $(k+1)$th-order accuracy in space, an piecewise polynomial vector function ${\bf U}_h ({\bm x})$ (i.e., ${\bf U}_h |_{K}$ is a polynomial vector of degree $k$ for all $K \in {\mathcal T}_h$) is also built, as approximation to
the exact solution ${\bf U}({\bm x},t_n)$. 
It is, either evolved in
the DG methods, or 
reconstructed in the finite volume methods from the cell averages $\{\overline {\bf U}_K^n: K \in {\mathcal T}_h \}$. Moreover, the cell average of ${\bf U}_h ({\bm x})$ over $K$ is equal to 
$\overline {\bf U}_K^n$. 
A high-order finite volume scheme as well as the scheme satisfied by
the cell averages of a standard DG method can then be written as 
\begin{equation}\label{eq:2Dcellaverage}
\overline {\bf U}_{K}^{n+1}  = \bar {\bf U}_{K}^{n}  - \frac{\Delta t}{|K|} 
\sum_{j=1}^{N_K}  |{\mathscr E}_K^j| \widehat {\bf F}_{{\mathscr E}_K^j},
\end{equation}
with 
\begin{align} \label{eq:2DNflux}
\widehat {\bf F}_{{\mathscr E}_K^j} & = \sum_{\nu=1}^Q
\omega_\nu
\widehat{\bf F} \left( {\bf U}_h^{{\rm int}(K)} ( {\bm x}_K^{(j\nu)} ), {\bf U}_h^{{\rm ext}(K)} ( {\bm x}_K^{(j\nu)} ); {\bm \xi}^{(j)}_{K} \right)  
\\ \nonumber
& \approx \frac{1}{|{\mathscr E}_K^j|} \int_{ {\mathscr E}_K^j }  \widehat{\bf F} \left( {\bf U}_h^{{\rm int}(K)} ( {\bm x} ), {\bf U}_h^{{\rm ext}(K)} ( {\bm x} ); {\bm \xi}^{(j)}_{K} \right) {\rm d} s.
\end{align}
Here the superscripts ``${\rm ext}(K)$'' and ``${\rm int}(K)$'' indicate that the corresponding limits of ${\bf U}_h (\bm x)$ at the cell edges are taken from the exterior and interior of $K$, respectively; the numerical flux $\widehat {\bf F}$ is taken as the LF flux defined in \eqref{eq:LFflux2D}; 
$\{{\bm x}_K^{(j\nu)}, \omega_\nu \}_{1\le \nu \le Q}$ denote the $Q$-point Gauss quadrature nodes and weights on 
${\mathscr E}_K^j$.

In the following theorem, we derive a satisfiable condition for achieving the provably IRP property of the scheme \eqref{eq:2Dcellaverage} when the polynomial degree $k\ge 1$. Assume that one can {\it exactly} decompose the cell average by certain 2D quadrature: 
\begin{equation}\label{eq:decomposition}
\overline {\bf U}_K^n = \frac{1}{|K|} \int_K {\bf U}_h ({\bm x}) {\rm d} {\bm x}
= 
\sum_{j=1}^{N_K} \sum_{\nu=1}^Q \varpi_{j\nu} {\bf U}_h^{{\rm int}(K)} ( {\bm x}_K^{(j\nu)} ) + \sum_{ \beta=1 }^{\widetilde Q} \widetilde \varpi_\beta   
{\bf U}_h^{{\rm int}(K)} ( \widetilde {\bm x}_K^{(\beta)} ),
\end{equation}
where $\{\widetilde {\bm x}_K^{(\beta)}\}$ are the (possible) involved quadrature points excluding 
$\{{\bm x}_K^{(j\nu)}\}$ in the cell $K$; 
$\{ \varpi_{j\nu} \}$ and $\{ \widetilde \varpi_\beta \}$ are positive weights 
satisfying $\sum_{j=1}^{N_K} \sum_{\nu=1}^Q \varpi_{j\nu} + \sum_{ \beta=1 }^{\widetilde Q} \widetilde \varpi_\beta =1$. 
Such a quadrature-based decomposition was first proposed by Zhang and Shu in \cite{zhang2010,zhang2010b} on rectangular cells by tensor products of Gauss and Gauss--Lobatto quadratures. It can also be 
designed on triangular cells and more general polygons, as demonstrated in, e.g.,  \cite{zhang2012maximum,LV2015715}. 
Define 
\begin{equation}\label{eq:2DSk}
\mathbb X_{K} := 
\left\{ {\bm x}_K^{(j\nu)} \right\}_{ 1\le j \le N_K, 1\le \nu \le Q } 
\bigcup \left\{ \widetilde {\bm x}_K^{(\beta)} \right\}_{ 1\le \beta \le \widetilde Q }.
\end{equation}

\begin{theorem} \label{thm:PP:2DRHD}
	If the piecewise polynomial vector function ${\bf U}_h$ satisfies 
	\begin{equation}
	\\ \label{eq:2DDG:con2}
	{\bf U}_h ( {\bm x} ) \in \Omega_{S_0}, \qquad 
	\forall  {\bm x}  \in {\mathbb X}_{K},~~\forall K \in {\mathcal T}_h, 
	\end{equation}
	then, 	 under the CFL condition
	\begin{equation}\label{eq:CFL:2DRHD}
	\alpha \Delta t \frac{| {\mathscr E}_K^j |}{|K|} 
	\le \min \limits_{1\le \nu \le Q} \frac{ \varpi_{j\nu}}{\omega_\nu},\qquad 1\le j \le N_K,~~\forall K\in{\mathcal T}_h,
	\end{equation}
	the solution $\overline {\bf U}_K^{n+1}$, computed by the high-order scheme \eqref{eq:2Dcellaverage}, 
	belongs to $\Omega_{S_0}$ for all $K \in {\mathcal T}_h$.
\end{theorem}

\begin{proof} 
Substituting the decomposition \eqref{eq:decomposition} and 
the numerical flux \eqref{eq:2DNflux} with \eqref{eq:LFflux2D} into \eqref{eq:2Dcellaverage}, 
we can rewrite the scheme \eqref{eq:2Dcellaverage} in the following convex combination form 
\begin{align}\nonumber
\overline {\bf U}_{K}^{n+1} &= 
\sum_{j=1}^{N_K} \sum_{\nu=1}^Q 
\left( 
\varpi_{j\nu} - \alpha \Delta t \omega_\nu \frac{| {\mathscr E}_K^j |}{|K|}  \right)  {\bf U}_h^{{\rm int}(K)} ( {\bm x}_K^{(j\nu)} ) 
\\ \label{eq:2Dproof}
& \quad + \sum_{ \beta=1 }^{\widetilde Q} \widetilde \varpi_\beta   
{\bf U}_h^{{\rm int}(K)} ( \widetilde {\bm x}_K^{(\beta)} ) 
+   \frac{ \alpha \Delta t }{2|K|} \left( \sum\limits_{j = 1}^{N} { | {\mathscr E}_K^j |     } \right) \Big( 
{\Xi}^{{\rm int}(K)} + {\Xi}^{{\rm ext}(K)}
\Big),
\end{align}
with 
\begin{align*}
& {\Xi}^{{\rm int}(K)} := \frac{1}{{\sum\limits_{j = 1}^{N_K} { | {\mathscr E}_K^j |     } }}\sum\limits_{j = 1}^{N_K} 
\sum_{\nu=1}^Q { { { | {\mathscr E}_K^j | \omega_\nu \left( {\bf U}_h^{{\rm int}(K)} ( {\bm x}_K^{(j\nu)} )  - \frac{1}{\alpha}
				{\bm \xi}^{(j)}_K \cdot {\bf F} \Big(  {\bf U}_h^{{\rm int}(K)} ( {\bm x}_K^{(j\nu)} )   \Big) 
			 \right)} } },
\\
& {\Xi}^{{\rm ext}(K)} := \frac{1}{{\sum\limits_{j = 1}^{N_K} { | {\mathscr E}_K^j |     } }}\sum\limits_{j = 1}^{N_K} 
\sum_{\nu=1}^Q { { { | {\mathscr E}_K^j | \omega_\nu \left( {\bf U}_h^{{\rm ext}(K)} ( {\bm x}_K^{(j\nu)} )  - \frac{1}{\alpha}
			{\bm \xi}^{(j)}_K \cdot {\bf F} \Big(  {\bf U}_h^{{\rm ext}(K)} ( {\bm x}_K^{(j\nu)} )   \Big) 
			\right)} } }.
\end{align*}
Thanks to the gLF splitting property in Theorem \ref{lem:main}, under the assumption \eqref{eq:2DDG:con2} we obtain ${\Xi}^{{\rm int}(K)} \in \Omega_{S_0}$ and ${\Xi}^{{\rm ext}(K)} \in \Omega_{S_0}$. 
Using the convexity of $\Omega_{S_0}$ (Lemma \ref{lem:convexity}), 
we conclude $\overline{\bf U}_K^{n+1} \in \Omega_{S_0}$ from the convex combination form \eqref{eq:2Dproof} under the condition \eqref{eq:CFL:2DRHD}. 
\end{proof}

Theorem \ref{thm:PP:2DRHD} provides a sufficient condition \eqref{eq:2DDG:con2} for the high-order scheme \eqref{eq:2Dcellaverage} to be IRP. 
The condition \eqref{eq:2DDG:con2}, which is not satisfied automatically in general, can again be enforced by a 
simple IRP limiting operator ${\Pi}_h$ similar to the 1D case; see Section \ref{sec:limiter} with the 1D point set ${\mathbb X}_j$ replaced by the 2D point set ${\mathbb X}_K$ \eqref{eq:2DSk} accordingly.  
With the IRP limiter applied to the approximation solution $\widetilde {\bf U}_h = {\Pi}_h {\bf U}_h$, the resulting scheme 
\begin{equation*}
\overline {\bf U}_{K}^{n+1}  = \bar {\bf U}_{K}^{n}  - \frac{\Delta t}{|K|} 
\sum_{j=1}^{N_K}  \sum_{\nu=1}^Q  |{\mathscr E}_K^j|
\omega_\nu
\widehat{\bf F} \left( \widetilde {\bf U}_h^{{\rm int}(K)} ( {\bm x}_K^{(j\nu)} ), \widetilde {\bf U}_h^{{\rm ext}(K)} ( {\bm x}_K^{(j\nu)} ); {\bm \xi}^{(j)}_{K} \right),
\end{equation*}
is IRP and also high-order accurate in space. As the 1D case, 
Theorem \ref{thm:PP:2DRHD} also remains valid if a high-order SSP time discretization \cite{GottliebShuTadmor2001} is used.

\begin{remark}
Assume that the mesh is rectangular with cells $\{[x_{i-1/2},x_{i+1/2}]\times [y_{\ell-1/2},y_{\ell+1/2}] \}$ 
and spatial step-sizes $\Delta x_i=x_{i+1/2}-x_{i-1/2}$ and $\Delta y_\ell=y_{\ell+1/2}-y_{\ell-1/2}$ in $x$- and $y$-directions respectively, where $(x,y)$ denotes the 2D spatial coordinate variables. 
Let ${\mathbb S}_i^x=\{ x_i^{(\mu)}  \}_{\mu=1}^Q$
and ${\mathbb S}_\ell^y=\{ y_\ell^{(\mu)}  \}_{\mu=1}^Q$ 
denote the $Q$-point Gauss quadrature nodes in the intervals 
$[x_{i-1/2},x_{i+1/2}]$ and 
$[y_{\ell-1/2},y_{\ell+1/2}]$ respectively. 
Let $\widehat{\mathbb S}_i^x=\{ \widehat x_i^{(\nu)}  \}_{\nu=1}^{L}$
and $\widehat {\mathbb S}_\ell^y=\{\widehat y_\ell^{(\nu)}  \}_{\nu=1}^{L}$ 
denote the $L$-point (${L} \ge \frac{k+3}2$) Gauss--Lobatto quadrature nodes in the intervals 
$[x_{i-1/2},x_{i+1/2}]$ and 
$[y_{\ell-1/2},y_{\ell+1/2}]$ respectively. 
For the cell $K=[x_{i-1/2},x_{i+1/2}]\times [y_{\ell-1/2},y_{\ell+1/2}]$, 
a suitable point set as $ {\mathbb X}_K$ in \eqref{eq:2DSk} is given by (cf.~\cite{zhang2010})
\begin{equation}\label{eq:RectS}
{\mathbb X}_K = \big(  \widehat{\mathbb S}_i^x \otimes 
{\mathbb S}_\ell^y \big) \cup \big(  {\mathbb S}_i^x \otimes 
\widehat{\mathbb S}_\ell^y \big),
\end{equation}
and the corresponding 2D quadrature \cite{zhang2010} satisfying \eqref{eq:decomposition} can be constructed as
\begin{equation} \label{eq:U2Dsplit}
\begin{split}
\frac{1}{|K|}\int_K u({\bf x}) d {\bf x}
&= \sum \limits_{\mu = 1}^{Q}   \frac{ \Delta x_i \widehat \omega_1 \omega_\mu }{ \Delta x_i + \Delta y_\ell }   \left(  
u\big( x_i^{(\mu)},y_{\ell-\frac12} \big) 
+ u\big( x_i^{(\mu)},y_{\ell+\frac12} \big) 
\right)
\\
&
+  \sum \limits_{\mu = 1}^{Q} \frac{ \Delta y_\ell \widehat \omega_1 \omega_\mu }{ \Delta x_i + \Delta y_\ell }  \left( u \big(x_{i-\frac12},y_\ell^{(\mu)}\big) +
u\big(x_{i+\frac12},y_\ell^{(\mu)}\big) \right)
\\
& +  \sum \limits_{\nu = 2}^{{L}-1} \sum \limits_{\mu = 1}^{Q} 
\frac{ \widehat \omega_\nu \omega_\mu  }{\Delta x_i + \Delta y_\ell}
\left( \Delta x_i u\big(  x_i^{(\mu)},\widehat y_\ell^{(\nu)} \big) + \Delta y_\ell u\big(\widehat x_i^{(\nu)},y_\ell^{(\mu)}\big) \right),
\quad~ \forall u \in {\mathbb P}^k(K),
\end{split}
\end{equation}
where $\{\widehat w_\mu\}_{\mu=1}^{L}$ are the weights of the $L$-point Gauss--Lobatto quadrature. 
If labeling the bottom, right, top and left edges of $K$ as  
${\mathscr E }_1$, ${\mathscr E }_2$, ${\mathscr E }_3$ and ${\mathscr E }_4$, respectively, then the identity \eqref{eq:U2Dsplit} implies, for $1\le \mu \le N$, that
$
\varpi_{ {\mathscr E }_j  }^{(\mu)} = \frac{ \Delta x_i \widehat \omega_1 \omega_\mu}{ \Delta x_i + \Delta y_\ell },~j=1,3;~
\varpi_{ {\mathscr E }_j  }^{(\mu)} = \frac{ \Delta y_\ell \widehat \omega_1 \omega_\mu}{ \Delta x_i + \Delta y_\ell },~j=2,4.
$ 
According to Theorem \ref{thm:PP:2DRHD}, the CFL condition \eqref{eq:CFL:2DRHD} for our positivity-preserving DG schemes on Cartesian meshes becomes  
\begin{equation}\label{eq:CFL:2DCart}
\alpha \Delta t 
\left( \frac{1}{\Delta x_i} 
+ \frac{1}{\Delta y_\ell} \right)  
\le \widehat \omega_1  = \frac{1}{L(L-1)}.
\end{equation}
\end{remark}

\section{Numerical tests}\label{sec:examples}
In this section, we present 
numerical tests 
on several benchmark 
RHD problems to validate the accuracy and effectiveness of our IRP DG methods on 1D and 2D uniform 
Cartesian meshes.  
The third-order SSP-RK method \eqref{eq:RK31D} or SSP-MS method \eqref{eq:MS31D} will be employed for time discretization. 
Unless otherwise stated, we use the ideal EOS \eqref{eq:iEOS} with $\Gamma=5/3$, and set the CFL numbers as $0.3$, $0.15$, $0.1$, respectively, for the second-order ($P^1$-based), third-order ($P^2$-based), fourth-order ($P^3$-based) DG methods with the SSP-RK time discretization; the CFL numbers for the SSP-MS-DG methods are one-third of those for the SSP-RK-DG methods. 

For convenience, we refer to the 1D bound-preserving limiter \eqref{eq:IRP1}--\eqref{eq:IRP2} (cf.~\cite{QinShu2016}) as 
the {\em BP limiter}. Our IRP limiter \eqref{eq:IRP1}--\eqref{eq:IRP3} corresponds to a combination of 
the BP limiter and the {\em entropy limiter} \eqref{eq:IRP3}. 
The same names/abbreviations will be also used for those corresponding limiters in the 2D case. We will compare the results with the proposed IRP limiter and those with only the BP limiter.

\begin{table}[htb]
	\centering
	\caption{\small Example 1: 
		Errors at $t=0.2$ in the rest-mass density 
		for the proposed $P^k$-based DG methods ($k=1,2,3$), with the SSP-RK or SSP-MS time discretization,   
		at different spatial grid resolutions.
	}\label{tab:Ex1Dsmooth}
	\begin{tabular}{c|c|c|c|c|c|c|c|c|c}
		\hline
		&     & \multicolumn{4}{c|}{SSP-RK}             & \multicolumn{4}{c}{SSP-MS}             \\ \hline
		$k$                   & $N$   & $l^1$ error & order & $l^2$ error & order & $l^1$ error & order & $l^2$ error & order \\ \hline
		\multirow{6}{*}{1} 
		& 10  & 1.75e-2  & --    & 1.98e-2  & --    & 1.72e-2  & --    & 1.96e-2  & --    \\  
		& 20  & 3.16e-3  & 2.47  & 4.47e-3  & 2.14  & 3.13e-3  & 2.46  & 4.41e-3  & 2.15  \\ 
		& 40  & 8.19e-4  & 1.95  & 1.08e-3  & 2.05  & 8.19e-4  & 1.94  & 1.07e-3  & 2.05  \\ 
		& 80  & 1.93e-4  & 2.08  & 2.49e-4  & 2.11  & 1.92e-4  & 2.09  & 2.46e-4  & 2.12  \\ 
		& 160 & 4.63e-5  & 2.06  & 5.87e-5  & 2.08  & 4.61e-5  & 2.06  & 5.84e-5  & 2.08  \\ 
		& 320 & 1.13e-5  & 2.04  & 1.38e-5  & 2.09  & 1.12e-5  & 2.04  & 1.37e-5  & 2.09  \\ \hline
		\multirow{6}{*}{2} 
		& 10  & 1.24e-3  & --    & 1.47e-3  & --    & 7.76e-4  & 3.21  & 9.15e-4  & 3.31  \\ 
		& 20  & 1.83e-4  & 2.76  & 2.37e-4  & 2.63  & 8.40e-5  & 3.04  & 9.24e-5  & 3.02  \\ 
		& 40  & 2.75e-5  & 2.74  & 4.98e-5  & 2.25  & 1.02e-5  & 3.00  & 1.14e-5  & 3.00  \\ 
		& 80  & 4.06e-6  & 2.76  & 1.03e-5  & 2.28  & 1.27e-6  & 3.00  & 1.42e-6  & 3.00  \\ 
		& 160 & 5.90e-7  & 2.78  & 2.12e-6  & 2.28  & 1.59e-7  & 3.00  & 1.77e-7  & 3.00  \\ 
		& 320 & 9.16e-8  & 2.69  & 4.51e-7  & 2.23  & 1.99e-8  & 3.00  & 2.22e-8  & 3.00  \\ \hline
		\multirow{6}{*}{3} 
		& 10  & 6.12e-5  & --    & 8.32e-5  & --    & 1.92e-5  & --    & 2.24e-5  & --    \\ 
		& 20  & 4.84e-6  & 3.66  & 9.94e-6  & 3.07  & 1.29e-6  & 3.90  & 1.50e-6  & 3.90  \\ 
		& 40  & 3.00e-7  & 4.01  & 9.89e-7  & 3.33  & 7.86e-8  & 4.04  & 8.89e-8  & 4.07  \\ 
		& 80  & 2.71e-8  & 3.47  & 1.26e-7  & 2.97  & 4.85e-9  & 4.02  & 5.48e-9  & 4.02  \\ 
		& 160 & 2.30e-9  & 3.55  & 1.52e-8  & 3.05  & 3.04e-10 & 4.00  & 3.42e-10 & 4.00  \\ 
		& 320 & 2.08e-10 & 3.47  & 1.84e-9  & 3.04  & 1.90e-11 & 4.00  & 2.13e-11 & 4.01  \\ \hline
	\end{tabular}
\end{table}

\subsection{Example 1: 1D smooth problem}
To examine the accuracy of our 1D DG 
methods we first test   
a smooth problem similar to \cite{QinShu2016,WuTang2017ApJS,zhang2010b,zhang2012minimum}. 
The initial conditions are $\rho(x,0) = 1+0.99999 \sin \big(2 \pi x\big) $, $v(x,0) = 0.9$, $p(x,0)=1$. 
The computational domain is taken as $[0,1]$ with periodic boundary conditions, so that the exact solution is 
$\rho(x,t) = 1+0.99999 \sin \big(2 \pi (x-0.9t)\big) $, $v(x,t) = 0.9$, $p(x,t)=1$. 
In the computations, the domain is partitioned into $N$ uniform cells 
with $N \in \{ 10, 20,40,80,160,320 \}$. 
For the $P^3$-based DG method, we take (only in this accuracy test) the time step-sizes as $\Delta t = 0.1 \Delta x ^{\frac43}$ 
and $\Delta t = \frac{0.1}3 \Delta x ^{\frac43}$ for the third-order SSP-RK and SSP-MS time discretizations respectively, 
so as to 
match the fourth-order accuracy of spatial discretization. 

Table \ref{tab:Ex1Dsmooth} lists the numerical errors at $t=0.2$ in the rest-mass density and the corresponding convergence rates 
for the $P^k$-based IRP DG methods ($k=1,2,3$) 
at different grid resolutions. 
As observed in \cite{zhang2012minimum,QinShu2016}, 
the accuracy degenerates for SSP-RK and $k\ge 2$, which is due to the lower order
accuracy in the RK intermediate stages as explained in \cite{zhang2012minimum}. 
The desired full order of accuracy is observed for the SSP-MS time discretization, indicating that the IRP limiter itself does not destroy the accuracy for smooth solutions as expected from the analyses in \cite{zhang2010,ZHANG2017301,jiang2018invariant}.


\subsection{Example 2: Two 1D Riemann problems}
This example investigates the capability of the 1D IRP DG methods in resolving discontinuous solutions, by testing two 1D Riemann problems. The computational domain is taken as $[0,1]$. 

\begin{figure}[htbp]
	\centering
	\begin{subfigure}[b]{0.49\textwidth}
		\begin{center}
			\includegraphics[width=0.99\linewidth]{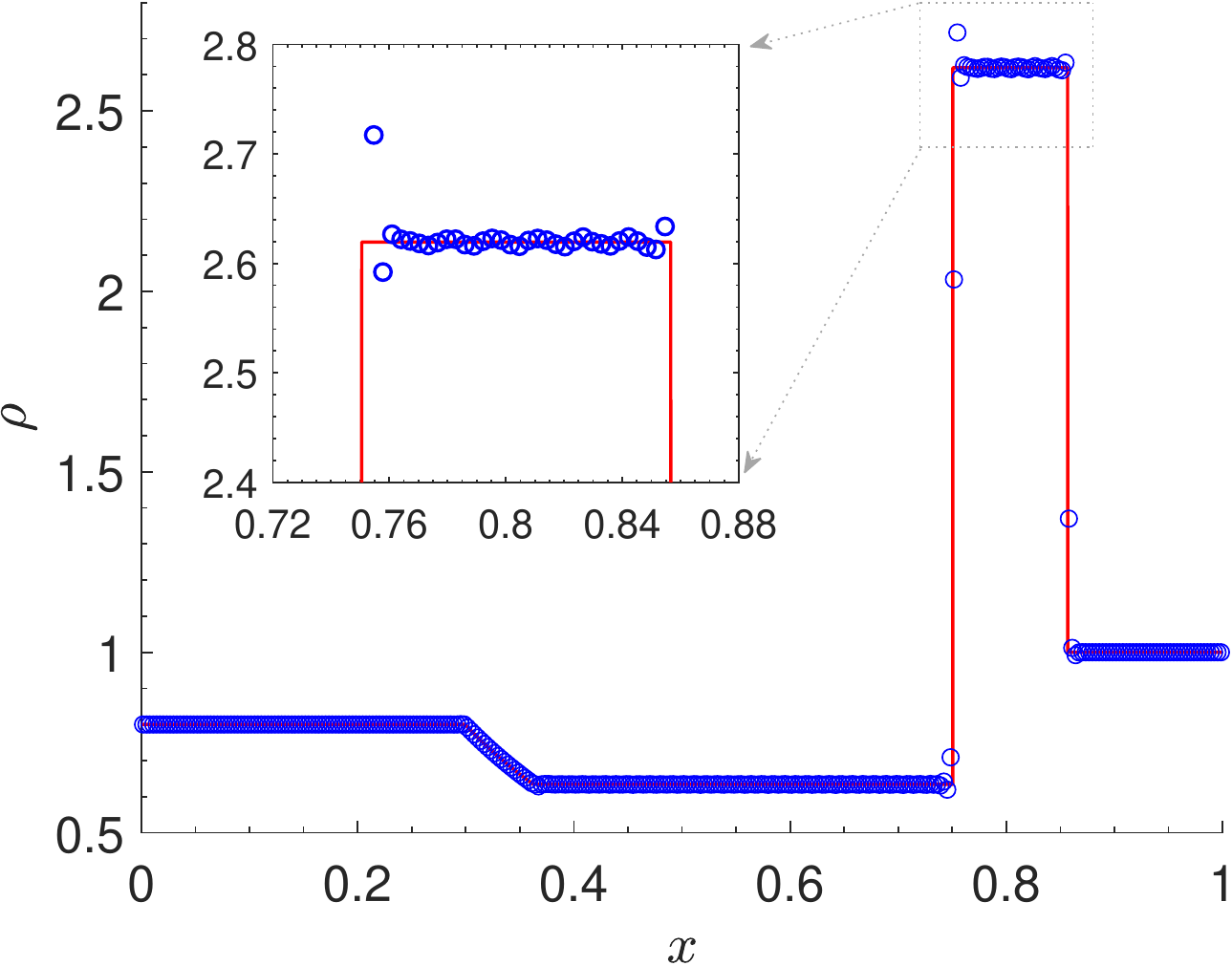}
		\end{center}
	\end{subfigure}
	\begin{subfigure}[b]{0.49\textwidth}
		\begin{center}
			\includegraphics[width=0.99\linewidth]{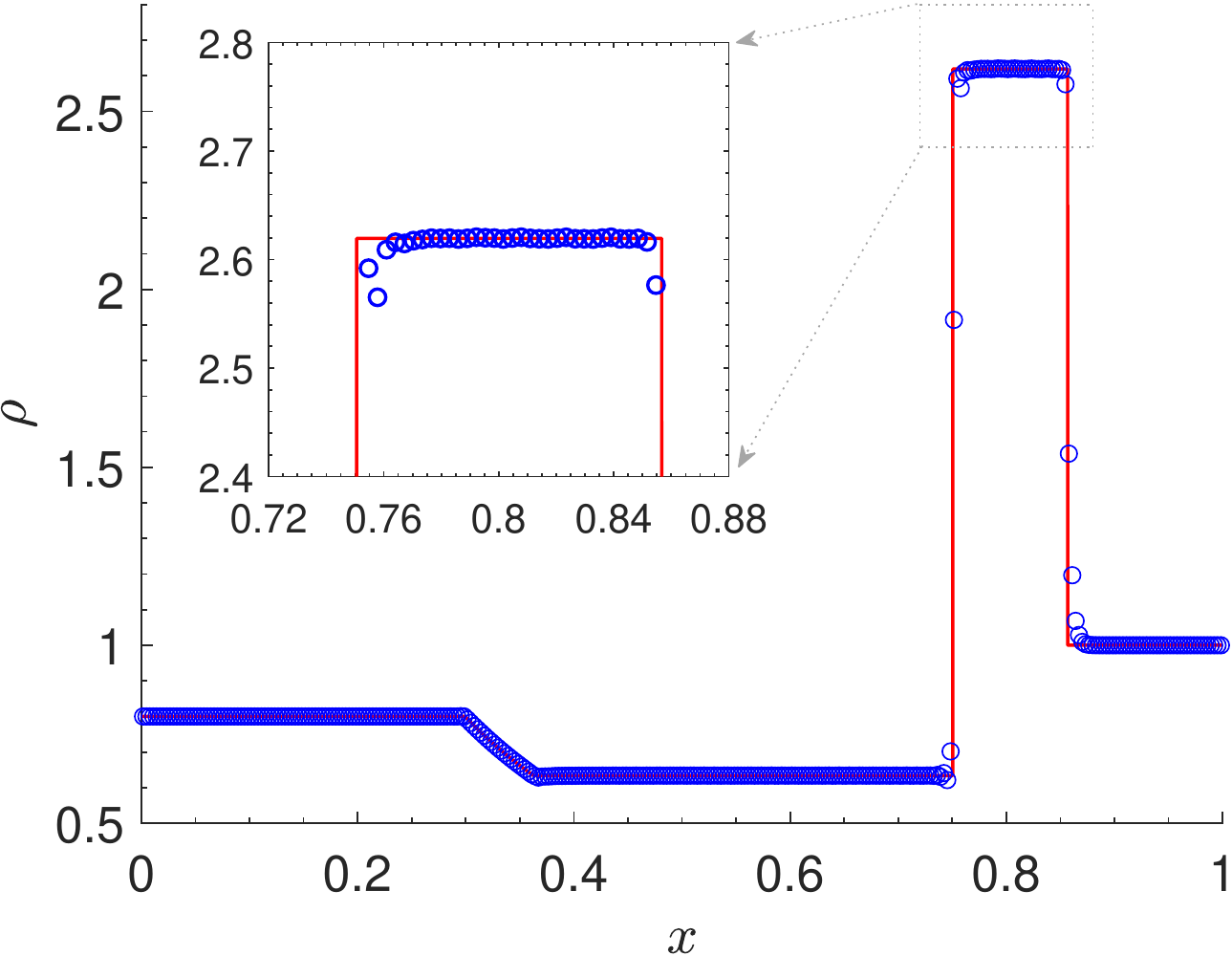}
		\end{center}
	\end{subfigure}
	\begin{subfigure}[b]{0.49\textwidth}
		\begin{center}
			\includegraphics[width=0.99\linewidth]{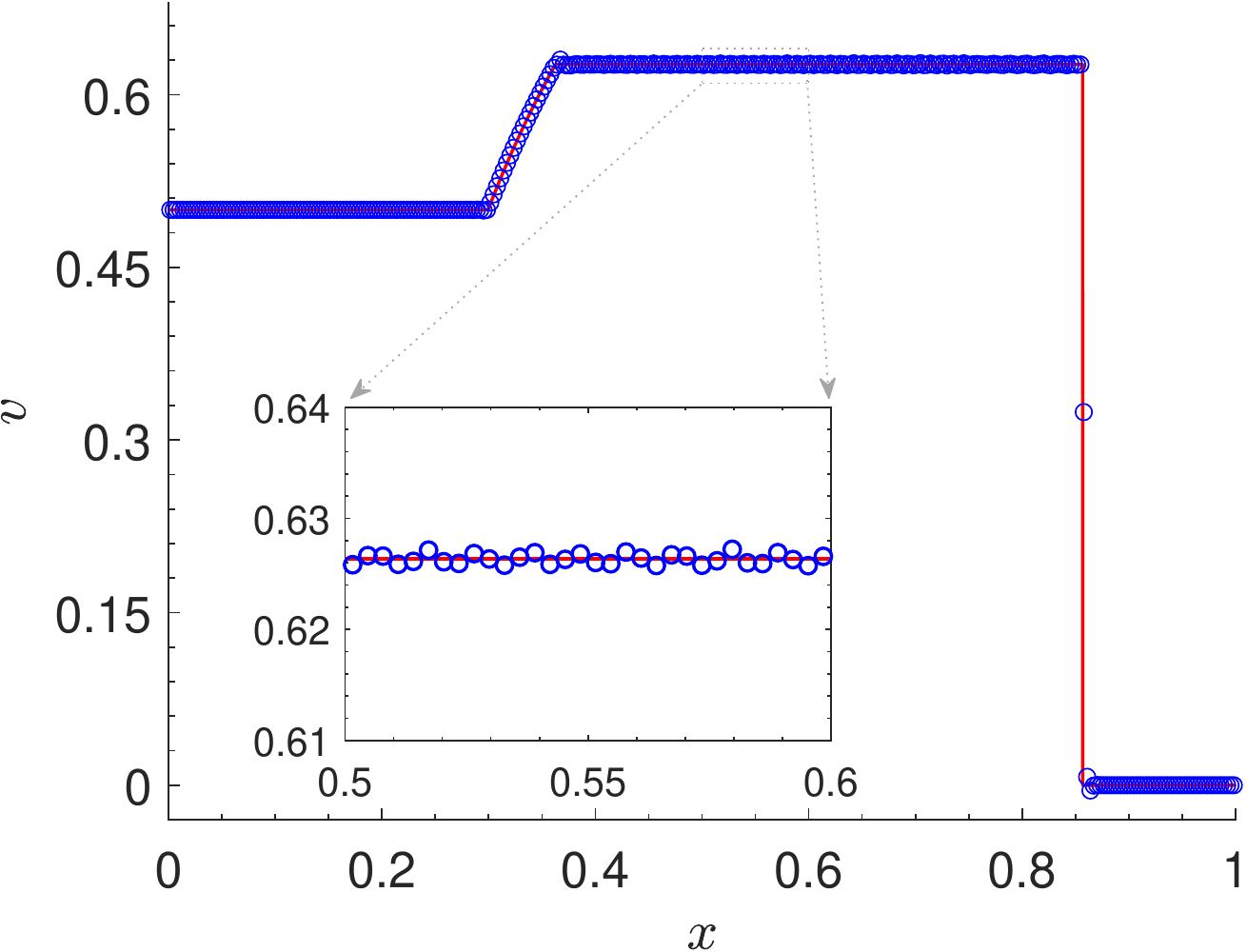}
		\end{center}
	\end{subfigure}
	\begin{subfigure}[b]{0.49\textwidth}
		\begin{center}
			\includegraphics[width=0.99\linewidth]{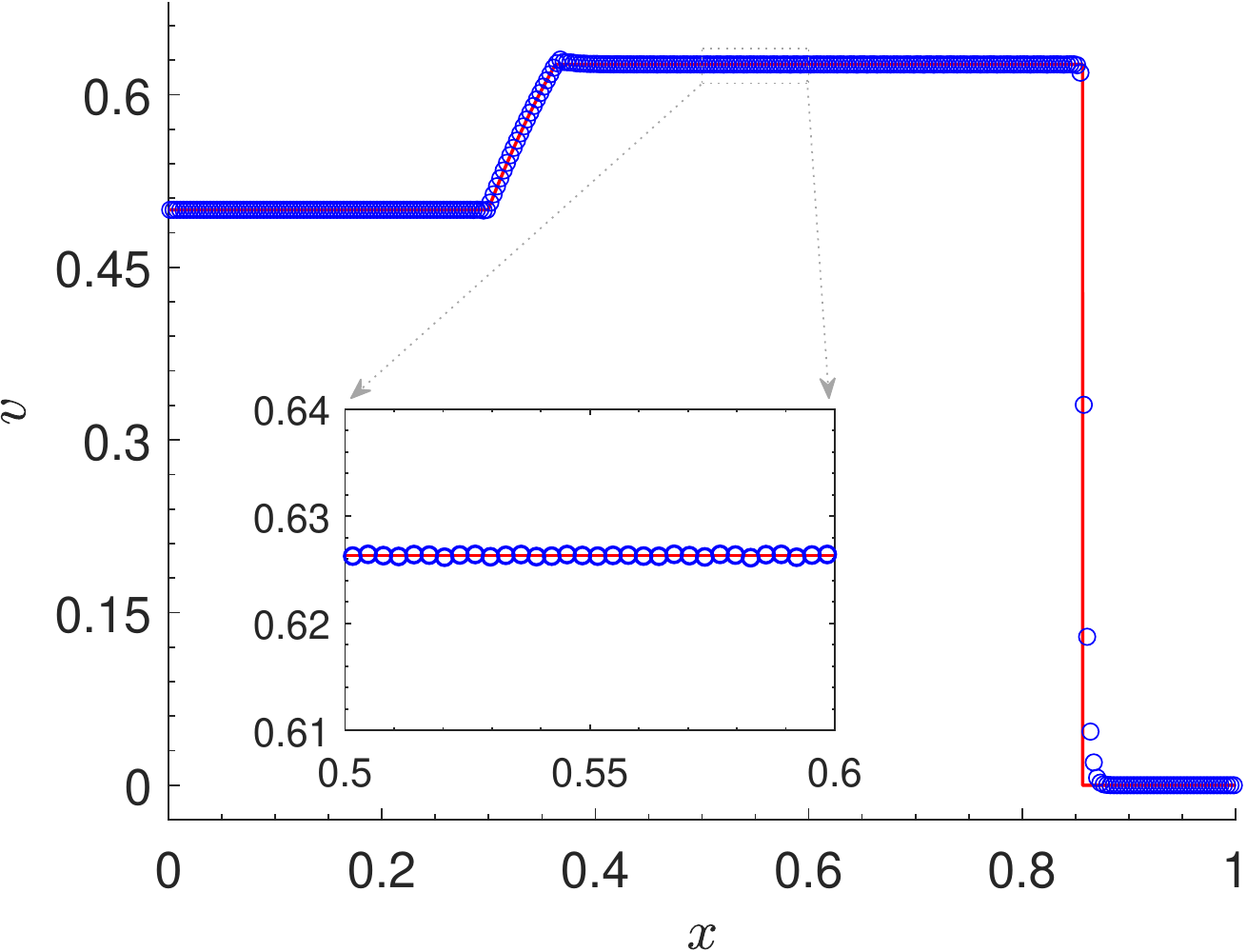}
		\end{center}
	\end{subfigure}
	\begin{subfigure}[b]{0.49\textwidth}
		\captionsetup{width=.8\linewidth}
		\begin{center}
			\includegraphics[width=0.99\linewidth]{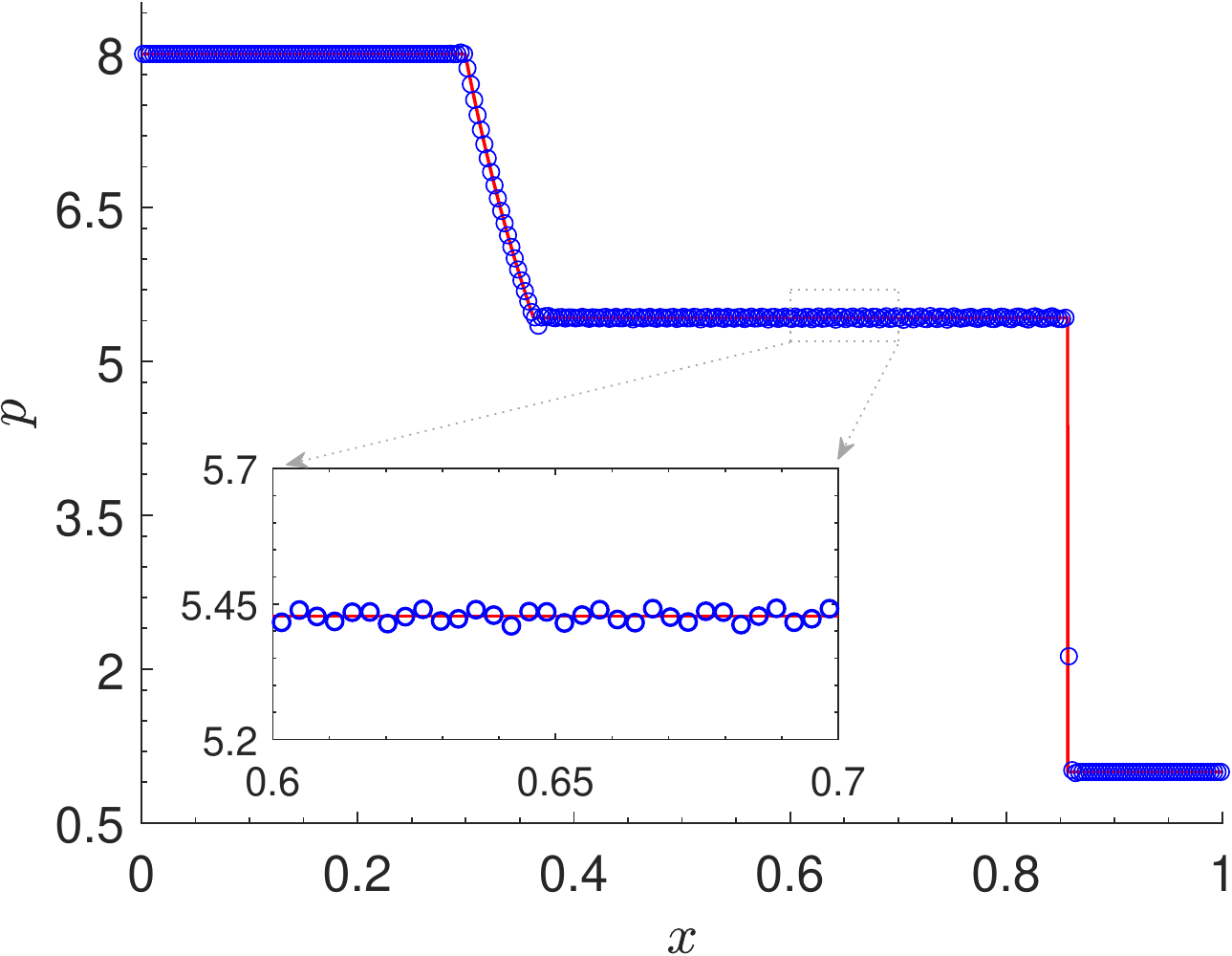}
		\end{center}
		\caption{\small With the BP limiter}
	\end{subfigure}
	\begin{subfigure}[b]{0.49\textwidth}
		\captionsetup{width=.8\linewidth}
		\begin{center}
			\includegraphics[width=0.99\linewidth]{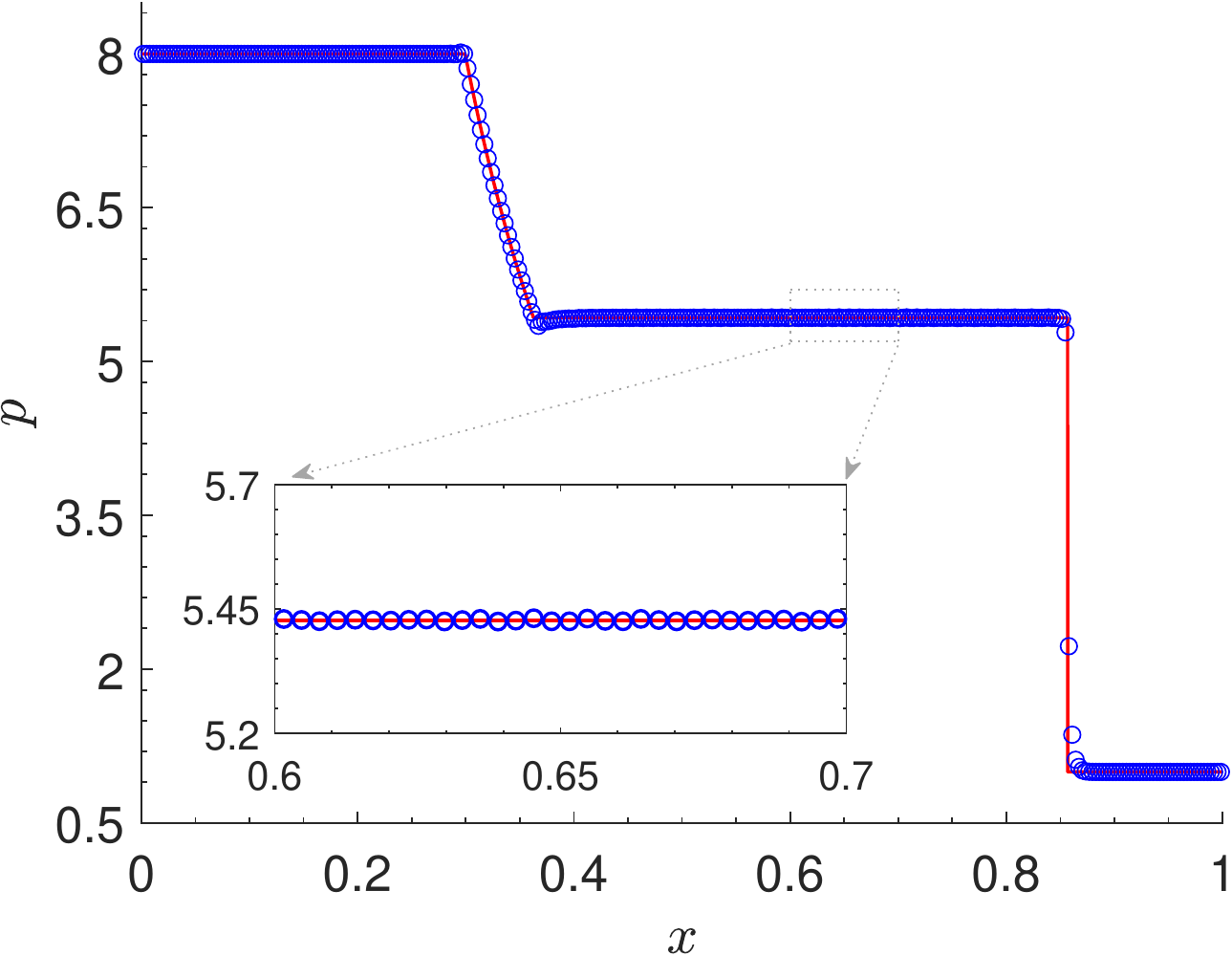}
		\end{center}
		\caption{\small With the IRP limiter}
	\end{subfigure}
	\caption{\small Example 2: Solutions (and their close-up) of the first 1D Riemann problem at $t=0.4$. The symbols ``$\circ$'' denote the numerical solutions obtained by 
		the fourth-order DG methods with the BP limiter (left) or with the IRP limiter (right) on the mesh of $320$ uniform cells, while the solid lines denote the exact solution. 
		(Here we do not use any other non-oscillatory limiters, e.g. TVD/TVB or WENO limiters.)} 
	\label{fig:RP1solu}
\end{figure}

The initial conditions of the first Riemann problem are 
\begin{equation*}
(\rho,v,p)(x,0)=
\begin{cases}
(0.8,0.5,8),  \quad &x<0.5,
\\
(1,0,1), \quad &x>0.5.
\end{cases}
\end{equation*}
The initial discontinuity will evolve as a left-moving rarefaction wave, a constant discontinuity, and a right-moving shock wave. 
Figure \ref{fig:RP1solu} displays the numerical solutions 
computed by the fourth-order DG methods with the BP or IRP limiter respectively 
 on a mesh of 320 uniform cells, against
 the exact solution. 
Note that, in this simulation, we do {\em not} use any other  
non-oscillatory limiters such as the TVD/TVB or WENO limiters. 
We can observe that the numerical results with only the BP limiter exhibit overshoot in the rest-mass density near the contact discontinuity
and some small oscillations. 
When the IRP limiter is applied (i.e., the entropy limiter \eqref{eq:IRP3} is added), 
the overshoot and oscillations in the DG solution are damped. 
This is consistent with the observation in  
\cite{khobalatte1994maximum,zhang2012minimum,jiang2018invariant} that enforcing the 
discrete minimum entropy principle could help to oppress numerical oscillations. 
Figure \ref{fig:RP1_MinS} shows the time evolution of the minimum
specific entropy values of the DG solutions. It is seen that 
the minimum remains the same for the DG scheme with the IRP limiter, which indicates that the 
minimum entropy principle is preserved, while the DG scheme with only the BP limiter fails to 
keep the principle.

\begin{figure}[htbp]
	\centering
	\begin{subfigure}[b]{0.49\textwidth}
		\begin{center}
			\includegraphics[width=0.99\linewidth]{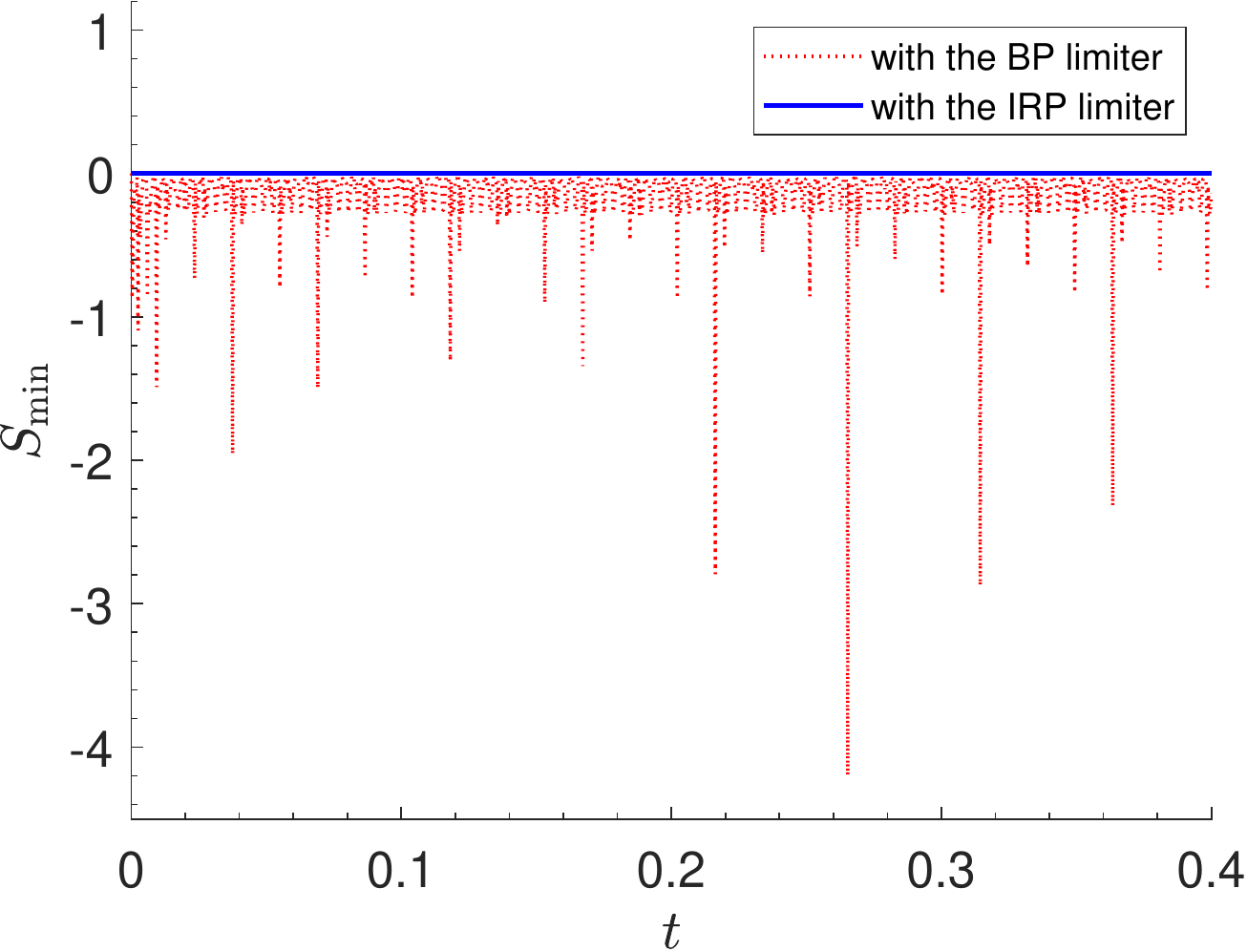}
		\end{center}
	\end{subfigure}
	\begin{subfigure}[b]{0.49\textwidth}
		\begin{center}
			\includegraphics[width=0.99\linewidth]{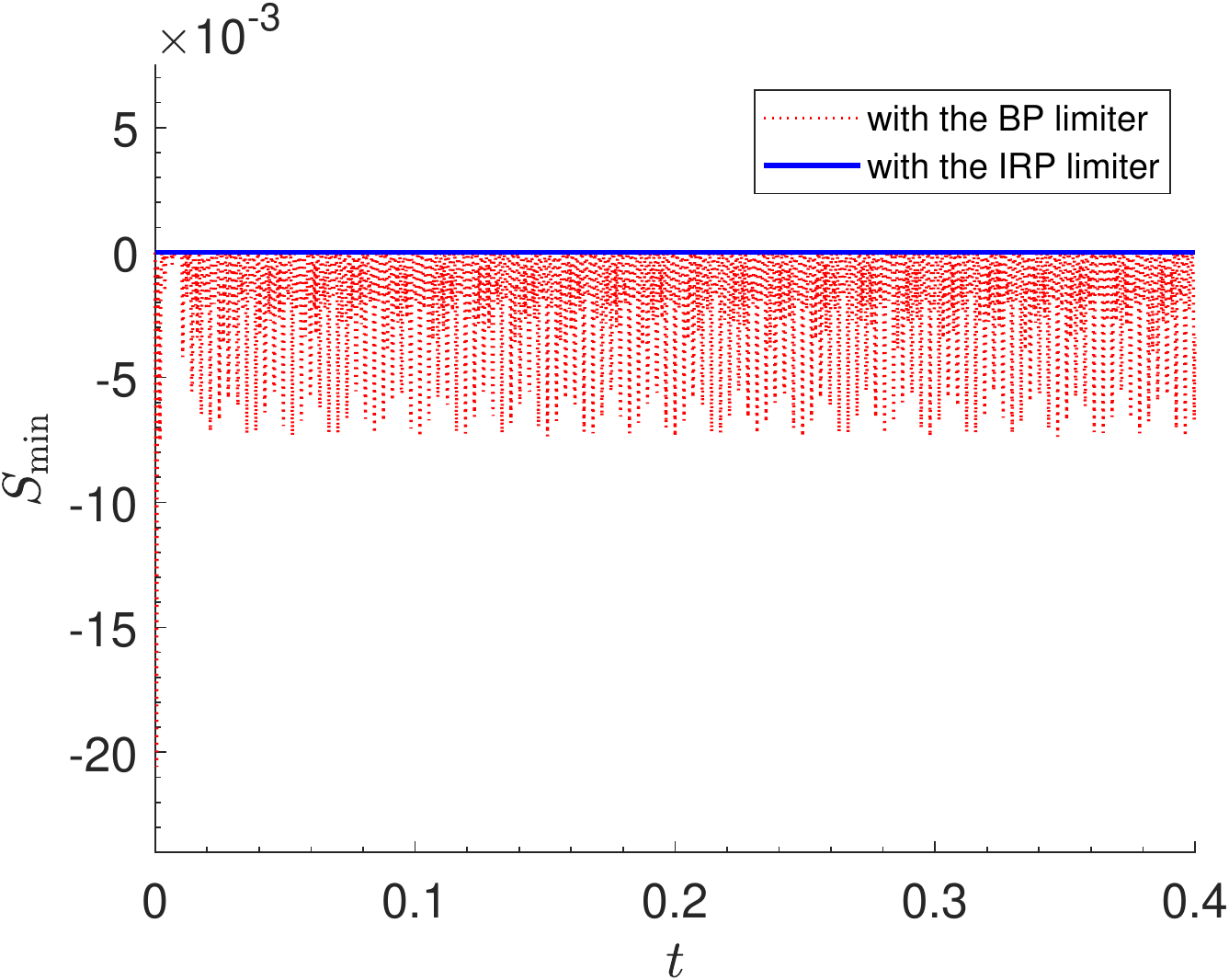}
		\end{center}
	\end{subfigure}
	\caption{\small The first Riemann problem of Example 2: Time evolution of $S_{\min}(t)$ for the DG solutions with the IRP limiter \eqref{eq:IRP1}--\eqref{eq:IRP3} or with the BP limiter \eqref{eq:IRP1}--\eqref{eq:IRP2}. Left: $S_{\min}(t) = \min_{j,\mu} S({\bf U}_h(\widehat x_j^{(\mu)},t))$; right: $S_{\min}(t) = \min_j S(\overline{\bf U}_j(t))$.}
	\label{fig:RP1_MinS}
\end{figure}


In order to demonstrate the robustness and resolution of the proposed IRP DG methods, we simulate a 
ultra-relativistic Riemann problem \cite{WuTang2015}. 
The initial conditions are
\begin{equation*}
(\rho,v,p)(x,0)=
\begin{cases}
(1,0,10^4),  \quad &x<0.5,
\\
(1,0,10^{-8}), \quad &x>0.5,
\end{cases}
\end{equation*}
which involve very low pressure and strong initial jump in pressure ($\Delta p := |p_R-p_L|/p_R \approx 10^{12} $), so that the simulation of this problem is challenging and the constraint-preserving or BP  techniques have to be used \cite{WuTang2015,QinShu2016,WuTang2017ApJS}. 
The initial discontinuity will evolve as a strong left-moving rarefaction wave, a quickly right-moving contact discontinuity, and a quickly right-moving shock wave. 
More precisely, the speeds of the contact discontinuity and the shock wave are about 0.986956 and 0.9963757 respectively, and are very close to the speed of light $c=1$. 
Figure \ref{fig:RP2solu} presents the numerical solutions 
computed by the fourth-order DG methods with the BP 
or IRP limiter respectively  
on a mesh of 400 uniform cells, against
the exact solution. 
Due to 
the ultra-relativistic effect, 
a clearly curved profile for
the rarefaction fan is yielded (see Figure \ref{fig:RP2solu}), as opposed to a linear one in the non-relativistic case. 
Again, here we do {\em not} use any other  
non-oscillatory limiters such as the TVD/TVB or WENO limiters. 
We see that both DG methods work very robustly and exhibit similar high resolution  (the numerical solutions are comparable to those obtained by 
the ninth-order bound-preserving finite difference WENO methods in \cite{WuTang2015}).  
This implies that the use of entropy limiter in the IRP DG method keeps 
the robustness and does not destroy the high resolution 
 of the scheme. However, without the entropy limiter (i.e., with only the BP limiter), the DG scheme would not preserve
the minimum entropy principle and thus is not IRP, as confirmed by the plots in Figure \ref{fig:RP2_MinS}. 
We also remark that if the BP or IRP limiter is not applied, the DG code would break down 
due to nonphysical numerical solutions exceeding the set ${\mathcal G}$.

\begin{figure}[htbp]
	\centering
	\begin{subfigure}[b]{0.49\textwidth}
		\begin{center}
			\includegraphics[width=0.99\linewidth]{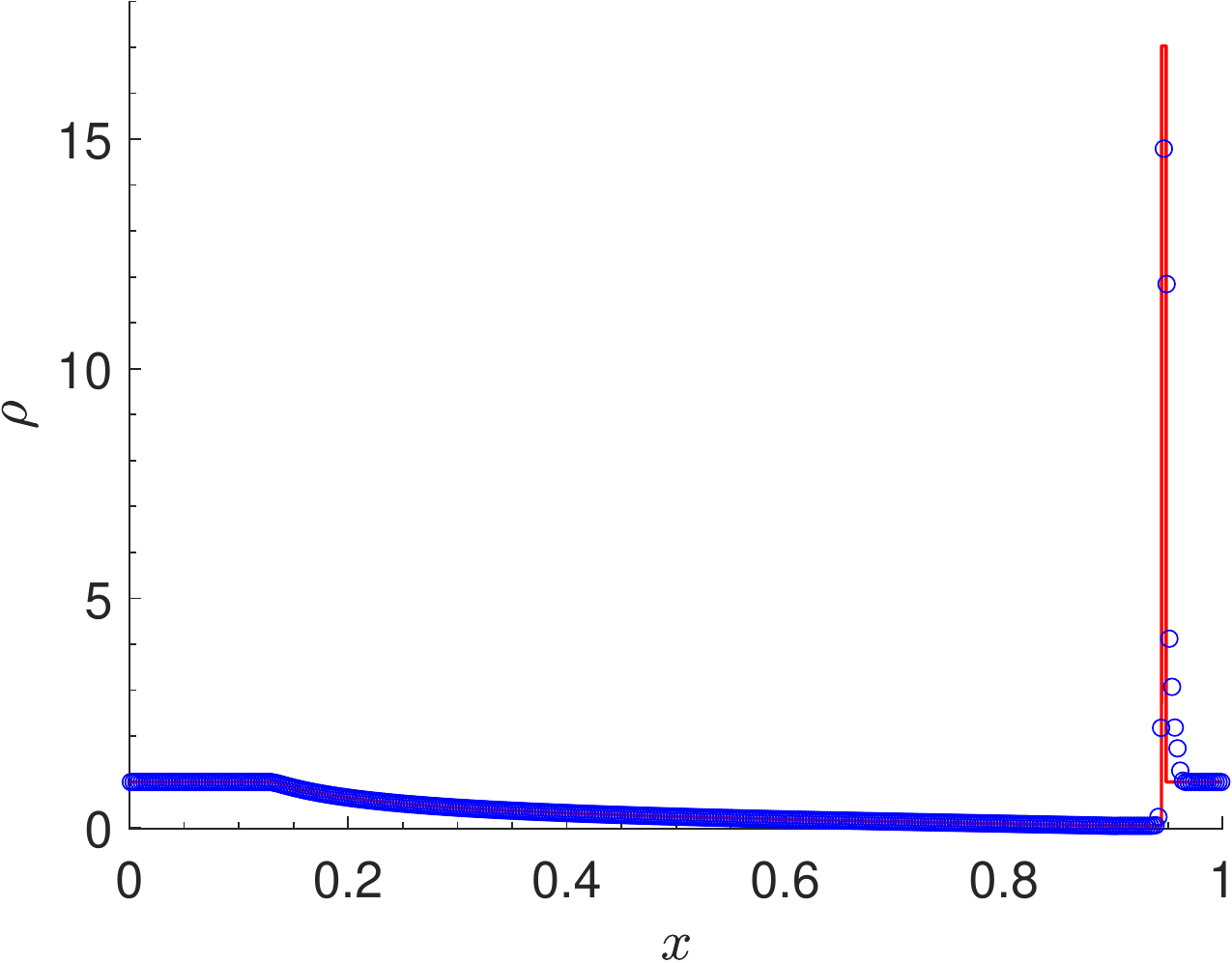}
		\end{center}
	\end{subfigure}
	\begin{subfigure}[b]{0.49\textwidth}
		\begin{center}
			\includegraphics[width=0.99\linewidth]{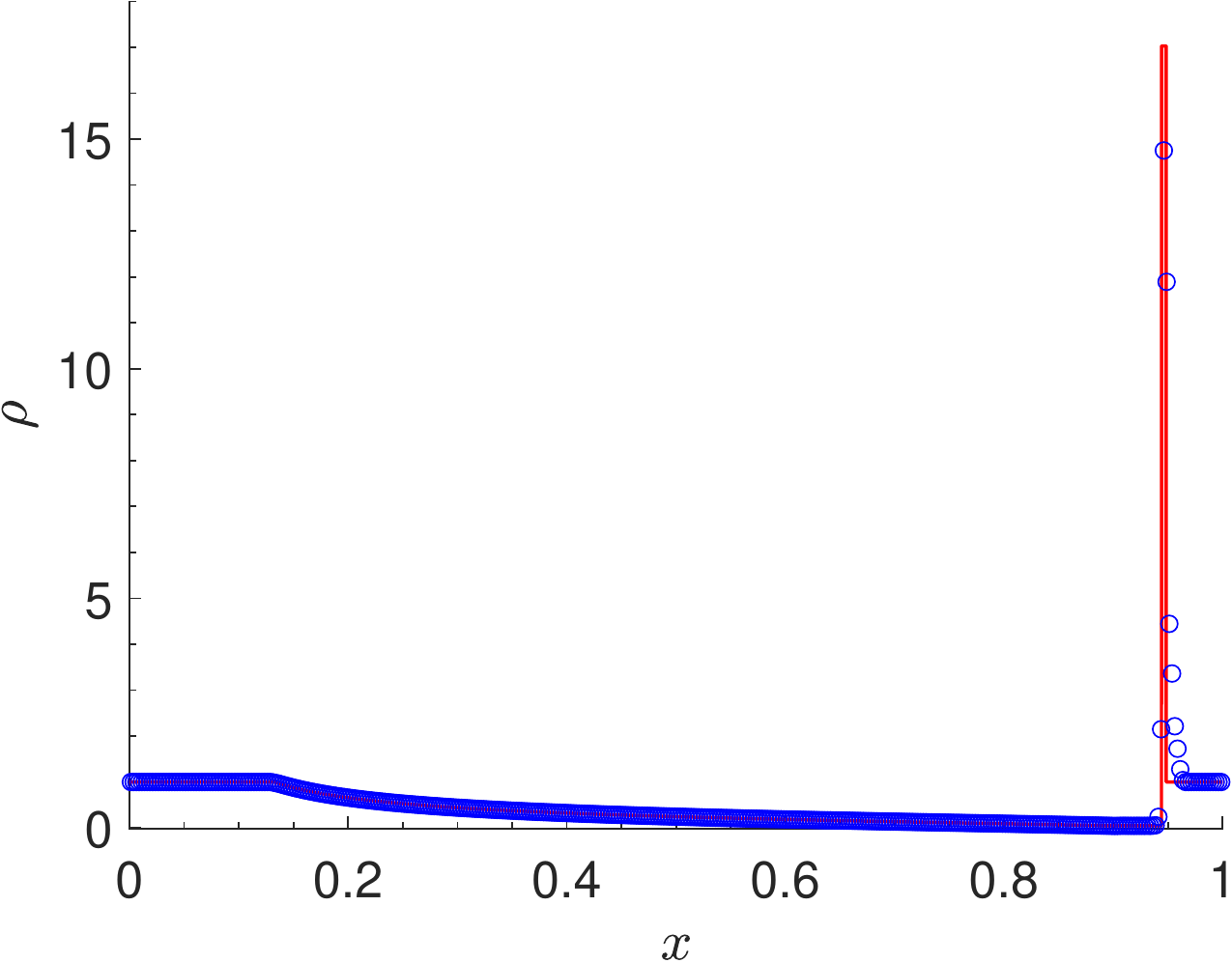}
		\end{center}
	\end{subfigure}
	\begin{subfigure}[b]{0.49\textwidth}
		\begin{center}
			\includegraphics[width=0.99\linewidth]{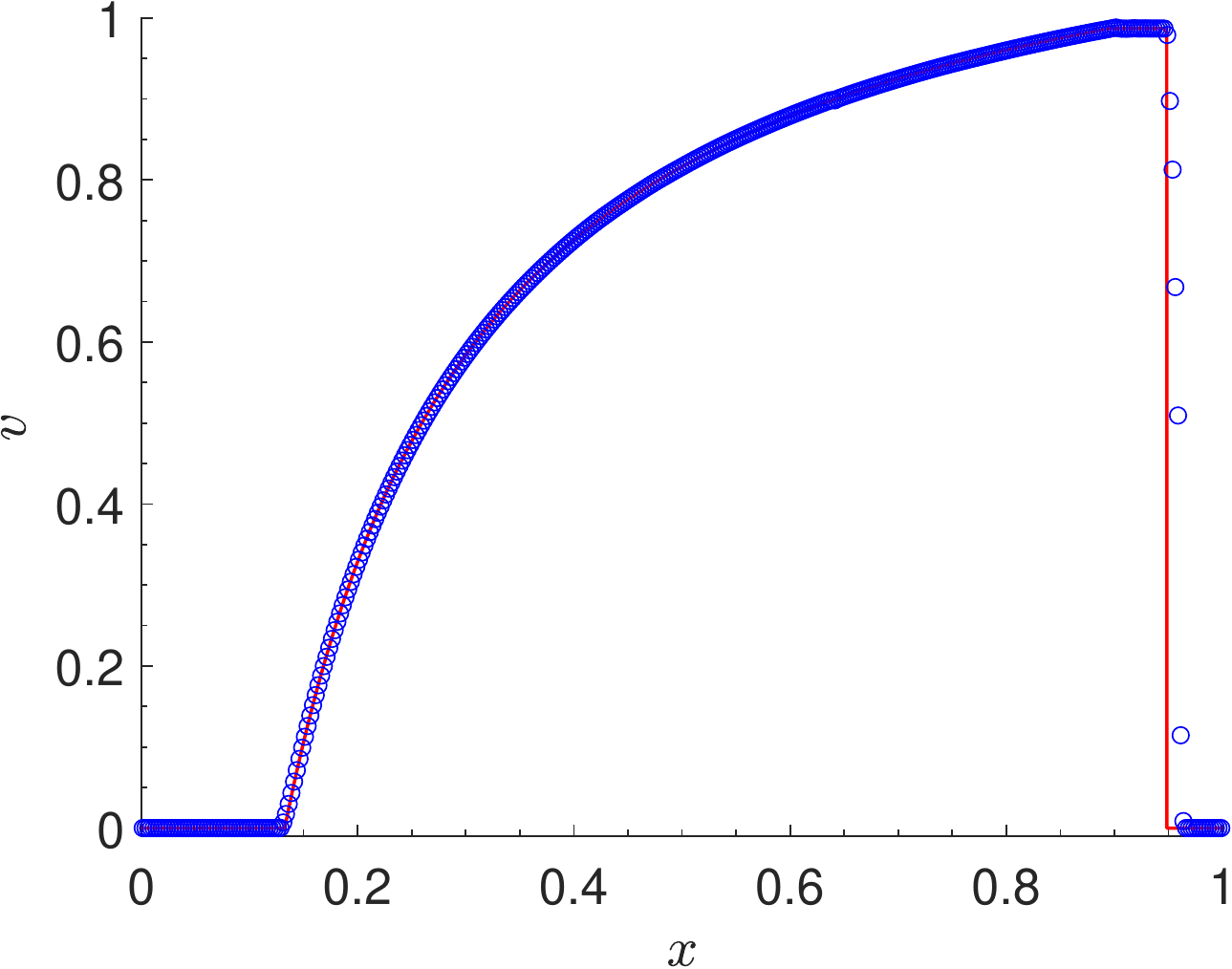}
		\end{center}
	\end{subfigure}
	\begin{subfigure}[b]{0.49\textwidth}
		\begin{center}
			\includegraphics[width=0.99\linewidth]{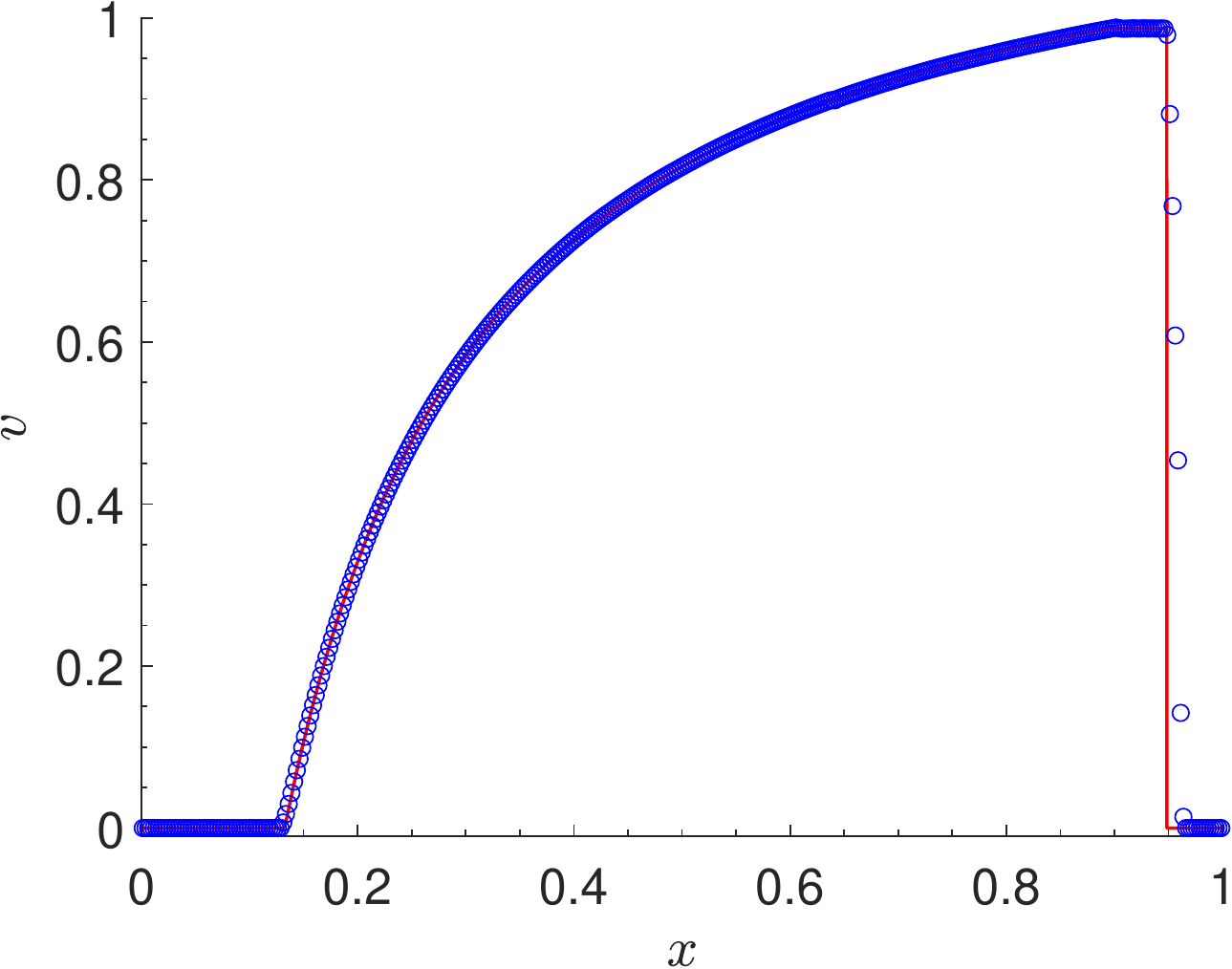}
		\end{center}
	\end{subfigure}
	\begin{subfigure}[b]{0.49\textwidth}
		\captionsetup{width=.8\linewidth}
		\begin{center}
			\includegraphics[width=0.99\linewidth]{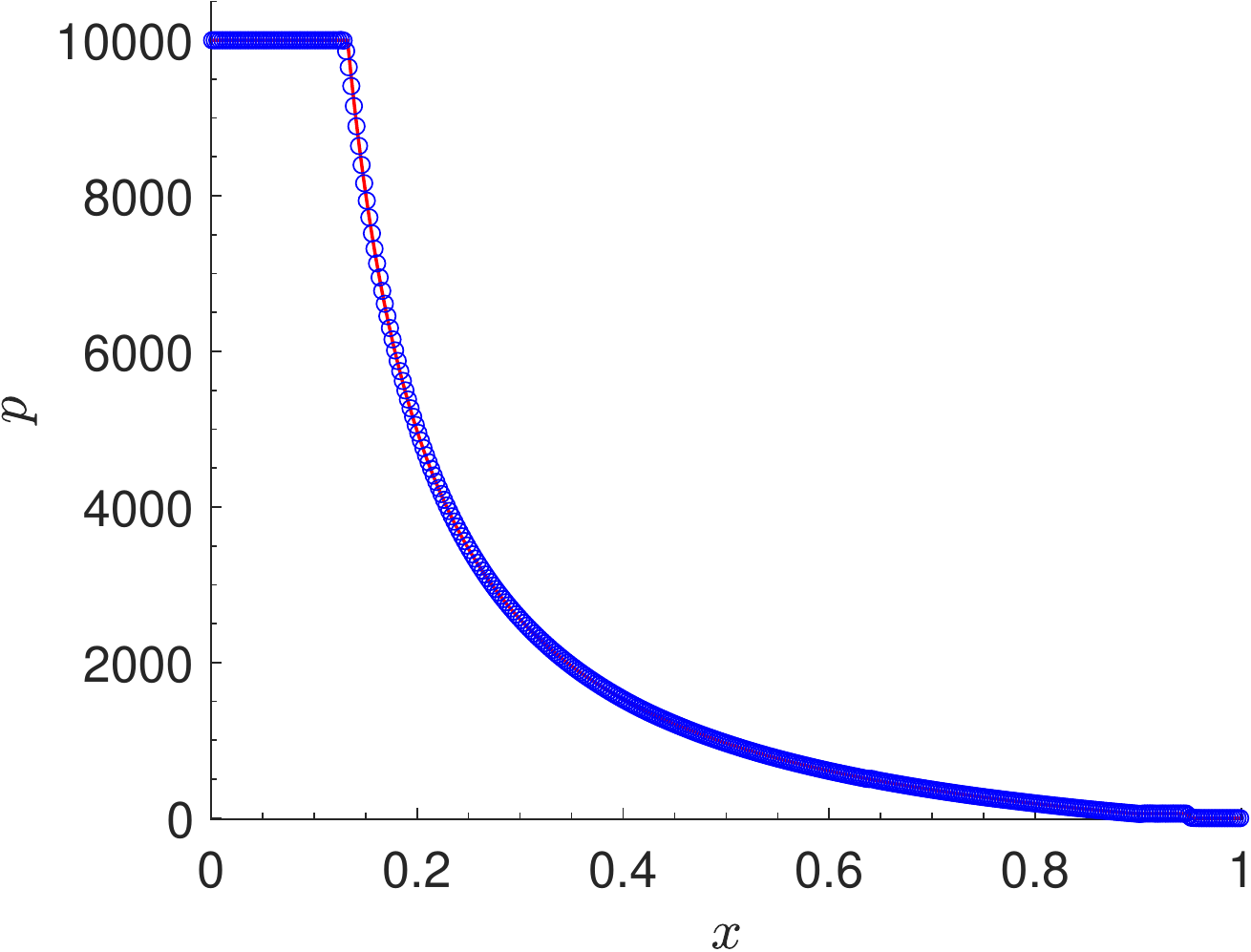}
		\end{center}
	\caption{\small With the BP limiter}
	\end{subfigure}
	\begin{subfigure}[b]{0.49\textwidth}
		\captionsetup{width=.8\linewidth}
		\begin{center}
			\includegraphics[width=0.99\linewidth]{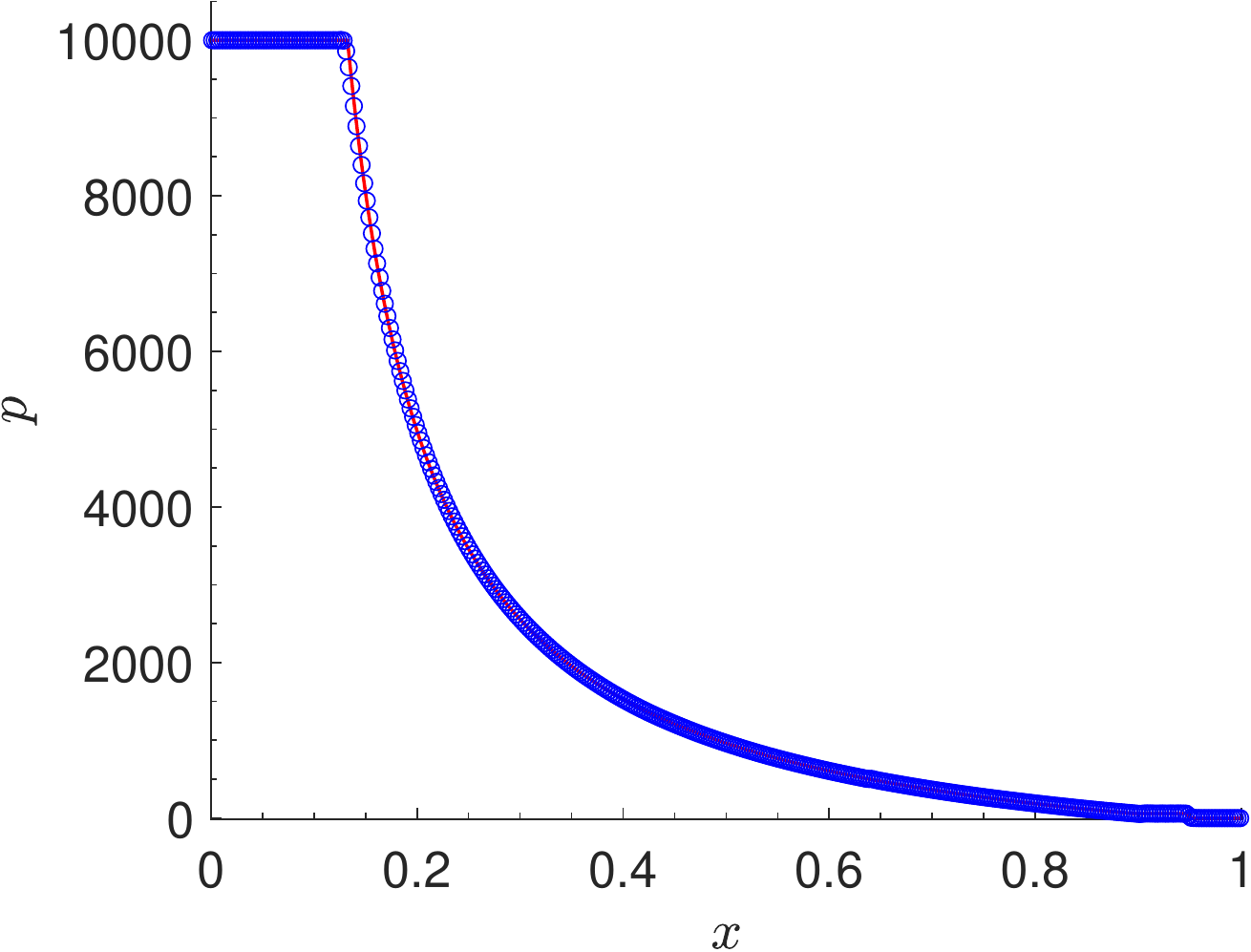}
		\end{center}
	\caption{\small With the IRP limiter}
	\end{subfigure}
	\caption{\small Example 2: Solutions of the second 1D Riemann problem at $t=0.45$. The symbols ``$\circ$'' denote the numerical solutions obtained by 
		the fourth-order DG methods with the BP limiter (left) or with the IRP limiter (right) on the mesh of $400$ uniform cells, while the solid lines denote the exact solution. (Here we do not use any other non-oscillatory limiters, e.g. TVD/TVB or WENO limiters.)} 
	\label{fig:RP2solu}
\end{figure}

\begin{figure}[htbp]
	\centering
	\begin{subfigure}[b]{0.49\textwidth}
		\begin{center}
			\includegraphics[width=0.99\linewidth]{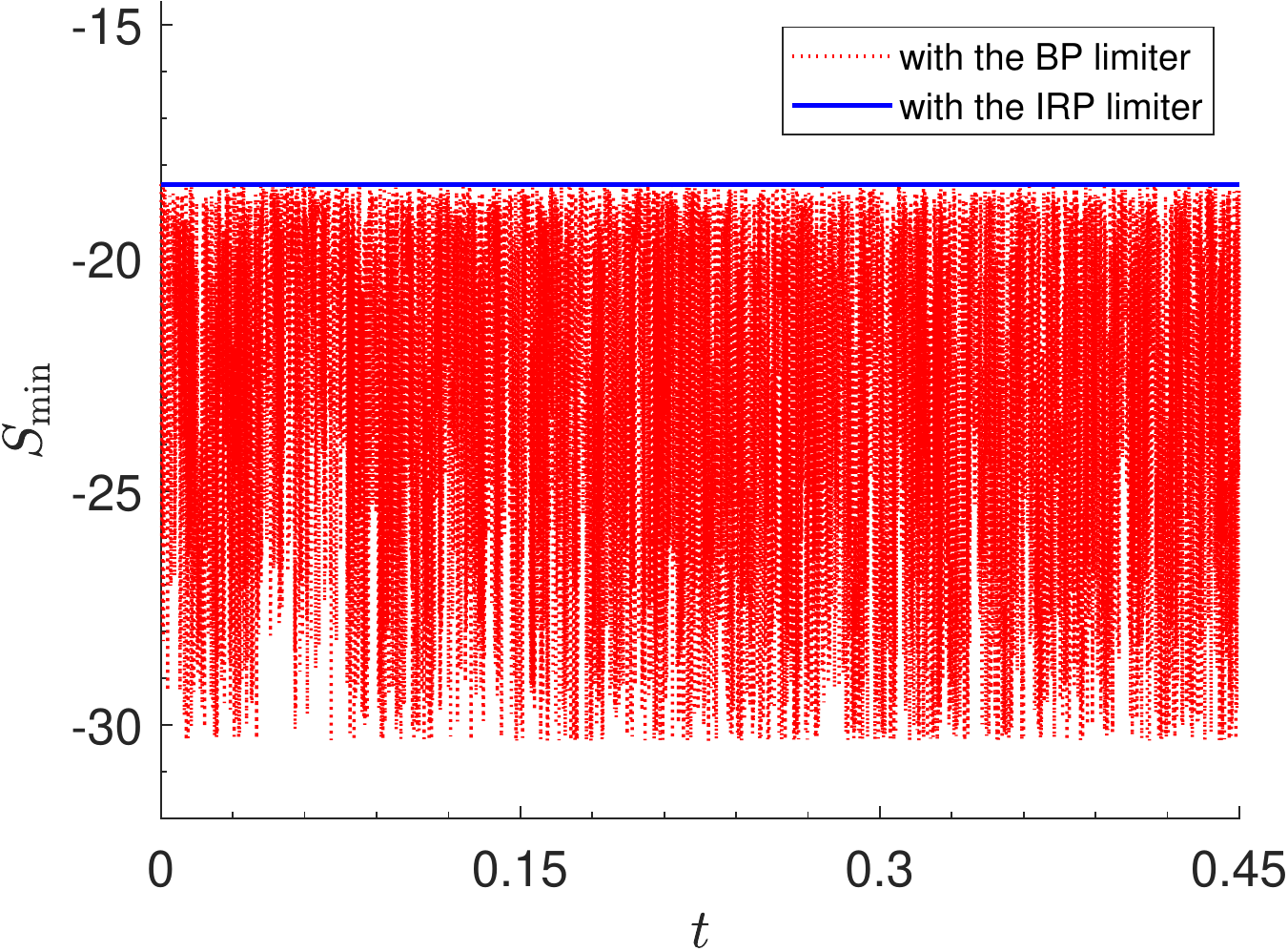}
		\end{center}
	\end{subfigure}
	\begin{subfigure}[b]{0.49\textwidth}
		\begin{center}
			\includegraphics[width=0.99\linewidth]{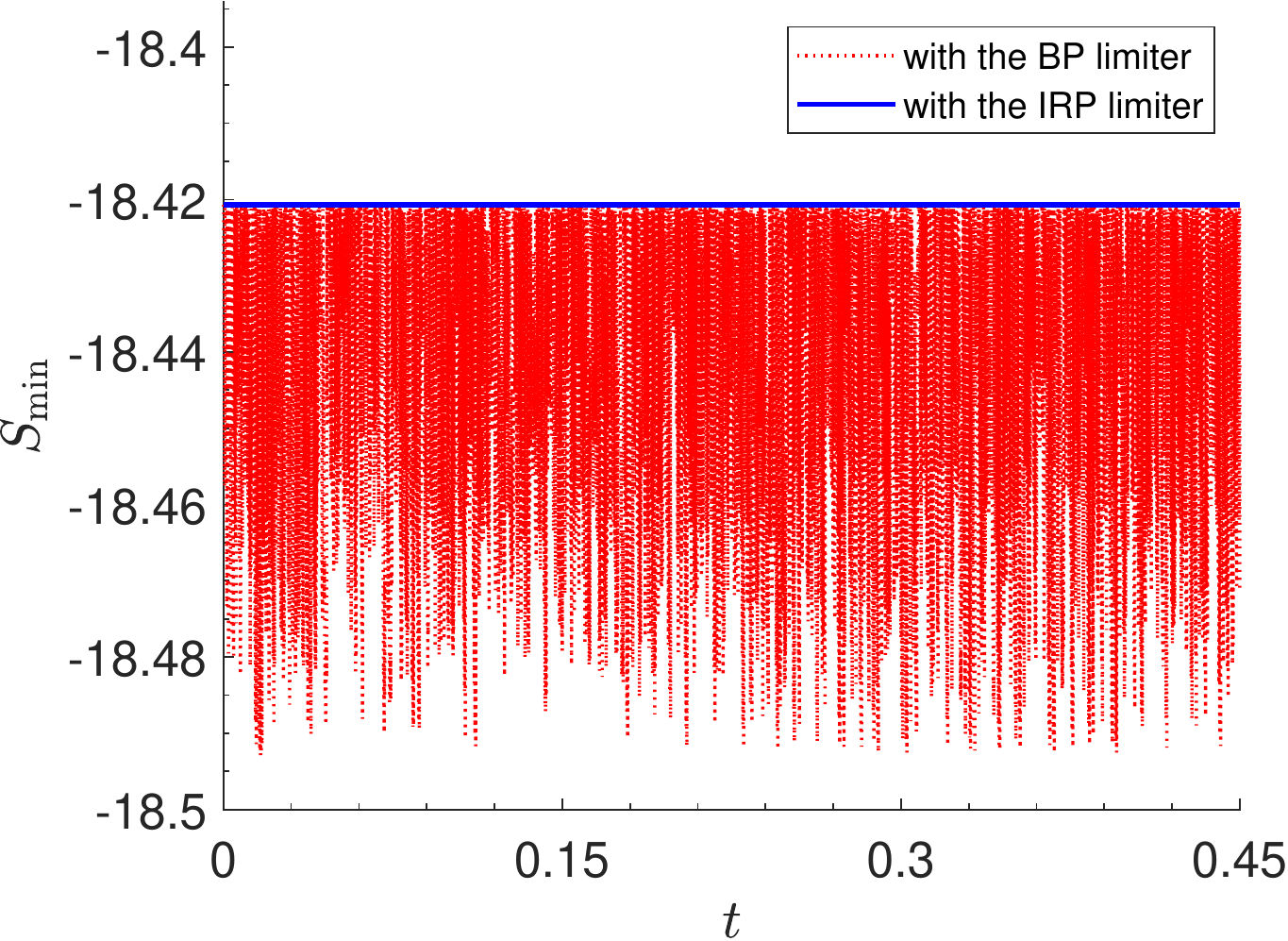}
		\end{center}
	\end{subfigure}
	\caption{\small The second Riemann problem of Example 2: Time evolution of $S_{\min}(t)$ for the DG solutions with the IRP limiter \eqref{eq:IRP1}--\eqref{eq:IRP3} or with the BP limiter \eqref{eq:IRP1}--\eqref{eq:IRP2}. Left: $S_{\min}(t) = \min_{j,\mu} S({\bf U}_h(\widehat x_j^{(\mu)},t))$; right: $S_{\min}(t) = \min_j S(\overline{\bf U}_j(t))$.}
	\label{fig:RP2_MinS}
\end{figure}


\subsection{Example 3: 2D smooth problem}
To check the accuracy of our 2D IRP DG
methods we simulate    
a smooth problem from \cite{WuTang2017ApJS}. 
The initial data are 
\begin{equation*}
(\rho,{\bm v},p) (x,y,0) = \big( 1+0.99999 \sin( 2 \pi (x+y) ),~0.99/\sqrt{2},~0.99/\sqrt{2},~10^{-2} \big). 
\end{equation*}
The computational domain is taken as $[0,1]^2$ with periodic boundary conditions, so that 
the exact solution is 
\begin{equation*}
(\rho,{\bm v},p) (x,y,t) = \big( 1+0.99999 \sin( 2 \pi (x+y-0.99\sqrt{2}t) ),~0.99/\sqrt{2},~0.99/\sqrt{2},~10^{-2} \big),  
\end{equation*}
which describes the propagation of an RHD sine wave with low density, low pressure, and large velocity,   
in the domain $[0,1]^2$ at an angle $45^{\circ}$ with the $x$-axis.

In the computations, the domain is partitioned into $N \times N$ uniform rectangular cells 
with $N \in \{ 10, 20,40,80,160 \}$. 
For the $P^3$-based DG method, we take (only in this accuracy test) the time step-sizes as $\Delta t = 0.1 \big( \frac{\Delta x}2 \big)^{\frac43}$ 
and $\Delta t = \frac{0.1}3 \big( \frac{\Delta x}2 \big)^{\frac43}$ for the third-order SSP-RK and SSP-MS time discretizations respectively,  
so as to 
match the fourth-order accuracy of spatial DG discretization. 
Table \ref{tab:Ex2Dsmooth} lists the numerical errors at $t=0.2$ in the rest-mass density and the corresponding convergence rates 
for the $P^k$-based IRP DG methods ($k=1,2,3$) 
at different grid resolutions. 
Similar to the 1D case and as also observed in \cite{zhang2012minimum,QinShu2016}, 
the accuracy slightly degenerates for SSP-RK and $k\ge 2$. 
The desired full order of accuracy is observed for the SSP-MS time discretization, confirming that the IRP limiter itself does not destroy the accuracy for smooth solutions as expected.

\begin{table}[htb]
	\centering
	\caption{\small Example 3: 
		Errors at $t=0.2$ in the rest-mass density 
		for the proposed $P^k$-based DG methods ($k=1,2,3$), with the SSP-RK or SSP-MS time discretization,   
		at different spatial grid resolutions.
	}\label{tab:Ex2Dsmooth}
	\begin{tabular}{c|c|c|c|c|c|c|c|c|c}
		\hline
		&     & \multicolumn{4}{c|}{SSP-RK}             & \multicolumn{4}{c}{SSP-MS}             \\ \hline
		$k$                 & $N$   & $l^1$ error & order & $l^2$ error & order & $l^1$ error & order & $l^2$ error & order \\ \hline
		\multirow{5}{*}{1} 
		& 10  & 3.95e-2  & --    & 4.80e-2  & --    & 4.02e-2  & --    & 4.80e-2  & --    \\ 
		& 20  & 7.62e-3  & 2.37  & 1.01e-2  & 2.24  & 7.54e-3  & 2.41  & 1.00e-2  & 2.26  \\ 
		& 40  & 1.65e-3  & 2.21  & 2.25e-3  & 2.17  & 1.65e-3  & 2.20  & 2.24e-3  & 2.15  \\ 
		& 80  & 3.85e-4  & 2.10  & 5.41e-4  & 2.05  & 3.84e-4  & 2.10  & 5.40e-4  & 2.05  \\ 
		& 160 & 9.49e-5  & 2.02  & 1.34e-4  & 2.01  & 9.48e-5  & 2.02  & 1.34e-4  & 2.01  \\ \hline
		\multirow{5}{*}{2} 
		& 10  & 1.14e-2  & --    & 1.46e-2  & --    & 1.06e-2  & --    & 1.34e-2  & --    \\ 
		& 20  & 3.90e-4  & 4.88  & 5.00e-4  & 4.87  & 3.80e-4  & 4.80  & 4.89e-4  & 4.78  \\ 
		& 40  & 4.89e-5  & 3.00  & 6.21e-5  & 3.01  & 4.11e-5  & 3.21  & 5.47e-5  & 3.16  \\ 
		& 80  & 6.55e-6  & 2.90  & 8.62e-6  & 2.85  & 4.90e-6  & 3.07  & 6.72e-6  & 3.03  \\ 
		& 160 & 7.65e-7  & 3.10  & 1.22e-6  & 2.83  & 6.08e-7  & 3.01  & 8.38e-7  & 3.00  \\ \hline
		\multirow{5}{*}{3} 
		& 10  & 2.51e-4  & --    & 3.75e-4  & --    & 2.47e-4  & --    & 3.70e-4  & --    \\ 
		& 20  & 1.82e-5  & 3.79  & 2.65e-5  & 3.82  & 1.81e-5  & 3.77  & 2.68e-5  & 3.79  \\ 
		& 40  & 9.60e-7  & 4.24  & 1.43e-6  & 4.21  & 9.02e-7  & 4.33  & 1.30e-6  & 4.37  \\ 
		& 80  & 6.55e-8  & 3.87  & 1.04e-6  & 3.78  & 5.56e-8  & 4.02  & 7.67e-8  & 4.08  \\ 
		& 160 & 4.51e-9  & 3.86  & 9.14e-9  & 3.51  & 3.46e-9  & 4.01  & 4.71e-9  & 4.03  \\ \hline
	\end{tabular}
\end{table}

\subsection{Example 4: Shock-bubble interaction}

This example simulates the interaction between a planar shock wave and a light bubble within the 
domain $[0,325]\times [-45,45]$. 
The setup is the same as in \cite{he2012adaptive,zhao2013runge}. 
Initially, a left-moving relativistic shock wave is located at $x=265$ 
with the left and right states given by 
\begin{equation*}
(\rho,{\bm v},p)(x,y,0)=
\begin{cases}
(1,~0,~0,~0.05), & \quad x < 265,
\\
(1.865225080631180,~-0.196781107378299,~0,~0.15), & \quad x>265.
\end{cases}
\end{equation*}
A light circular bubble with the radius of $25$ is initially centered at $(215,0)$ in front of the initial shock wave. The fluid state within the bubble is given by 
\begin{equation*}
(\rho,{\bm v},p)(x,y,0)= ( 0.1358,~0,~0,~0.05 ), \qquad \sqrt{ (x-215)^2+y^2 } \le 25.
\end{equation*}
The reflective conditions are specified at both the top and bottom boundaries $\{y=\pm45, 0 \le x \le 325 \}$, and the inflow (resp. outflow) boundary condition is enforced at the right (resp. left) boundary.

\begin{figure}[htbp]
	\centering
	\begin{subfigure}[b]{0.49\textwidth}
		\begin{center}
			\includegraphics[width=0.99\linewidth]{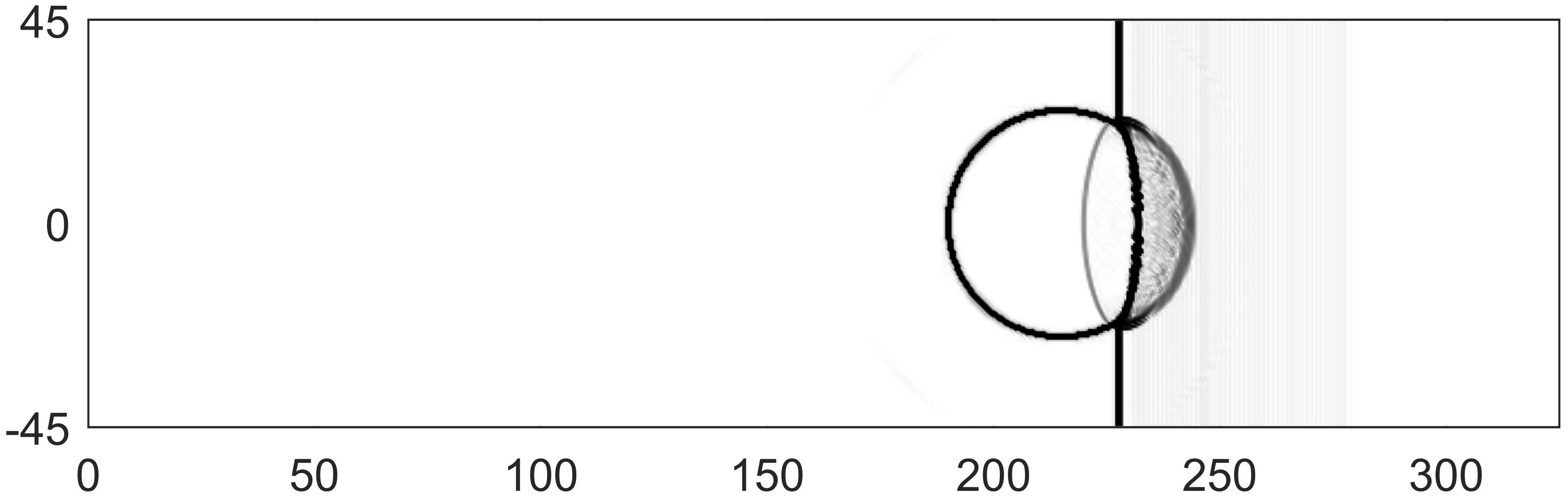}
		\end{center}
	\end{subfigure}
	\begin{subfigure}[b]{0.49\textwidth}
		\begin{center}
			\includegraphics[width=0.99\linewidth]{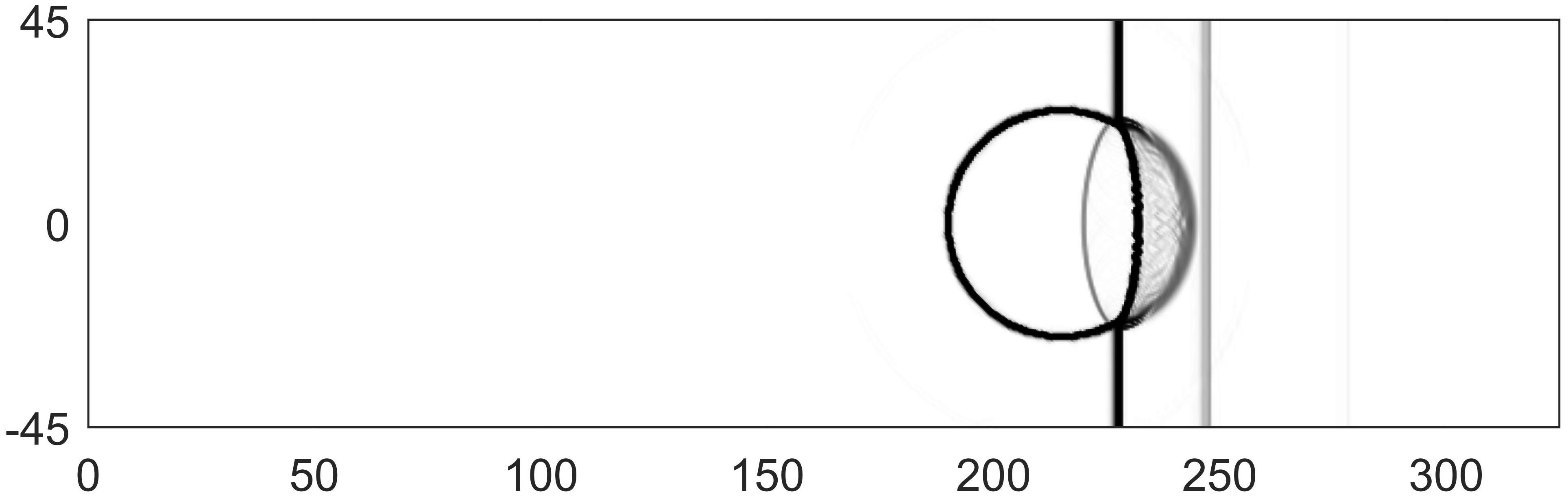}
		\end{center}
	\end{subfigure}
	\begin{subfigure}[b]{0.49\textwidth}
		\begin{center}
			\includegraphics[width=0.99\linewidth]{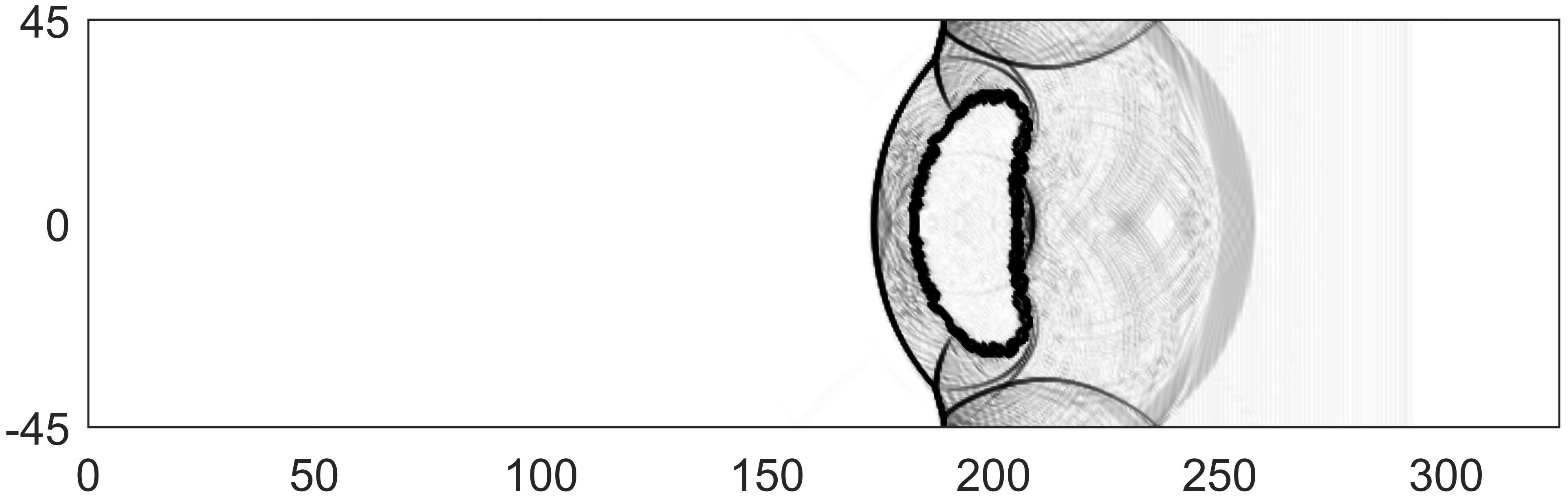}
		\end{center}
	\end{subfigure}
	\begin{subfigure}[b]{0.49\textwidth}
		\begin{center}
			\includegraphics[width=0.99\linewidth]{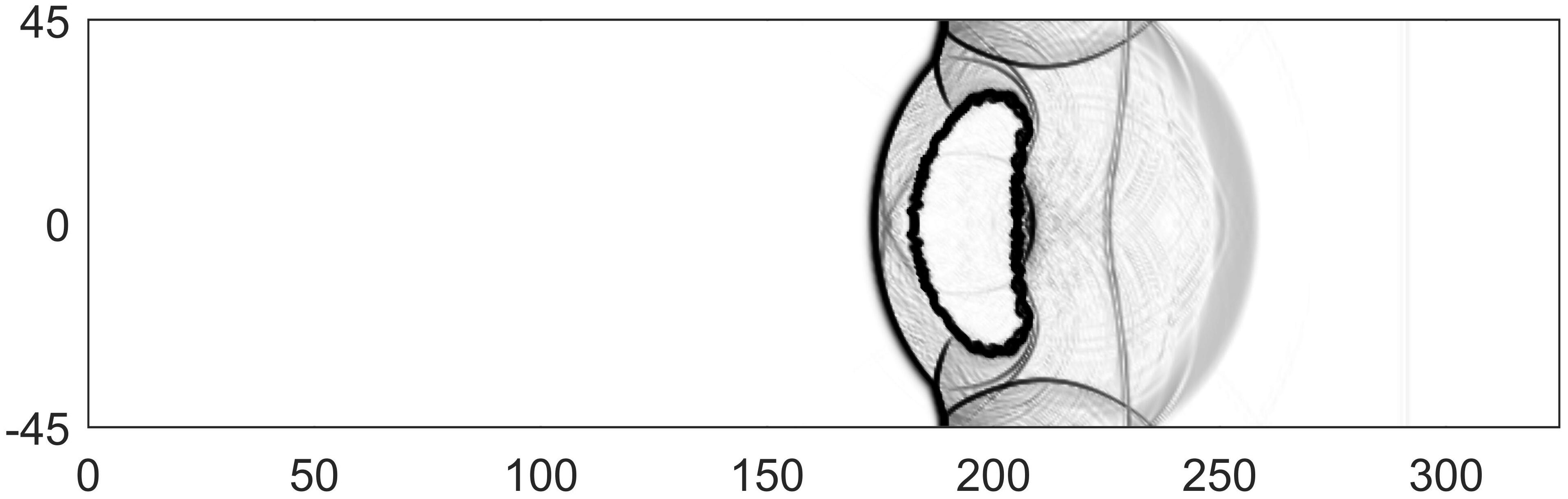}
		\end{center}
	\end{subfigure}
	\begin{subfigure}[b]{0.49\textwidth}
		\begin{center}
			\includegraphics[width=0.99\linewidth]{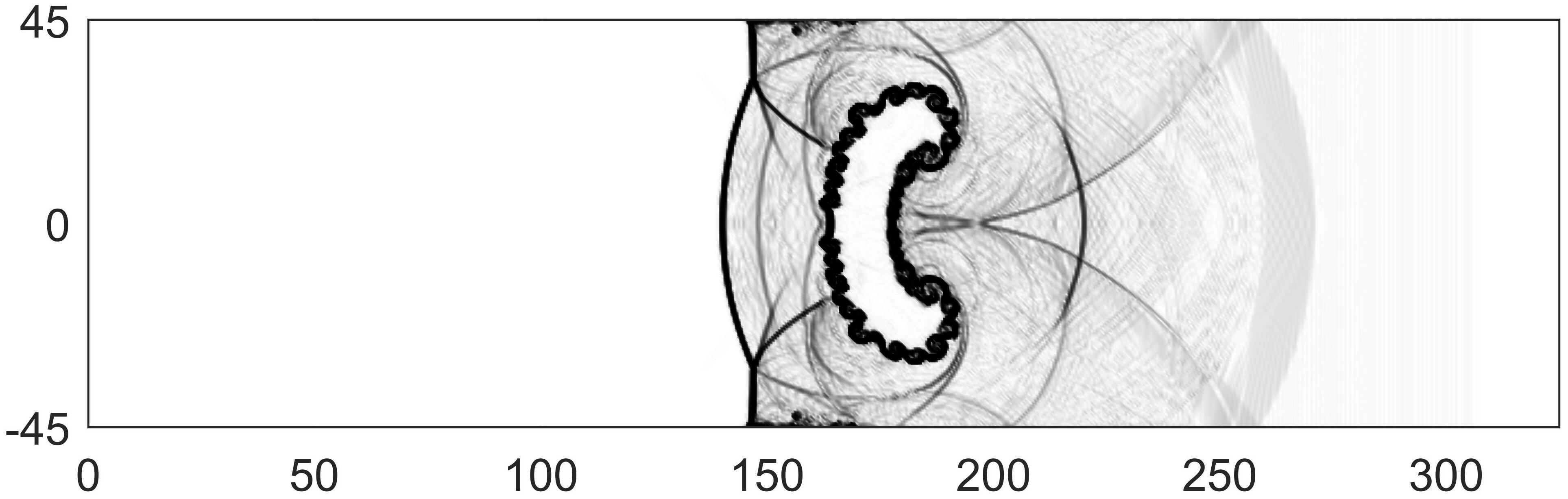}
		\end{center}
	\end{subfigure}
	\begin{subfigure}[b]{0.49\textwidth}
		\begin{center}
			\includegraphics[width=0.99\linewidth]{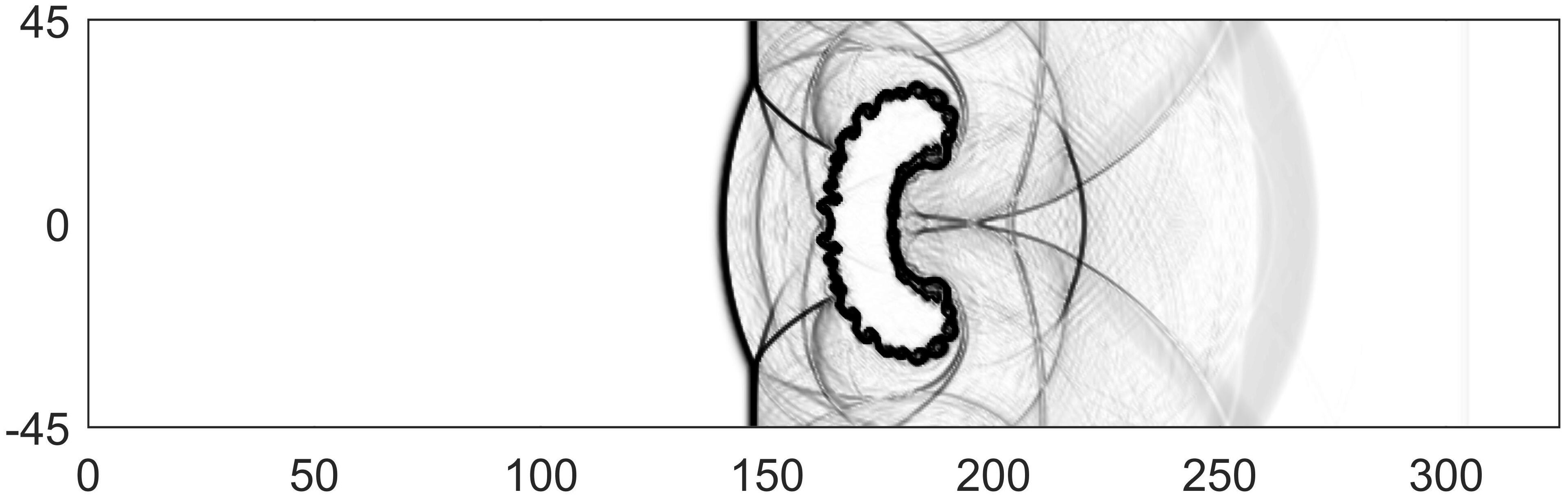}
		\end{center}
	\end{subfigure}
	\begin{subfigure}[b]{0.49\textwidth}
	\begin{center}
		\includegraphics[width=0.99\linewidth]{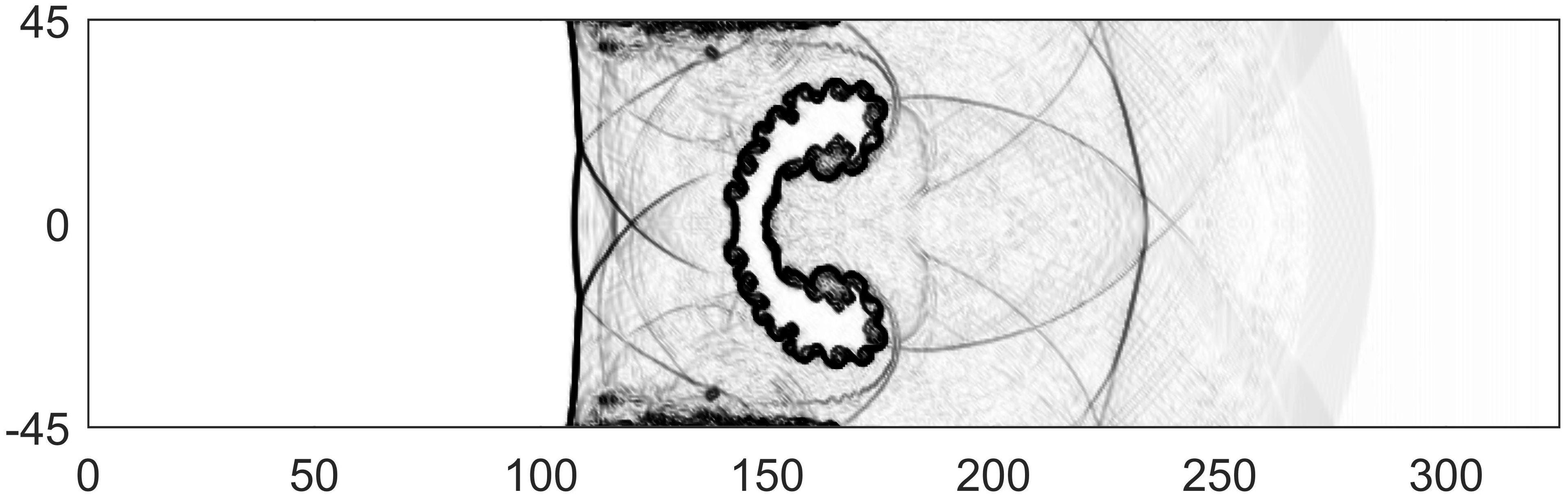}
	\end{center}
	\end{subfigure}
	\begin{subfigure}[b]{0.49\textwidth}
		\begin{center}
			\includegraphics[width=0.99\linewidth]{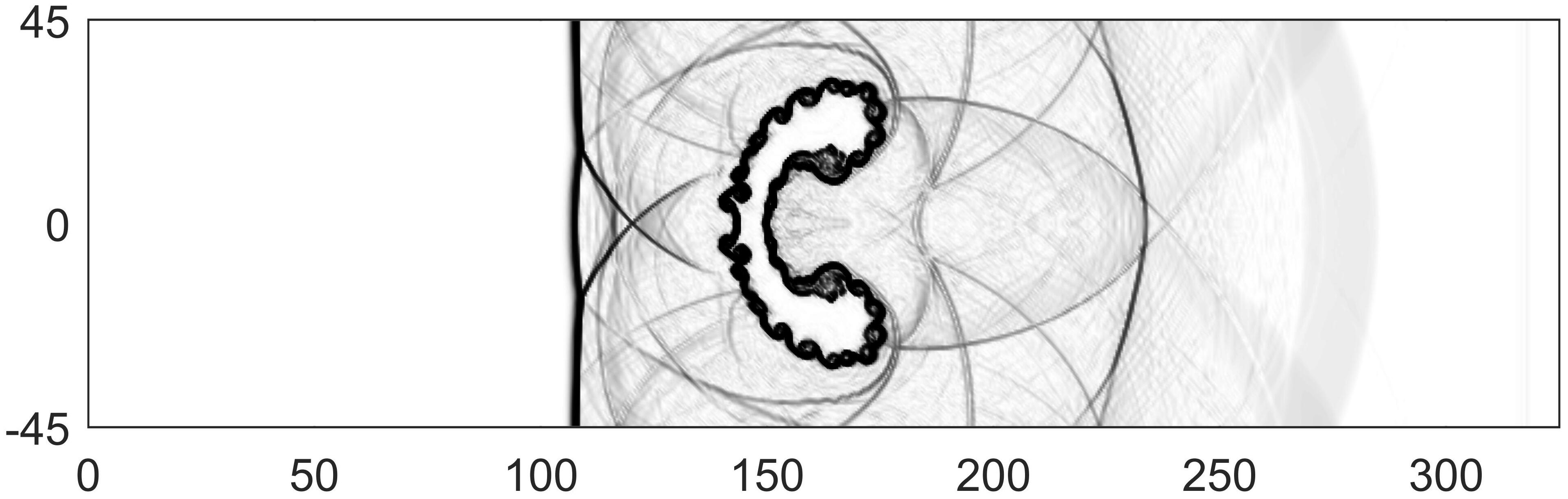}
		\end{center}
	\end{subfigure}
	\begin{subfigure}[b]{0.49\textwidth}
		\captionsetup{width=.8\linewidth}
	\begin{center}
		\includegraphics[width=0.99\linewidth]{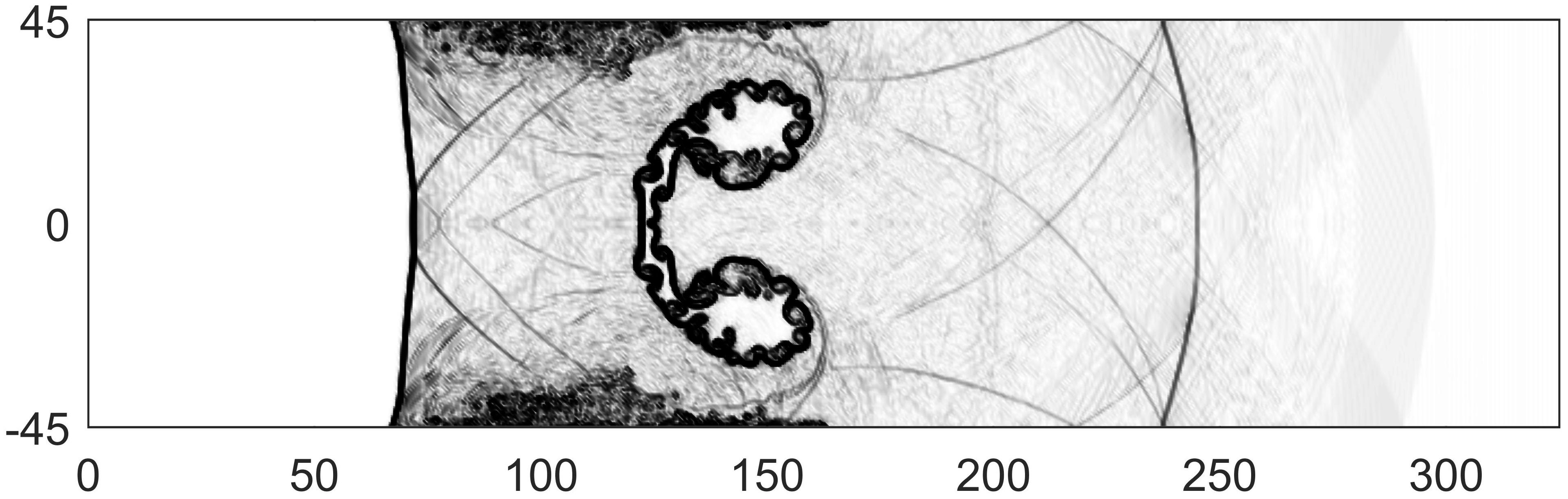}
	\end{center}
\caption{\small With the BP limiter}
	\end{subfigure}
	\begin{subfigure}[b]{0.49\textwidth}
		\captionsetup{width=.8\linewidth}
		\begin{center}
			\includegraphics[width=0.99\linewidth]{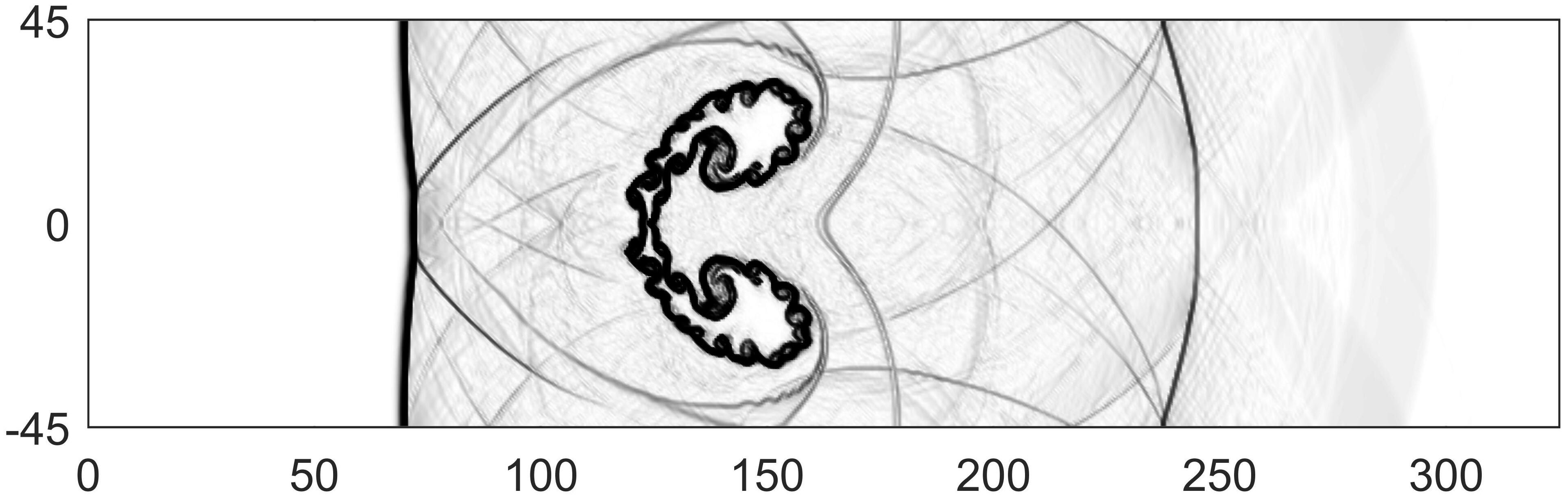}
		\end{center}
	\caption{\small With the IRP limiter}
	\end{subfigure}
	\caption{\small Example 4: Schlieren images of the density 
		 at $t=90$, $180$, $270$, $360$, and $450$ (from top to bottom) obtained by 
		the fourth-order DG methods on the mesh of $650 \times 180$ cells. (Here we do not use any other non-oscillatory limiters, e.g.~TVD/TVB or WENO limiters.)} 
	\label{fig:SB}
\end{figure}

In this simulation, we also do not use any other non-oscillatory limiters, e.g.~TVD/TVB or WENO limiters. 
Figure \ref{fig:SB} shows the numerical results 
obtained by the fourth-order DG methods, with the BP limiter or the IRP limiter respectively, on a mesh of $650 \times 180$ uniform cells. 
We observe serious oscillations developing near the top and bottom boundaries in 
the DG solution with only the BP limiter. 
However, if the IRP limiter is used, the undesirable oscillations are almost oppressed, and the discontinuities and some small wave structures including the curling of the bubble interface are well resolved. 
To check the preservation of minimum entropy principle, 
we plot the time evolution of the minimum
specific entropy values of the DG solutions in Figure \ref{fig:SB_MinS}. 
It shows that 
the minimum stays the same for the DG scheme with the IRP limiter, which indicates that the 
minimum entropy principle is maintained, while the DG scheme with only the BP limiter does not preserve the principle.

\begin{figure}[htbp]
		\begin{subfigure}[b]{0.49\textwidth}
		\begin{center}
			\includegraphics[width=0.99\linewidth]{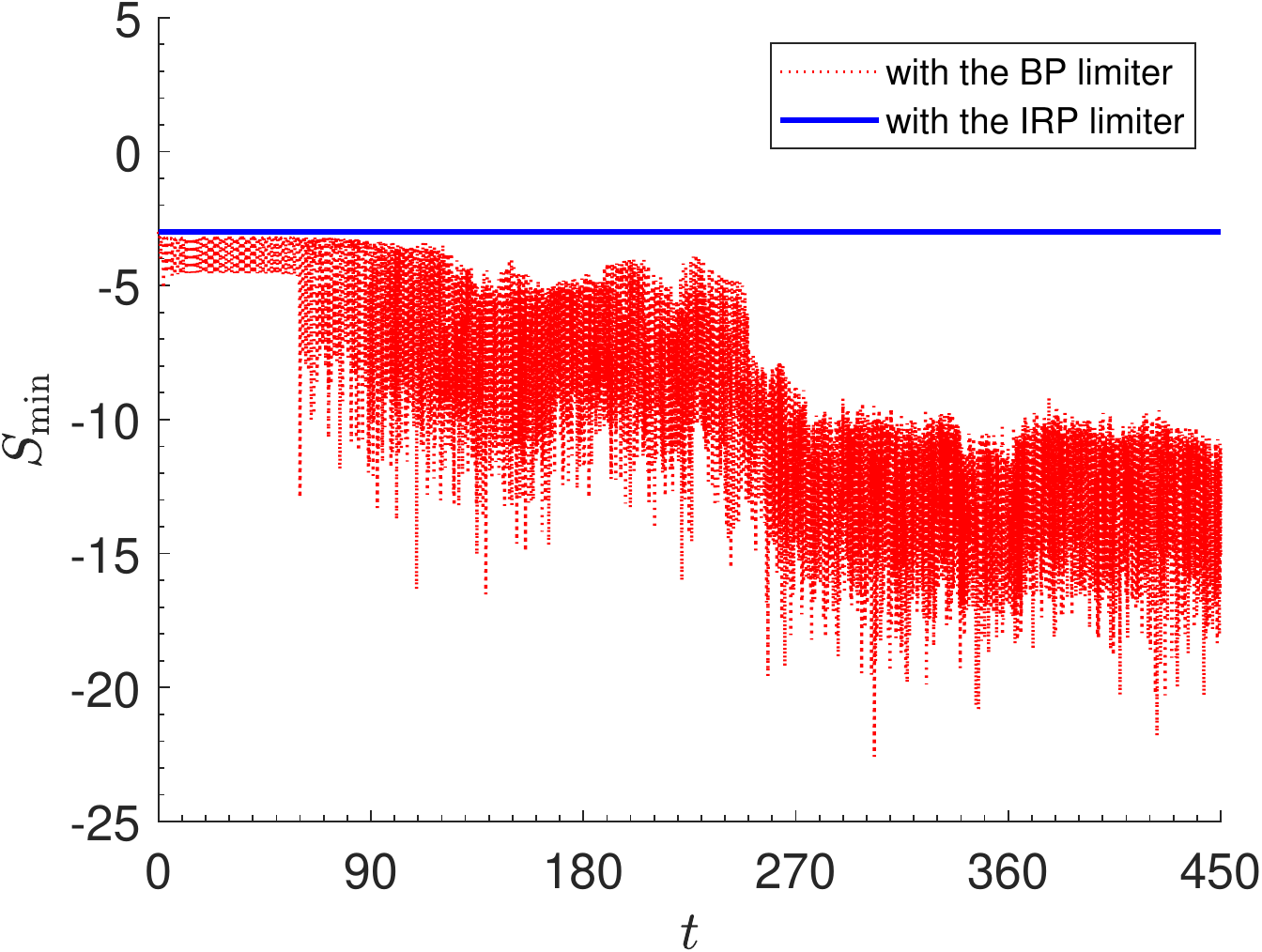}
		\end{center}
	\end{subfigure}
	\begin{subfigure}[b]{0.49\textwidth}
		\begin{center}
			\includegraphics[width=0.99\linewidth]{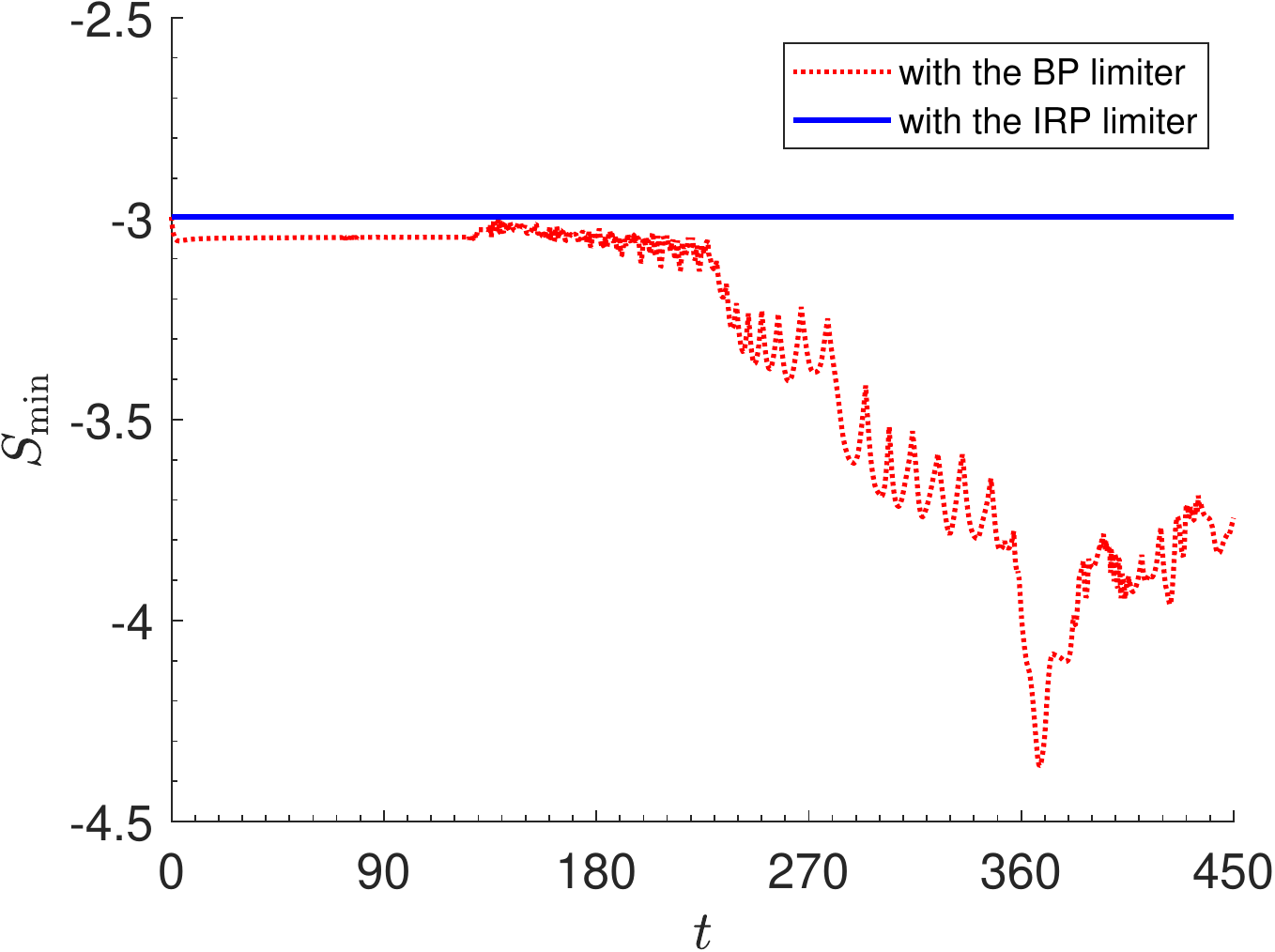}
		\end{center}
	\end{subfigure}
	\caption{\small Example 4: Time evolution of $S_{\min}(t) $ for the DG solutions with the IRP limiter or with the BP limiter. Left: $S_{\min}(t) = \min_{K \in {\mathcal T}_h} \min_{ {\bm x} \in {\mathbb X}_K } S({\bf U}_h({\bm x},t))$; right: $S_{\min}(t) = \min_{K \in {\mathcal T}_h} S(\overline{\bf U}_K(t))$.   }
		\label{fig:SB_MinS}
\end{figure}

\subsection{Example 5: Two 2D Riemann problems} 
In this test, we simulate two Riemann problems of the 2D RHD equations, 
which have become benchmark tests for checking the accuracy and resolution of 2D RHD codes (cf.~\cite{he2012adaptive,zhao2013runge,WuTang2015,WuTang2017ApJS,bhoriya2020entropy}).

\begin{figure}[htbp]
	\centering
		\begin{subfigure}[b]{0.49\textwidth}
		\captionsetup{width=.85\linewidth}
		\begin{center}
			\includegraphics[width=0.99\linewidth]{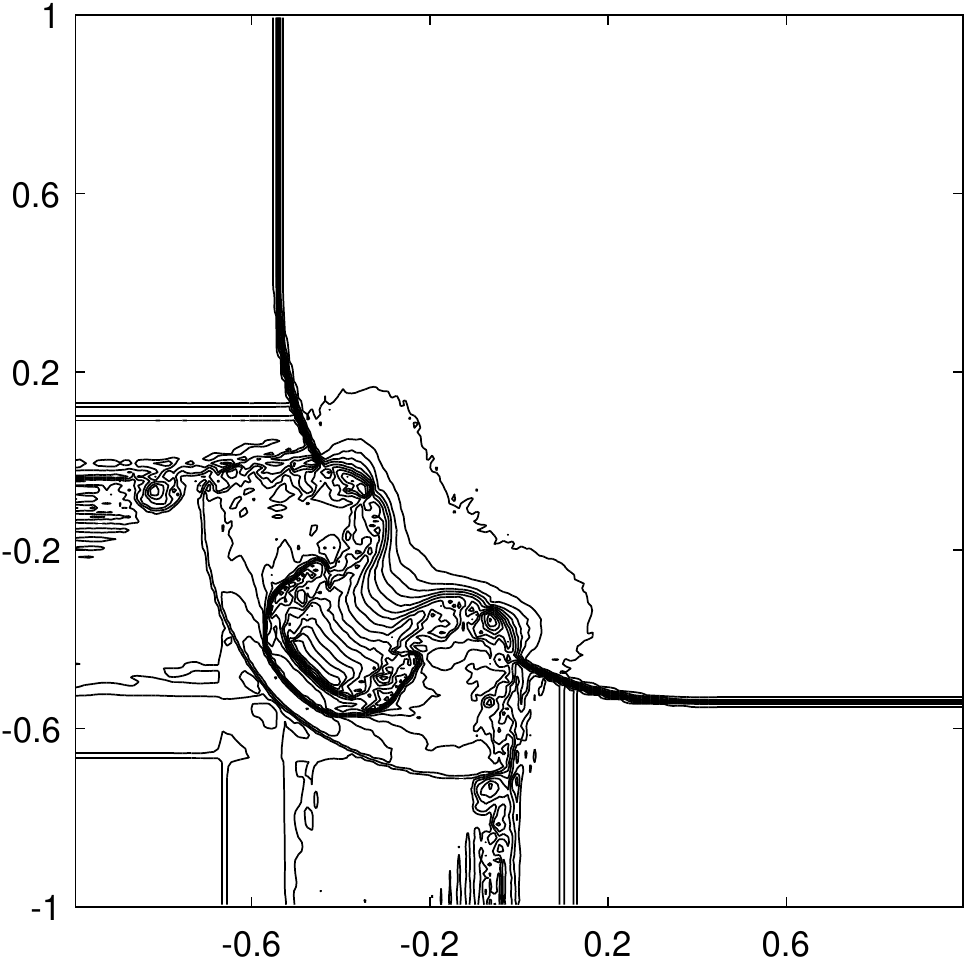}
		\end{center}
		\caption{With the BP limiter (without any other non-oscillatory limiters)}
	\end{subfigure}
	\begin{subfigure}[b]{0.49\textwidth}
		\captionsetup{width=.85\linewidth}
		\begin{center}
			\includegraphics[width=0.99\linewidth]{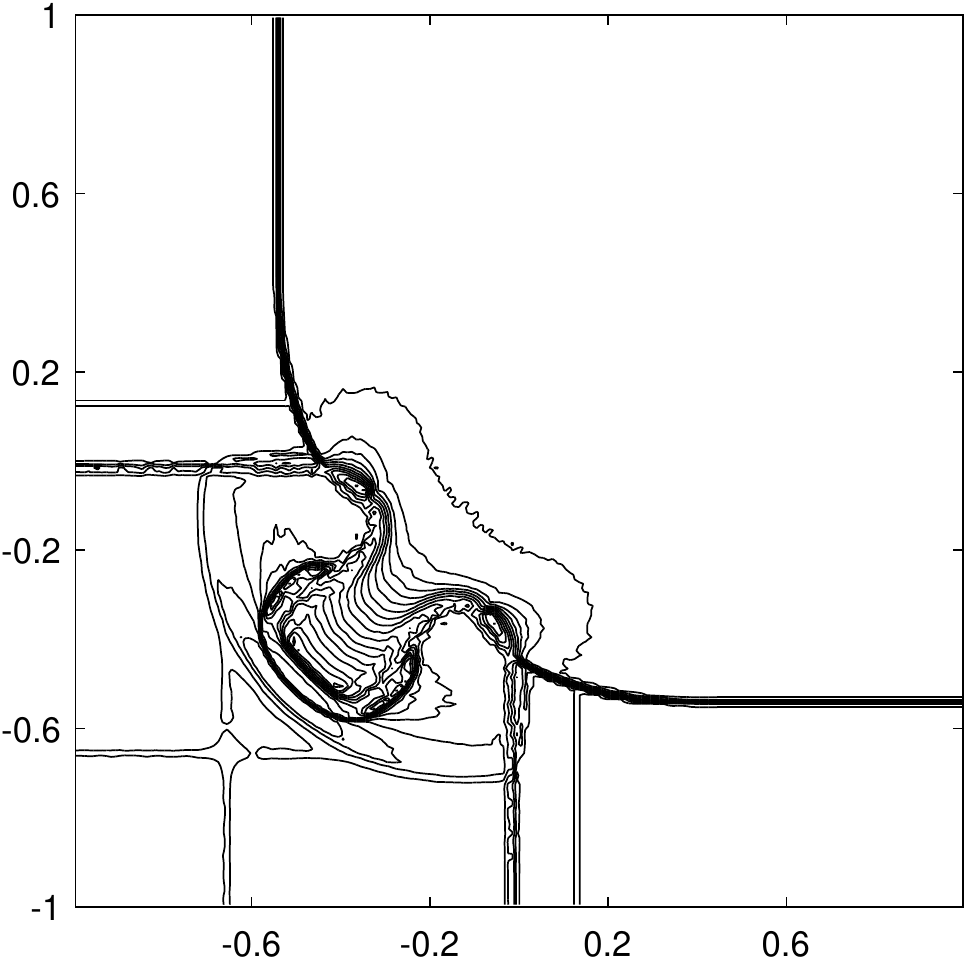}
		\end{center}
		\caption{With the IRP limiter (without any other non-oscillatory limiters)}
	\end{subfigure}
	\caption{\small The first Riemann problem of Example 5: 
		The contours of the rest-mass density logarithm at $t=0.8$ obtained by using 
		the fourth-order DG methods with the BP limiter (left) or with the proposed IRP limiter (right), on the mesh of $200 \times 200$ cells. (18 equally spaced contour lines from $-7.8981$ to $-2.5631$ are displayed.)} 
	\label{fig:2DRP2solu}
\end{figure}

\begin{figure}[htbp]
	\begin{subfigure}[b]{0.49\textwidth}
		\begin{center}
			\includegraphics[width=0.99\linewidth]{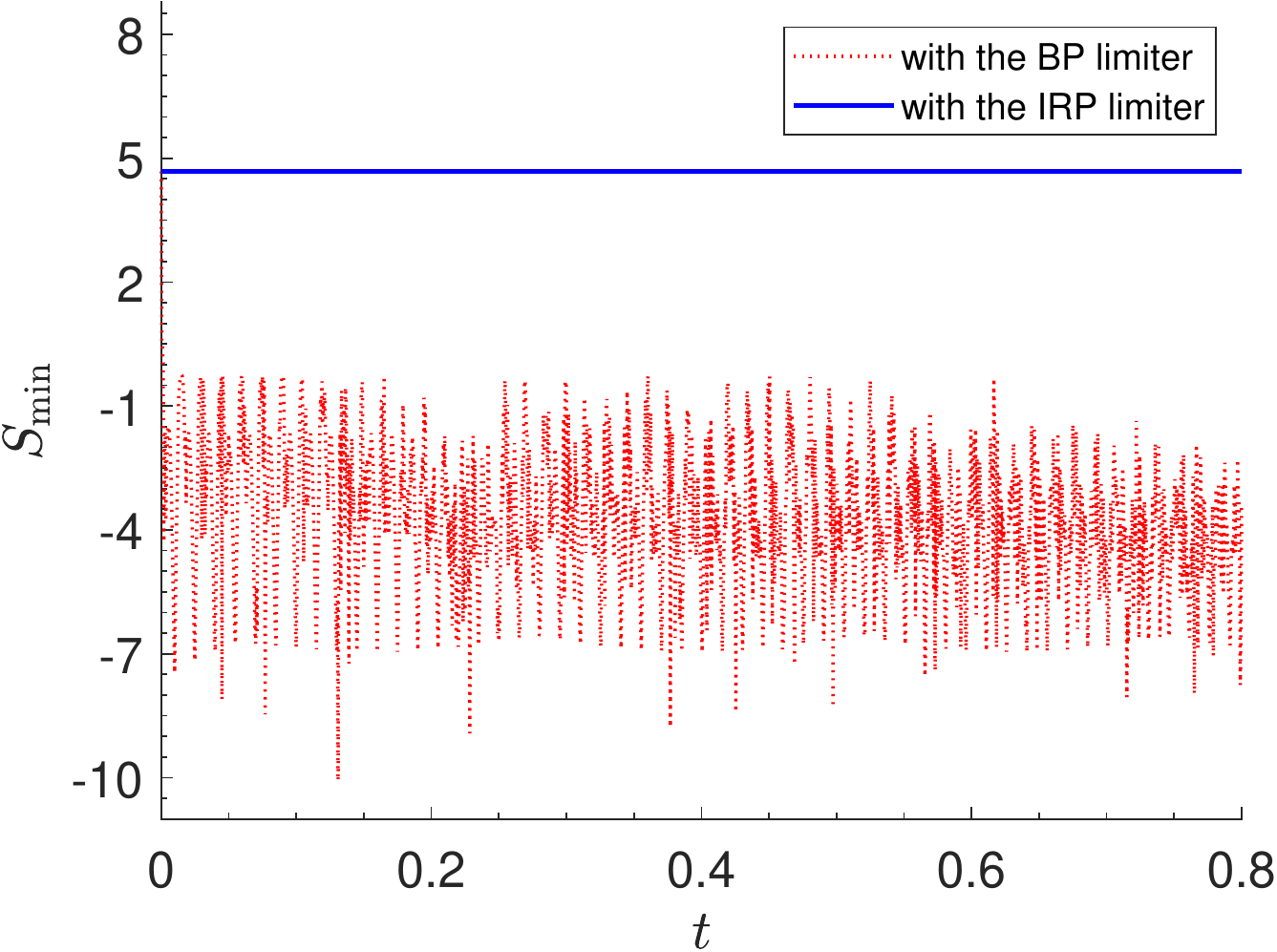}
		\end{center}
	\end{subfigure}
	\begin{subfigure}[b]{0.49\textwidth}
		\begin{center}
			\includegraphics[width=0.99\linewidth]{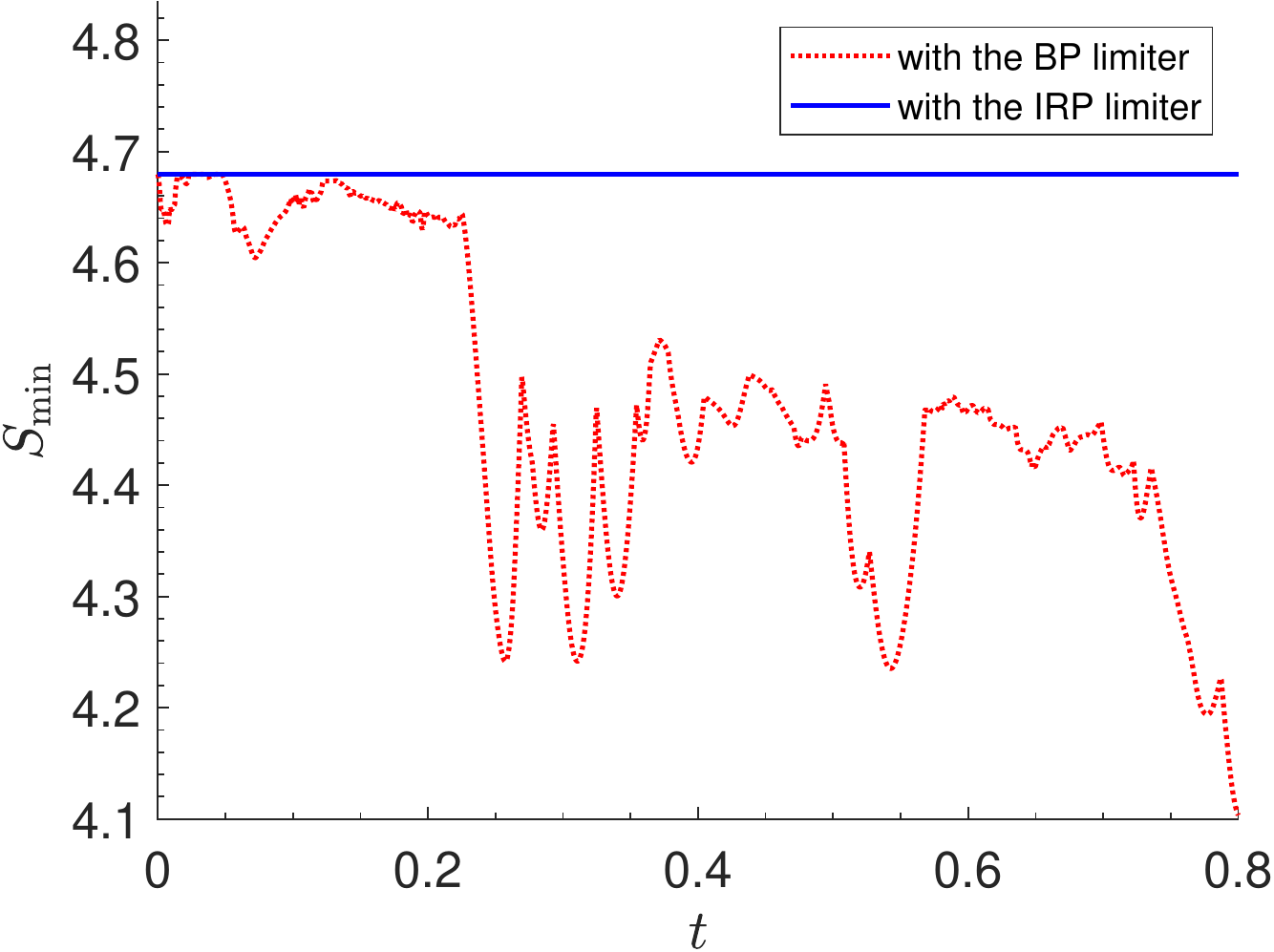}
		\end{center}
	\end{subfigure}
	\caption{\small The first Riemann problem of Example 5: $S_{\min}(t) $ for the DG solutions with the IRP limiter or with the BP limiter. Left: $S_{\min}(t) = \min_{K \in {\mathcal T}_h} \min_{ {\bm x} \in {\mathbb X}_K } S({\bf U}_h({\bm x},t))$; right: $S_{\min}(t) = \min_{K \in {\mathcal T}_h} S(\overline{\bf U}_K(t))$.   }
	\label{fig:2DRP2_MinS}
\end{figure}

The first Riemann problem was proposed in \cite{WuTang2015}, with the initial data given by  
\begin{equation*}
	(\rho,{\bm v},p)(x,y,0)	=
	\begin{cases}(0.1,0,0,20)^\top,& x>0,~y>0,\\
		(0.00414329639576,0.9946418833556542,0,0.05)^\top,&    x<0,~y>0,\\
		(0.01,0,0,0.05)^\top,&      x<0,~y<0,\\
		(0.00414329639576,0,0.9946418833556542, 0.05)^\top,&    x>0,~y<0,
	\end{cases}
\end{equation*}
where the left and lower initial discontinuities are contact discontinuities,
and the upper and right are shock waves. 
Because the maximal value of the fluid velocity is very close to the speed of light ($c=1$), nonphysical numerical solutions can be easily produced in the simulation, making this test challenging. 
For testing purpose, we do {\em not} use any other non-oscillatory limiters, e.g.~TVD/TVB or WENO limiters. 
We evolve the solution up to $t=0.8$ on the mesh of $200 \times 200$ cells within the domain $[-1,1]^2$. 
The contours of $\log(\rho)$ are displayed in  
Figure \ref{fig:2DRP2solu}, obtained by the fourth-order DG methods, with the BP or IRP limiter respectively. 
Serious oscillations are observed in the DG solution with only the BP limiter, while, if 
the IRP limiter is applied (i.e., the entropy limiter is added), the oscillations are much reduced. 
The minimum entropy principle of the DG solution with the IRP limiter 
is validated in Figure \ref{fig:2DRP2_MinS}. 
As the numerical solution is preserved in the set $\Omega_{S_0}$, 
the IRP DG scheme exhibits strong robustness in such ultra-relativistic flow simulation; the computed flow structures agree well with those reported in \cite{WuTang2015,bhoriya2020entropy}. 
We remark that if the BP or IRP limiter is turned off, the DG code would break down  
due to nonphysical numerical solutions violating the physical constraints \eqref{eq:PHYSconstraints}.

\begin{figure}[htbp]
	\centering
	\begin{subfigure}[b]{0.49\textwidth}
		\captionsetup{width=.88\linewidth}
		\begin{center}
			\includegraphics[width=0.99\linewidth]{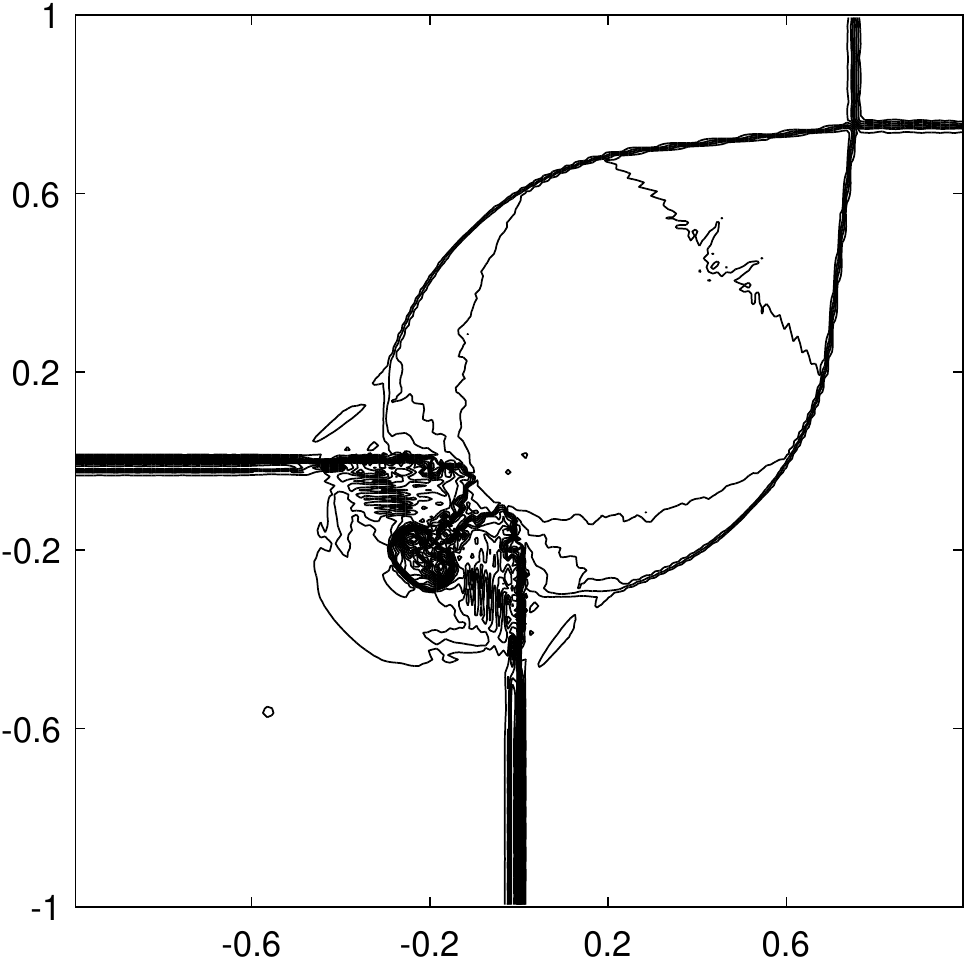}
		\end{center}
	\caption{With the BP limiter (without any other non-oscillatory limiters); $200\times 200$ cells}
	\end{subfigure}
	\begin{subfigure}[b]{0.49\textwidth}
	\captionsetup{width=.88\linewidth}
	\begin{center}
		\includegraphics[width=0.99\linewidth]{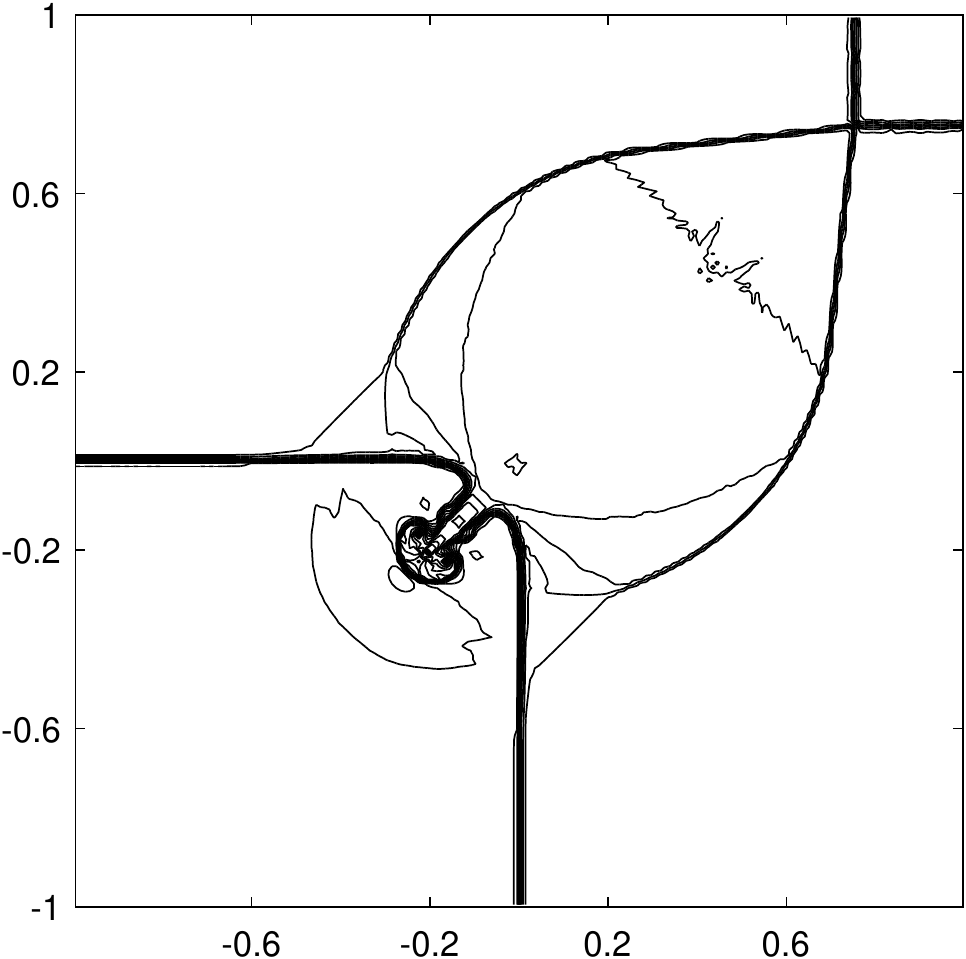}
	\end{center}
	\caption{With the IRP limiter (without any other non-oscillatory limiters); $200\times 200$ cells}
\end{subfigure}
\\
\vspace{4mm}
	\begin{subfigure}[b]{0.49\textwidth}
		\captionsetup{width=.88\linewidth}
		\begin{center}
			\includegraphics[width=0.99\linewidth]{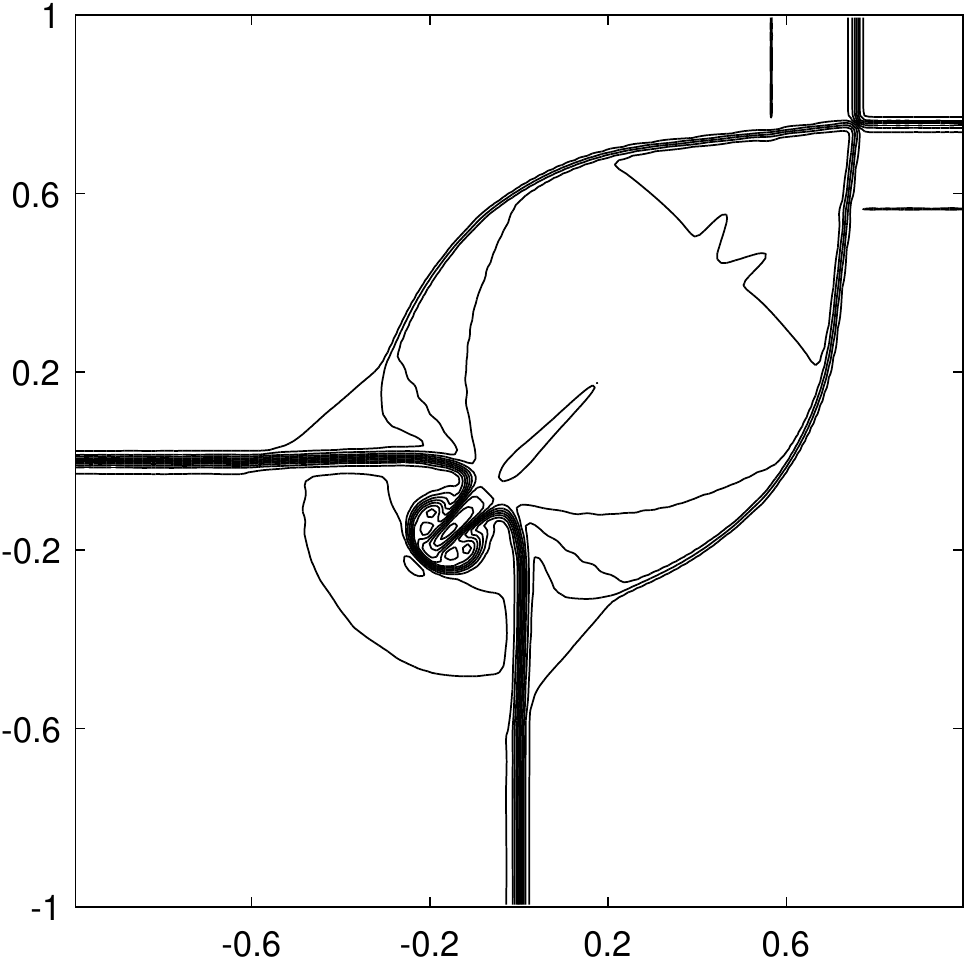}
		\end{center}
	\caption{With the WENO limiter; $200\times 200$ cells}
	\end{subfigure}
	\begin{subfigure}[b]{0.49\textwidth}
		\captionsetup{width=.88\linewidth}
		\begin{center}
			\includegraphics[width=0.99\linewidth]{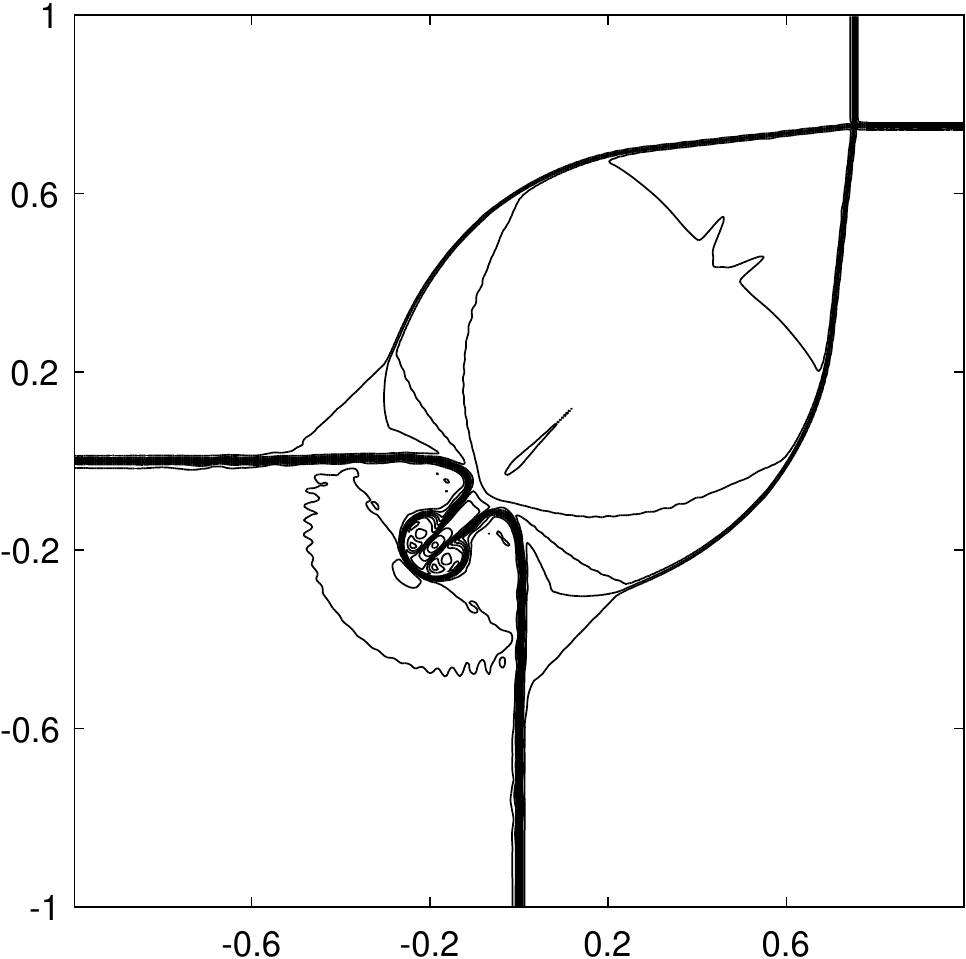}
		\end{center}
	\caption{With the WENO limiter; $400\times 400$ cells}
	\end{subfigure}
	\caption{\small The second Riemann problem of Example 5: 
		The contours of the rest-mass density logarithm at $t=0.8$ obtained by using 
		the fourth-order DG methods with different limiters. (20 equally spaced contour lines from $-3.2533$ to $-0.426$ are displayed.)} 
	\label{fig:2DRP1solu}
\end{figure}

\begin{figure}[htbp]
	\begin{subfigure}[b]{0.49\textwidth}
		\begin{center}
			\includegraphics[width=0.99\linewidth]{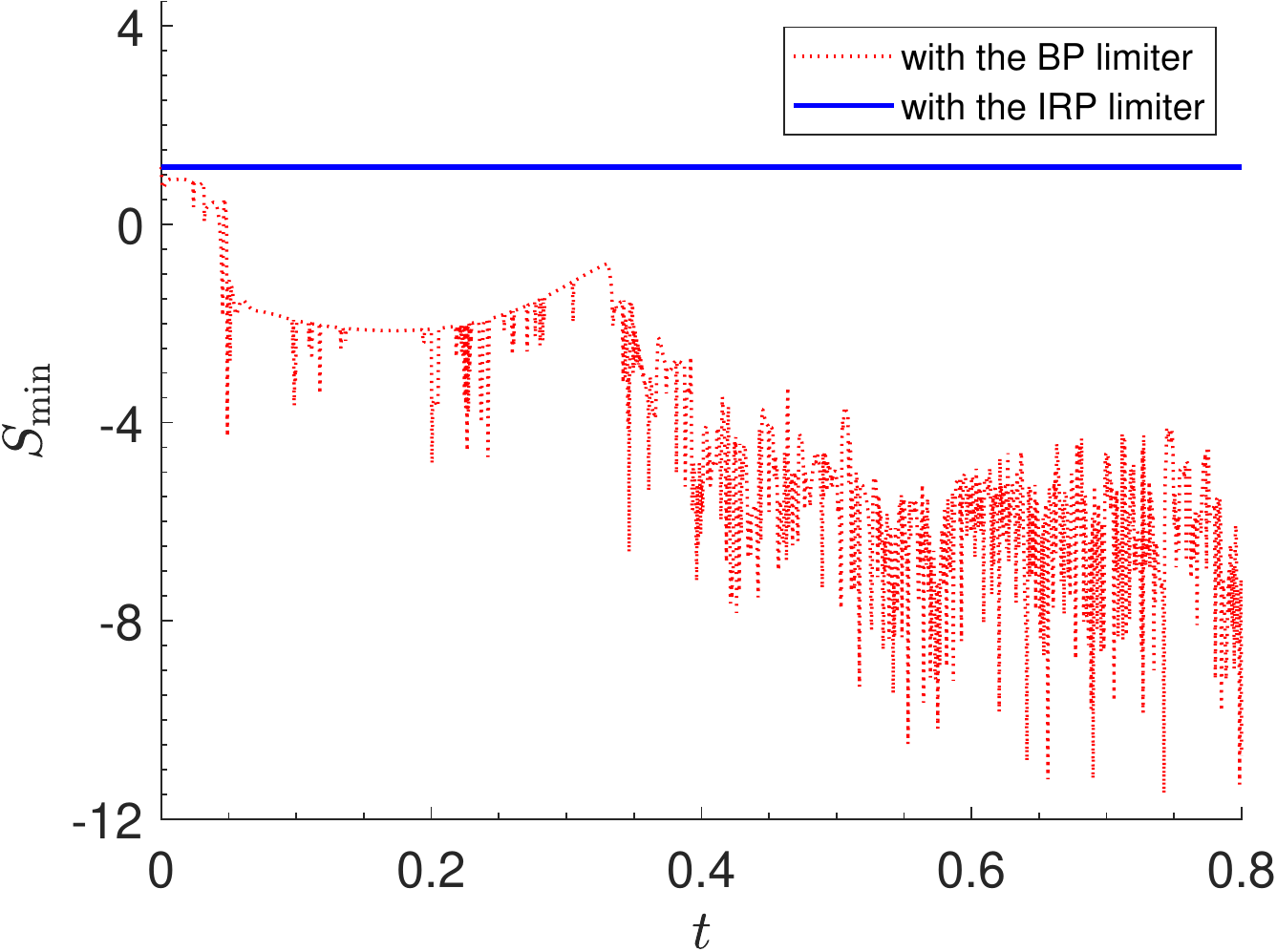}
		\end{center}
	\end{subfigure}
	\begin{subfigure}[b]{0.49\textwidth}
		\begin{center}
			\includegraphics[width=0.99\linewidth]{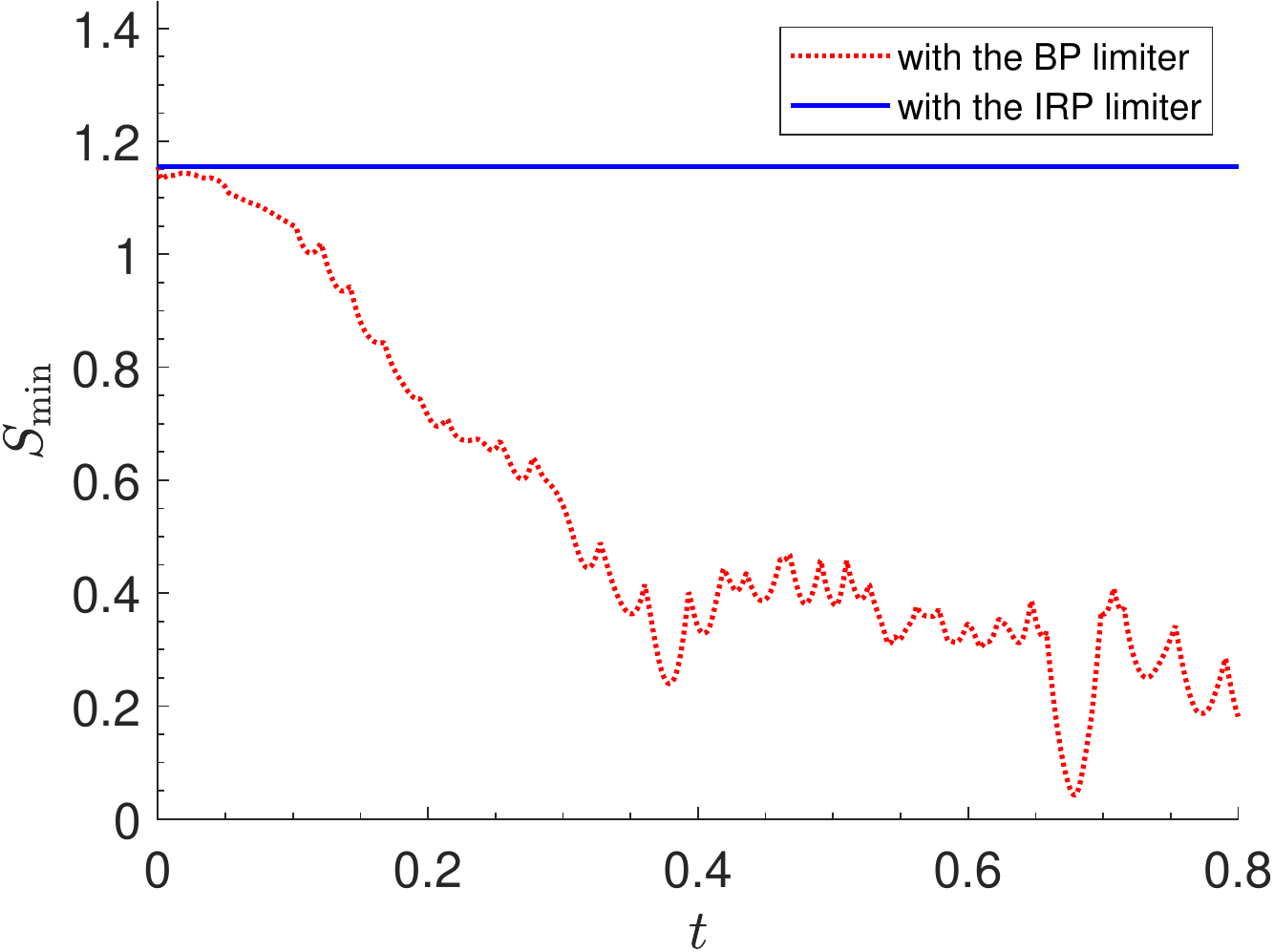}
		\end{center}
	\end{subfigure}
	\caption{\small The second Riemann problem of Example 5: $S_{\min}(t) $ for the DG solutions with the IRP limiter or with the BP limiter. Left: $S_{\min}(t) = \min_{K \in {\mathcal T}_h} \min_{ {\bm x} \in {\mathbb X}_K } S({\bf U}_h({\bm x},t))$; right: $S_{\min}(t) = \min_{K \in {\mathcal T}_h} S(\overline{\bf U}_K(t))$.   }
	\label{fig:2DRP1_MinS}
\end{figure}

The second Riemann problem was first proposed in \cite{he2012adaptive}, and its initial condition is given by  
\begin{equation*}
	(\rho,{\bm v},p)(x,y,0)	=
	\begin{cases}(0.035145216124503, 0, 0, 0.162931056509027)^\top,& x>0,~y>0,\\
		(0.1, 0.7, 0, 1)^\top,&    x<0,~y>0,\\
		(0.5, 0, 0, 1)^\top,&      x<0,~y<0,\\
		(0.1, 0, 0.7, 1)^\top,&    x>0,~y<0. 
	\end{cases}
\end{equation*}
Figures \ref{fig:2DRP1solu}(a) and \ref{fig:2DRP1solu}(b) display 
the contours of $\log(\rho)$ at $t=0.8$, obtained by the fourth-order DG methods, with the BP or IRP limiter respectively 
and without any non-oscillatory limiters, 
on the mesh of $200 \times 200$ cells within the domain $[-1,1]^2$. 
For comparison and reference, we also present in Figures \ref{fig:2DRP1solu}(c--d) the numerical results obtained by the fourth-order DG method with only the WENO limiter \cite{zhao2013runge} on two meshes of $200 \times 200$ cells and $400 \times 400$ cells, respectively. 
(The WENO limiter is 
applied only within some ``trouble'' cells adaptively identified  by the KXRCF indicator 
 \cite{Krivodonova}.) 
We can see that the result with only the BP limiter is oscillatory, while enforcing the minimum entropy principle by the IRP limiter helps to damp some of the oscillations.  
Although the WENO limiter may completely suppress the undesirable oscillations, 
the resulting numerical solution in Figure \ref{fig:2DRP1solu}(c) is much more dissipative than that computed with the IRP limiter in Figure \ref{fig:2DRP1solu}(b).
Figure \ref{fig:2DRP1_MinS} shows the time evolution of the minimum
specific entropy values of the DG solutions. One can see that 
the minimum remains the same for the DG scheme with the IRP limiter. This implies the preservation of 
minimum entropy principle, which, however, is not ensured by using either only the BP limiter or the WENO limiter.

From the above numerical results, we observe that the IRP limiter 
helps to preserve numerical solutions in the invariant region $\Omega_{S_0}$ 
and to damp some numerical oscillations while keeping the high resolution.  
However, as pointed out in \cite{jiang2018invariant}, the IRP limiter is still a mild limiter and may not completely suppress all the oscillations. 
For problems involving strong shocks, one may 
use the IRP limiter together with a non-oscillatory limiter (e.g., the WENO limiter), to sufficiently control all the undesirable oscillations while also preserving the invariant region $\Omega_{S_0}$.

\subsection{Example 6: Two astrophysical jets}
We now further examine the robustness and the IRP property of our DG method  
by simulating two relativistic jets. For high-speed jet problems,  
the internal energy is exceedingly 
small compared to the kinetic energy so that  
negative numerical pressure can be generated easily in the simulations.
Since shear waves, 
strong
relativistic shock waves, and 
ultra-relativistic regions are involved in this challenging test, we use the IRP or BP limiter 
together with the WENO limiter to preserve the physical constraints \eqref{eq:PHYSconstraints} 
and suppress all the undesirable oscillations.

\begin{figure}[htbp]
	\centering
	\begin{subfigure}[b]{0.32\textwidth}
		\begin{center}
			\includegraphics[width=1.0\linewidth]{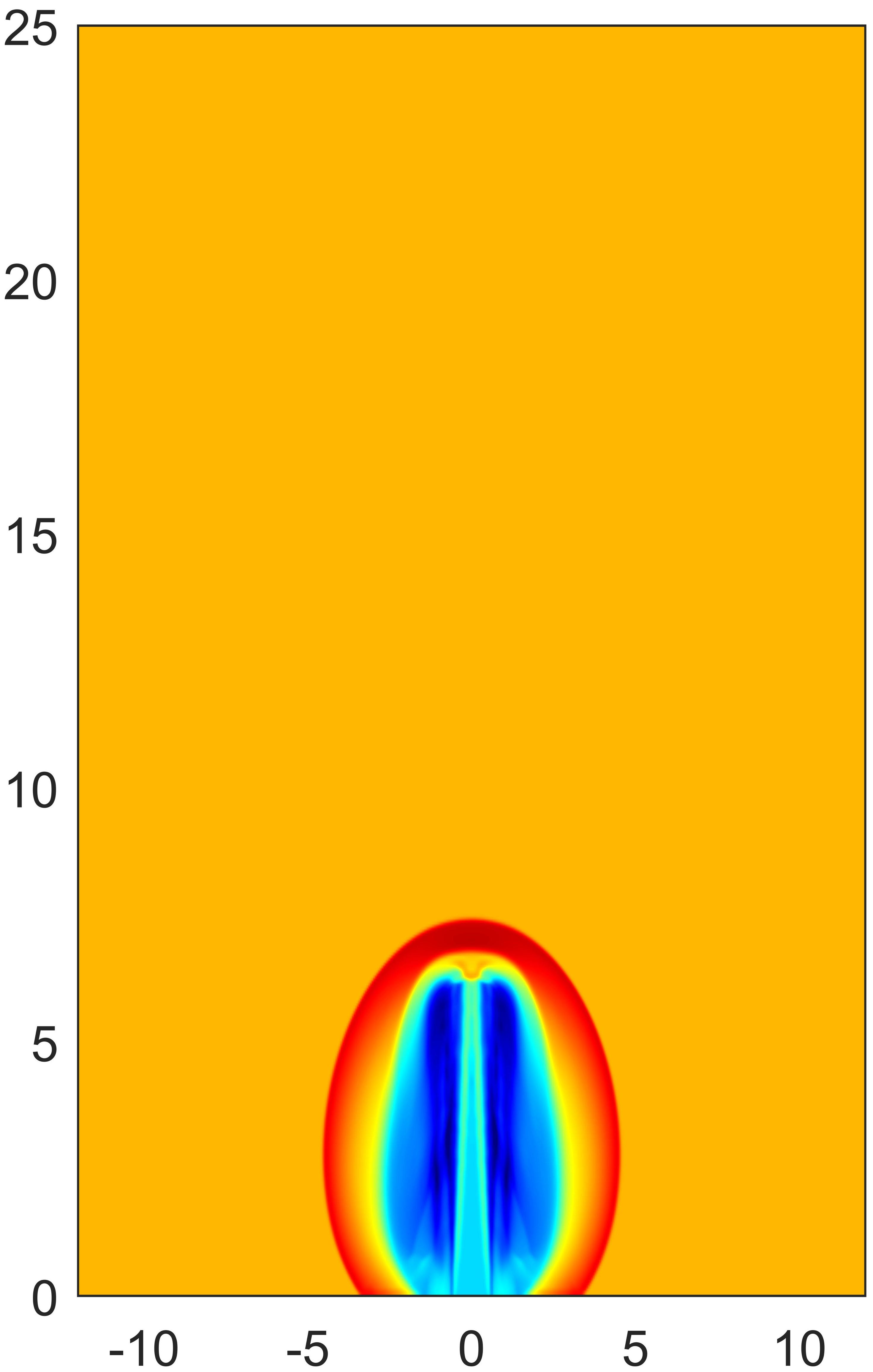}
		\end{center}
	\end{subfigure}
	\begin{subfigure}[b]{0.32\textwidth}
		\begin{center}
			\includegraphics[width=1.0\linewidth]{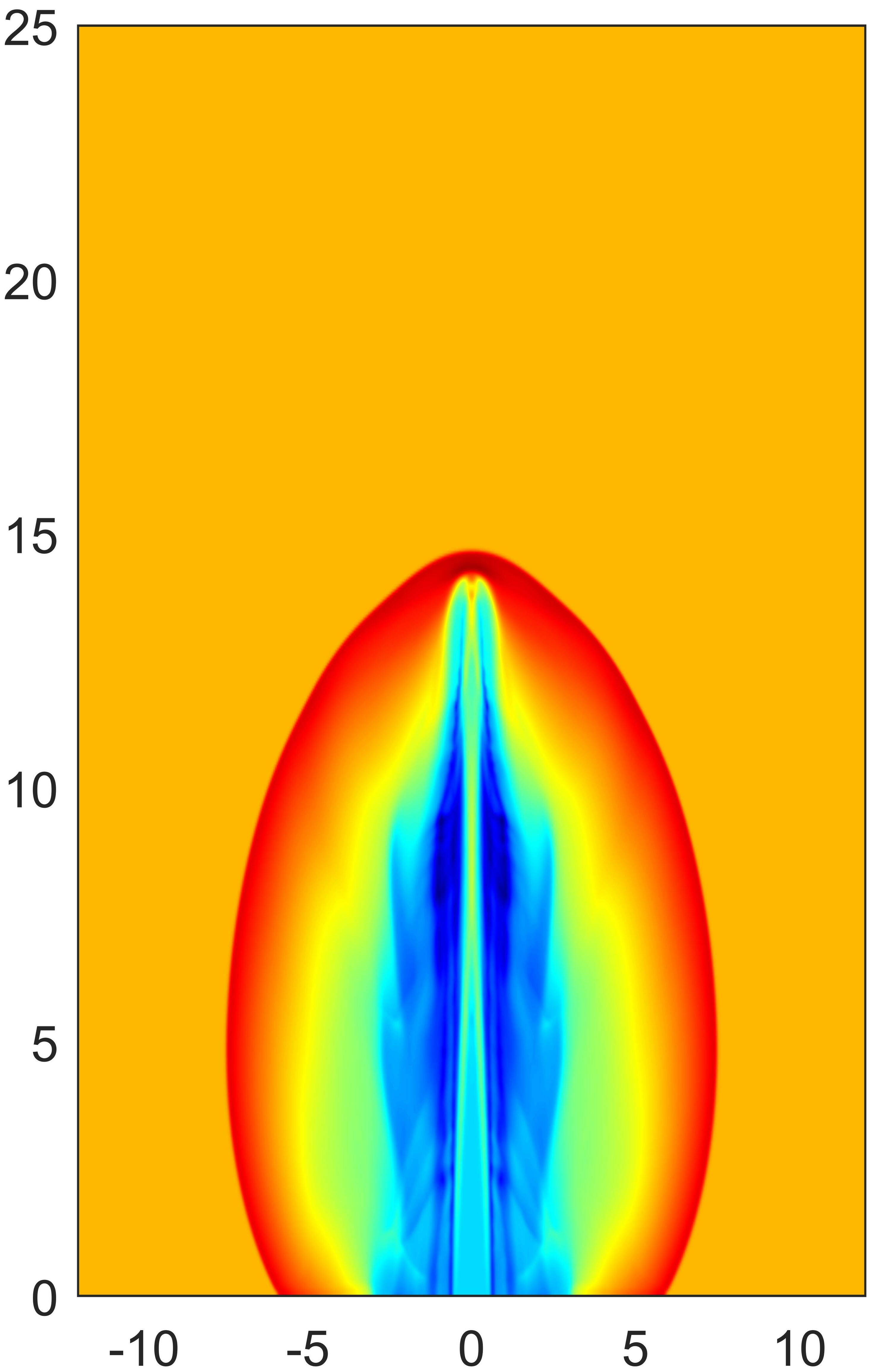}
		\end{center}
	\end{subfigure}
	\begin{subfigure}[b]{0.32\textwidth}
		\begin{center}
			\includegraphics[width=1.0\linewidth]{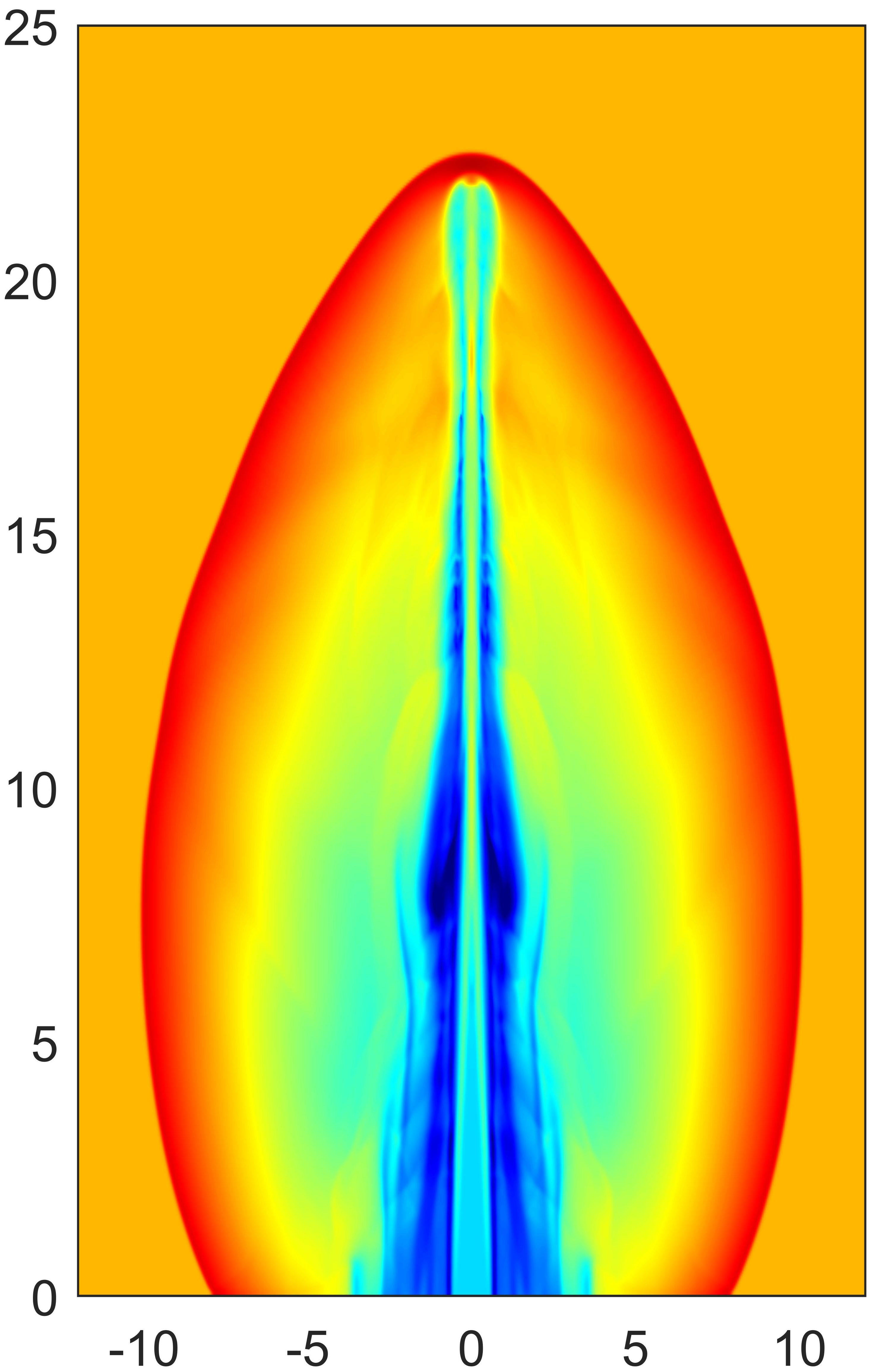}
		\end{center}
	\end{subfigure}
	\caption{\small Example 6: The pseudocolor plots 
		of $\log(\rho)$ for the cold relativistic jet at 
		$t=10,$ $20$, and $30$ (from left to right), respectively.}
	\label{fig:JetC}
\end{figure}

We first simulate a cold (highly supersonic) jet model from \cite{WuTang2017ApJS}. 
Initially, 
a RHD jet (density $\rho_b=0.1$; speed $v_b=0.99$; classical Mach number $M_b=50$) 
is injected along the $y$-direction into the domain $[-12,12]\times [0,25]$, which is initially filled with a static uniform medium of density $\rho_a =1$. 
The fixed inflow condition is enforced at the jet nozzle  $\{(x,y): |x| \le 0.5,y=0\}$  
on the bottom boundary, while the other boundary conditions are outflow. 
For this problem, the jet pressure matches that of the ambient medium, the corresponding initial 
Lorentz factor $W \approx 7.09$, and the relativistic Mach number is $M_r:=M_b W/W_s \approx 354.37$, where $W_s$ denotes the Lorentz factor of the acoustic wave speed. 
The exceedingly high Mach number and jet speed cause the simulation of this problem difficult, so that the BP or constraint-preserving technique is required for high-order schemes to keep the numerical solutions in the invariant region. 
We set the computational domain as $[0,12]\times [0,25]$ with the reflective boundary
condition at $x=0$, and uniformly divided the domain into $240 \times 500$ cells. 
The results, simulated by the fourth-order IRP DG method,  
are presented in Figure \ref{fig:JetC}. 
 Those plots clearly shows the jet evolution, and the flow structures at the final time $t=30$ are in good agreement with those computed in \cite{WuTang2017ApJS}.

\begin{figure}[htbp]
	\centering
	\begin{subfigure}[b]{0.32\textwidth}
		\begin{center}
			\includegraphics[width=1.0\linewidth]{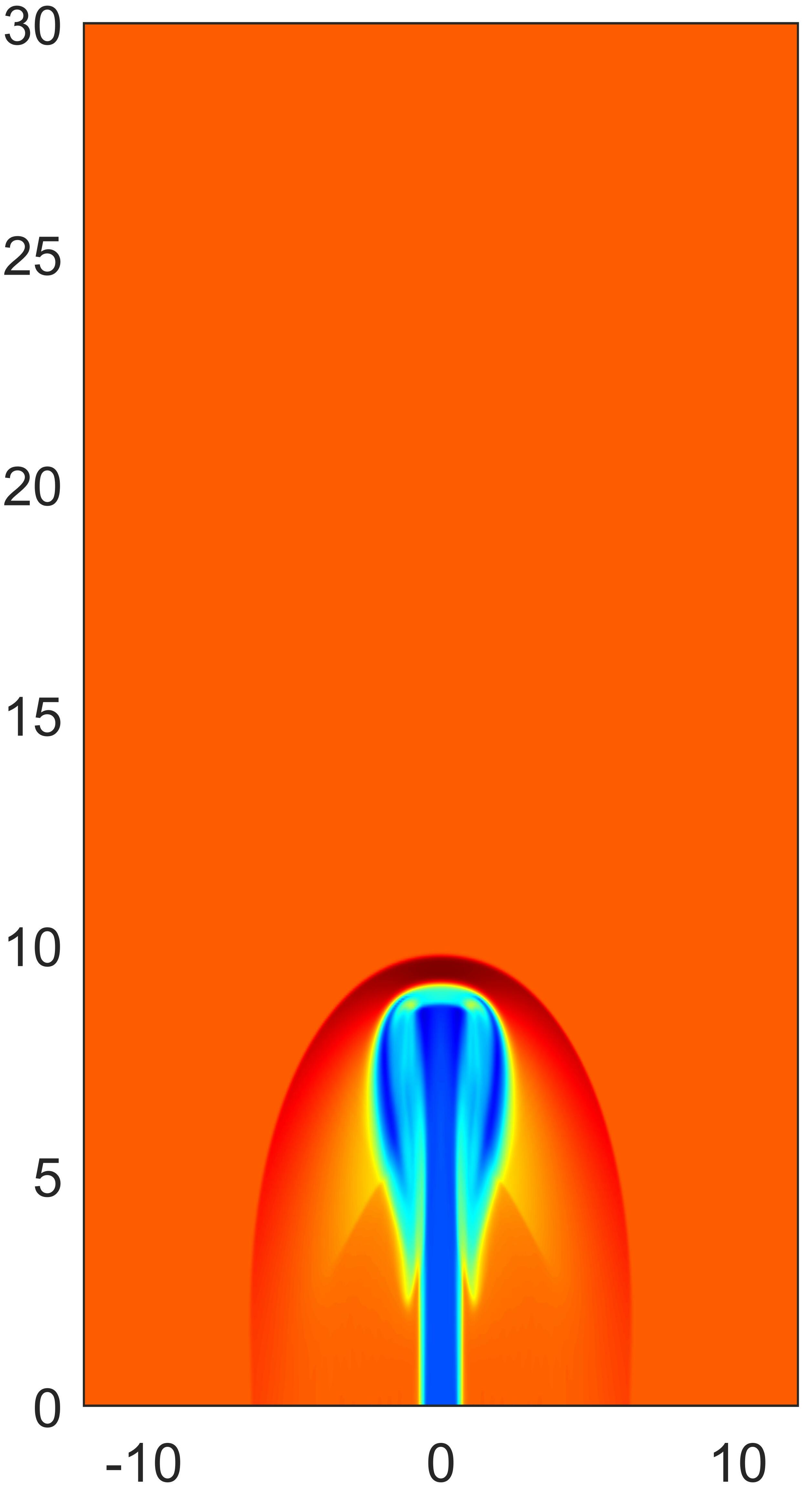}
		\end{center}
	\end{subfigure}
	\begin{subfigure}[b]{0.32\textwidth}
		\begin{center}
			\includegraphics[width=1.0\linewidth]{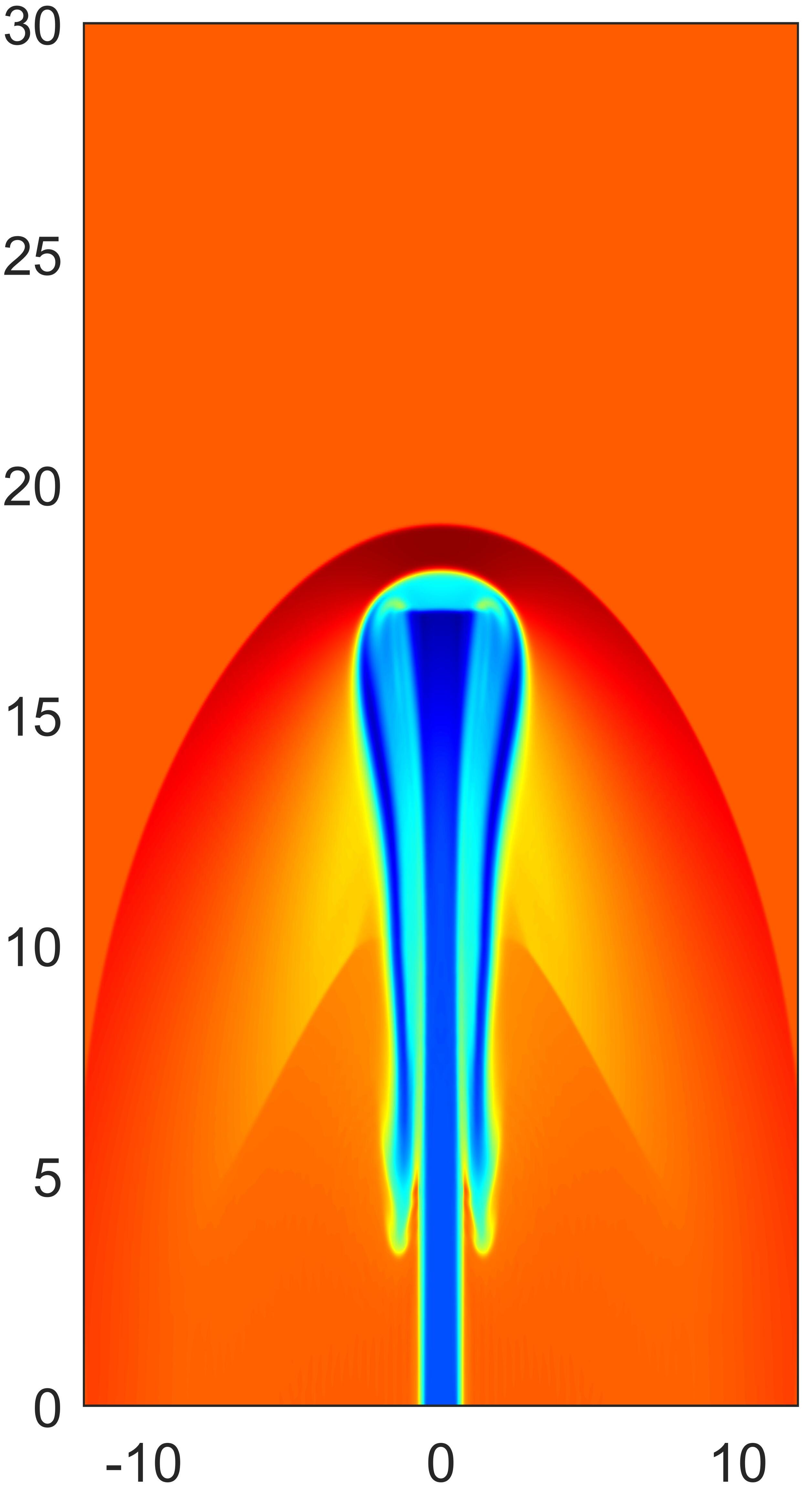}
		\end{center}
	\end{subfigure}
	\begin{subfigure}[b]{0.32\textwidth}
		\begin{center}
			\includegraphics[width=1.0\linewidth]{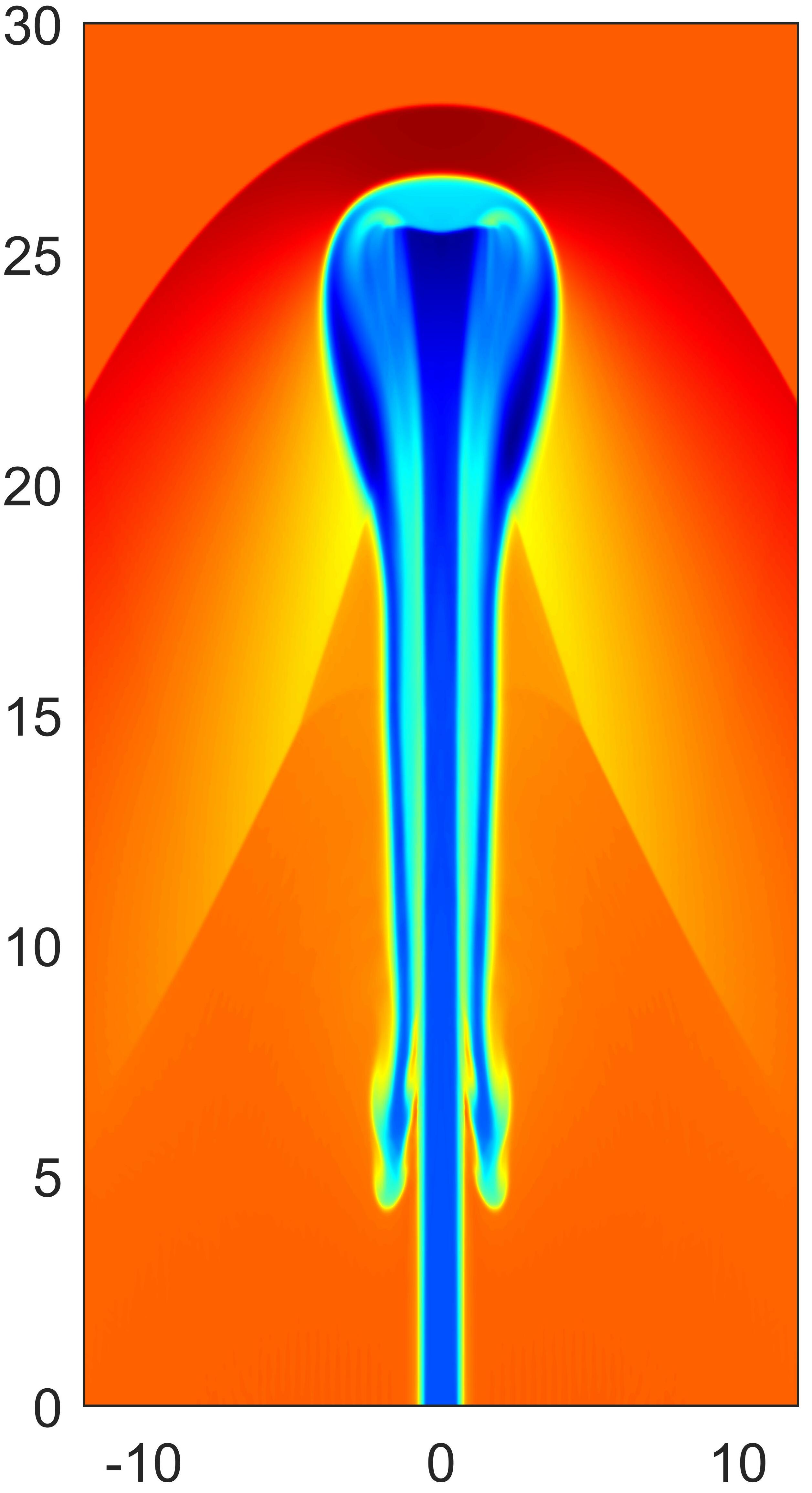}
		\end{center}
	\end{subfigure}
	\caption{\small Example 6: The pseudocolor plots 
		of $\log(\rho)$ for the hot relativistic jet at 
		$t=11,$ $22$, and $33$ (from left to right), respectively.}
	\label{fig:JetA}
\end{figure}

We then simulate a pressure-matched hot jet model. Different from \cite{WuTang2017ApJS}, the ideal EOS \eqref{eq:iEOS} with $\Gamma = \frac43$ is used here. The setups
are the same as the above clod jet model, except for that the
density of the inlet jet becomes $\rho_b =0.01$, and
the classical Mach number is set as $M_b = 1.72$ and near the minimum Mach number $M^{\min} = v_b/\sqrt{\Gamma -1}$ (such that the relativistic effects from large beam internal energies and large fluid velocity are comparable). 
The computational domain is
$[0, 12]\times[0, 30]$ and divided into $240 \times 600$ uniform cells. 
Figure \ref{fig:JetA} gives the numerical result computed the fourth-order IRP DG method. 
One can see that the 
flow structures are clearly resolved by our method, and the 
patterns are  different from those in the cold jet model.

\begin{figure}[htbp]
	\begin{subfigure}[b]{0.49\textwidth}
		\begin{center}
			\includegraphics[width=0.99\linewidth]{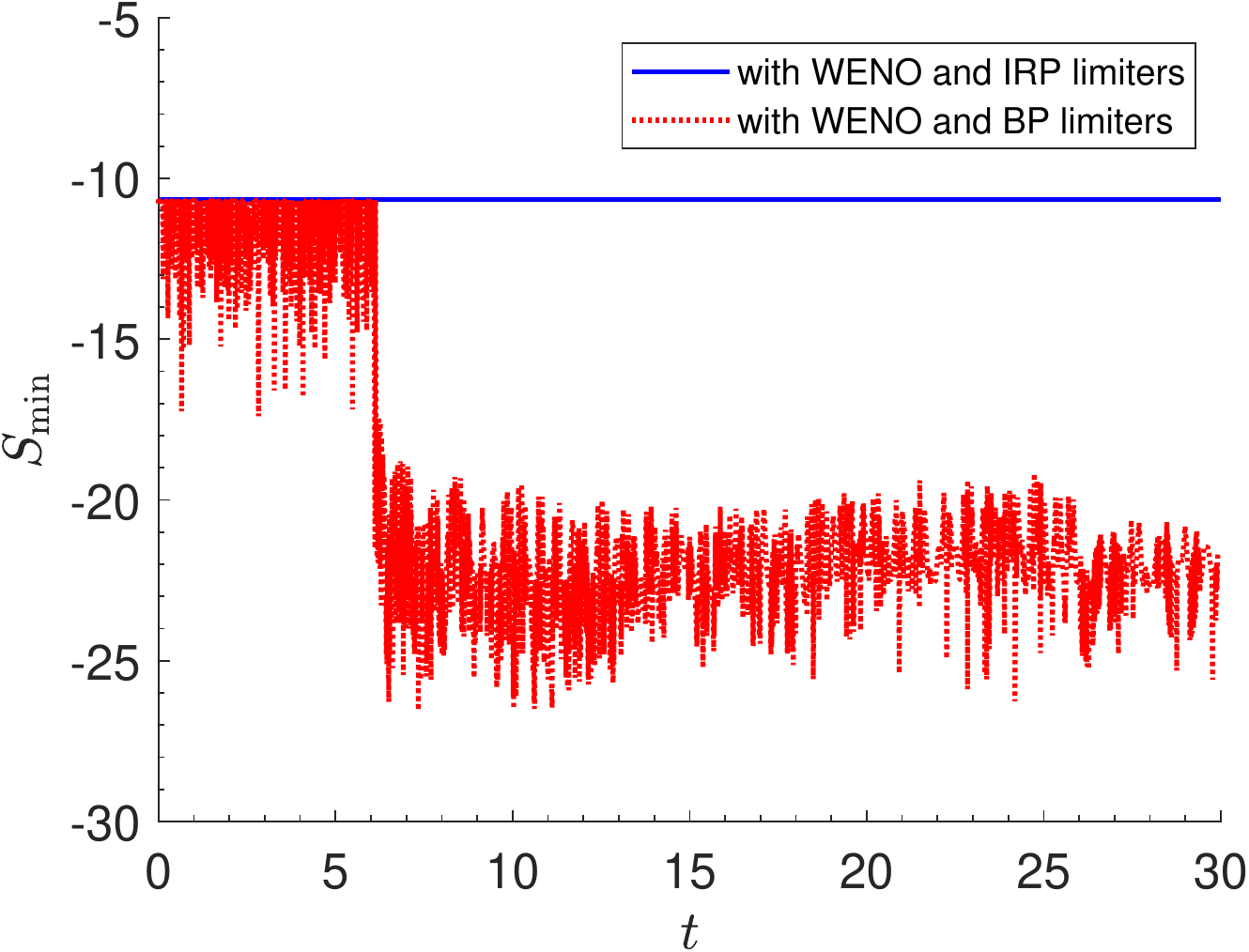}
		\end{center}
		\caption{for the cold jet}
	\end{subfigure}
	\begin{subfigure}[b]{0.49\textwidth}
		\begin{center}
			\includegraphics[width=0.99\linewidth]{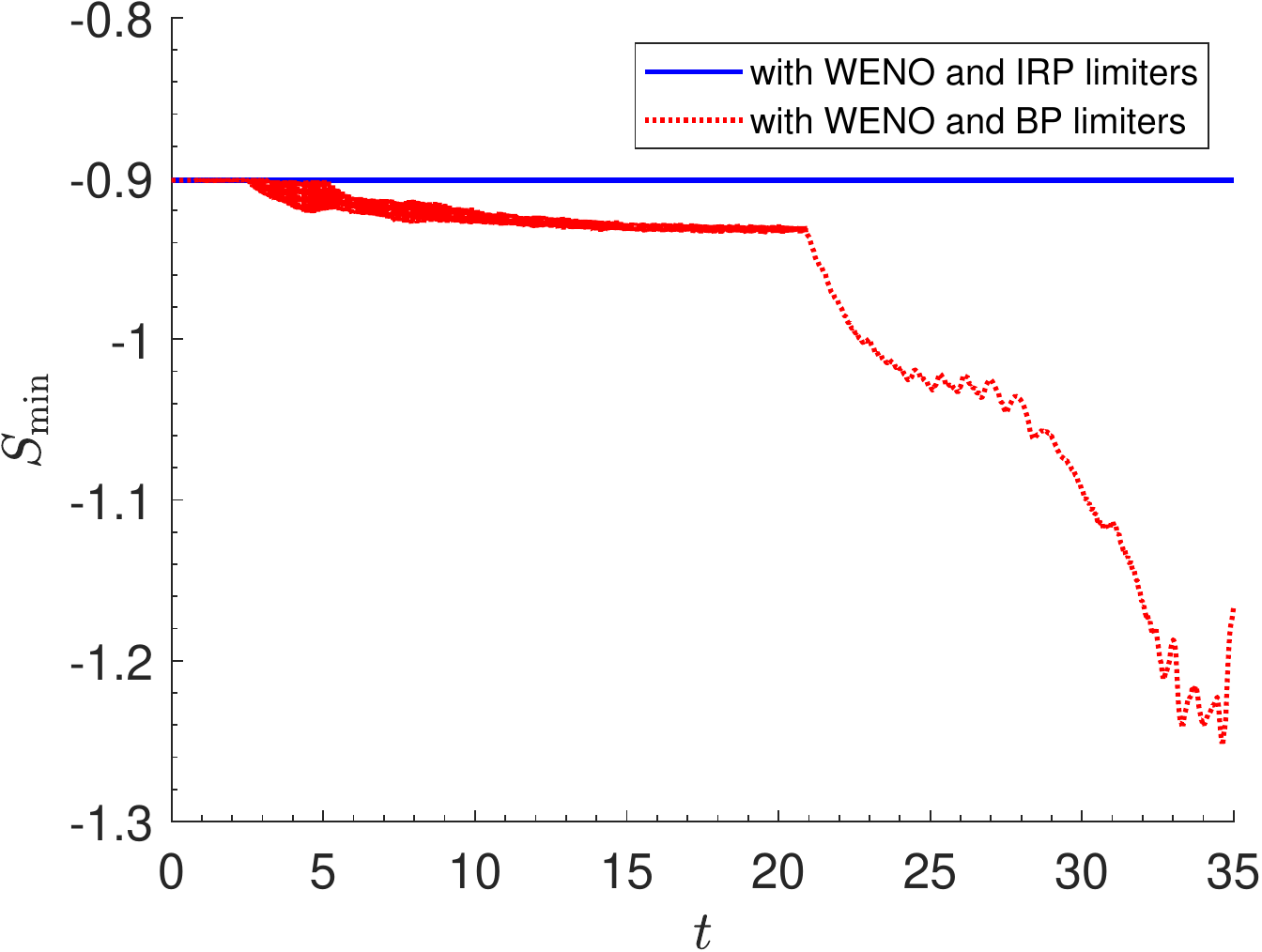}
		\end{center}
		\caption{for the hot jet}
	\end{subfigure}
	\caption{\small Example 6: Evolution of $S_{\min}(t) := \min_{K \in {\mathcal T}_h} \min_{ {\bm x} \in {\mathbb X}_K } S({\bf U}_h({\bm x},t))$ for the DG solutions with the IRP limiter or the BP limiter. The WENO limiter is applied right before the IRP or BP limiting procedure.}  
	\label{fig:Jet_MinS}
\end{figure}

As we have seen, the IRP DG method is very robust in such demanding jet simulations. 
To further verify the minimum entropy principle, we plot in 
Figure \ref{fig:Jet_MinS} the time evolution of the minimum
specific entropy values. It is observed that 
the minimum does not change with time 
for the DG scheme with the WENO and IRP limiters, which means that the 
minimum entropy principle is preserved, while using the WENO and BP limiters is unable to maintain the principle.

\section{Conclusions}\label{sec:con}
In this work, we showed that the Tadmor's minimum entropy principle \eqref{eq:minEntropy} holds for the RHD equations \eqref{eq:RHD} with the ideal EOS \eqref{eq:iEOS}, 
and then developed high-order accurate IRP DG
and finite volume schemes for RHD, which provably preserve a discrete minimum entropy principle as well as the intrinsic physical constraints \eqref{eq:PHYSconstraints}. 
It was the first time that such a minimum entropy principle was explored for  
RHD at the continuous and discrete levels. 
Due to the relativistic effects, 
the specific entropy is a highly nonlinear function of the conservative variables and cannot be explicitly expressed. This led to some essential difficulties in this work, 
which were not encountered in 
the non-relativistic case. 
In order to address the difficulties, we first proposed a novel equivalent form of the invariant region.  
As a notable feature,   
all the constraints in this novel form became explicit and linear with respect to the conservative variables. 
This provided a highly effective approach to  
theoretically analyze the IRP property of numerical RHD schemes. 
We showed the convexity of the invariant region and established the 
generalized Lax--Friedrichs splitting properties via technical estimates. 
We rigorously proved that the first-order Lax--Friedrichs type scheme for the RHD equations satisfies a local minimum entropy principle and is IRP under a CFL condition. 
We then developed and analyzed   
provably IRP high-order accurate DG and finite volume methods for the RHD. 
Numerical examples confirmed that 
enforcing the minimum entropy principle could help to damp some undesirable numerical oscillations 
and demonstrated the robustness 
of the proposed high-order IRP DG schemes.
The verified minimum entropy principle is an important property that 
can be incorporated into improving or designing other RHD schemes. 
Besides, the proposed novel analysis techniques 
can be useful for investigating or seeking other IRP schemes for the RHD or other physical systems. 


\bibliographystyle{siamplain}
\bibliography{references}

\end{document}